\documentclass[11pt]{amsart}

\title{Finite element spaces of double forms}
\author{Yakov Berchenko-Kogan and Evan S. Gawlik}

\usepackage{amsthm,amssymb,hyperref,graphicx,stmaryrd, mathtools,tikz}
\usepackage{xcolor}
\usepackage{multirow}
\usetikzlibrary{decorations.pathreplacing,
                calligraphy,
                matrix}
\usepackage{tikz-3dplot}

\colorlet{darkgreen}{green!50!black}

\usepackage[style=numeric-comp,giveninits=true,url=false,doi=false,isbn=false,maxbibnames=100,backend=bibtex]{biblatex}
\bibliography{references.bib}

\newtheorem{theorem}{Theorem}[section]
\newtheorem{proposition}[theorem]{Proposition}
\newtheorem{corollary}[theorem]{Corollary}
\newtheorem{lemma}[theorem]{Lemma}
\theoremstyle{definition}
\newtheorem{definition}[theorem]{Definition}
\newtheorem{notation}[theorem]{Notation}
\newtheorem{example}[theorem]{Example}
\newtheorem{remark}[theorem]{Remark}

\newcommand\contract{\mathbin\lrcorner}
\newcommand\abs[1]{\left\lvert{#1}\right\rvert}
\renewcommand\phi\varphi
\newcommand\R{\mathbb R}
\newcommand\cP{\mathcal P}
\newcommand\cH{\mathcal H}
\newcommand\oP{\mathring{\cP}}

\newcommand\oH{\mathring{\cH}}

\newcommand\cT{\mathcal T}

\newcommand\rs{\mathring s}

\makeatletter
\newcommand{\ostarsymb}{{\mathpalette\make@circled\star}}
\newcommand{\make@circled}[2]{%
  \ooalign{$\m@th#1\smallbigcirc{#1}$\cr\hidewidth$\m@th#1#2$\hidewidth\cr}%
}
\newcommand{\smallbigcirc}[1]{%
  \vcenter{\hbox{\scalebox{0.77778}{$\m@th#1\bigcirc$}}}%
}
\makeatother

\DeclareMathOperator\ostar\ostarsymb
\DeclareMathOperator\starop\star
\DeclareMathOperator\id{id}
\DeclareMathOperator\Lie{\mathfrak L}
\DeclareMathOperator\tr{tr}
\DeclareMathOperator\Tr{Tr}

\newcommand\pp[2][]{\frac{\partial{#1}}{\partial{#2}}}
\newcommand\D{\mathop{}\!d}
\newcommand\dx{\D x}
\newcommand\dl{\D\lambda}
\newcommand\du{\D u}

\newcommand\dlx[2]{\dl_{#1}\otimes\dl_{#2}}
\newcommand\dlsym[2]{\dl_{#1}\odot\dl_{#2}}

\begin{document}

\begin{abstract}
  The tensor product of two differential forms of degree $p$ and $q$ is a multilinear form that is alternating in its first $p$ arguments and alternating in its last $q$ arguments.  These forms, which are known as double forms or $(p,q)$-forms, play a central role in certain differential complexes that arise when studying partial differential equations.  We construct piecewise polynomial finite element spaces for all of the natural subspaces of the space of $(p,q)$-forms, excluding one subspace which fails to admit a piecewise constant discretization.  As special cases, our construction recovers known finite element spaces for symmetric matrices with tangential-tangential continuity (the Regge finite elements), symmetric matrices with normal-normal continuity, and trace-free matrices with normal-tangential continuity.  It also gives rise to new spaces, like a finite element space for tensors possessing the symmetries of the Riemann curvature tensor.
\end{abstract}

\maketitle

\section{Introduction} \label{sec:intro}

Over the last several decades, exterior calculus has played an important role in the development of numerical methods for partial differential equations (PDEs).  Notably, Arnold, Falk, and Winther~\cite{arnold2006finite,arnold2010finite} showed that finite element methods for many PDEs can be best understood by viewing the unknowns as differential forms and seeking approximate solutions in finite-dimensional spaces of differential forms.  These finite-dimensional spaces, or finite element spaces, consist of differential forms which are piecewise polynomial with respect to a simplicial triangulation of the domain on which the PDE is posed.  When chosen carefully, such spaces give rise to stable mixed discretizations of PDEs involving the Hodge-Laplace operator.  Arnold, Falk, and Winther's work led to a complete classification of such spaces, generalizing and unifying finite element spaces that are attributed to Whitney~\cite{whitney1957geometric}, Raviart and Thomas~\cite{raviart2006mixed}, N\'ed\'elec~\cite{nedelec1980mixed,nedelec1986new}, Brezzi, Douglas, and Marini~\cite{brezzi1985two}, and others.  Their work also highlighted the importance of differential complexes---particularly the de Rham complex---in the design and analysis of finite element methods for PDEs.

In this paper, we construct finite element spaces for \emph{double forms}: tensor products of differential forms.  Unlike ordinary differential $k$-forms, which are multilinear and alternating in all $k$ of their arguments, the tensor product of a $p$-form and a $q$-form is a multilinear form that is alternating in its first $p$ arguments and alternating in its last $q$ arguments.  These forms, which are known as double forms or $(p,q)$-forms, have a long history in differential geometry~\cite{calabi1961compact,gray1970some,kulkarni1972bianchi,de2012differentiable,labbi2005double,kupferman2024double,troyanov2024doubleforms} and have recently drawn the attention of numerical analysts~\cite{arnold2021complexes,bonizzoni2023discrete,hu2025finite} due to their role in certain differential complexes that arise when studying PDEs. 

To be specific, we consider an $N$-dimensional simplicial triangulation $\mathcal{T}$ and focus on constructing piecewise polynomial $(p,q)$-forms that have single-valued trace on element interfaces.  Here, taking the trace of a $(p,q)$-form $\phi$ on a simplex $F$ means that we not only evaluate $\phi$ at points on $F$, but we also feed into $\phi$ only vectors which are tangent to $F$.  It turns out that the full space of $(p,q)$-forms does not admit such a discretization if one uses piecewise constants (unless $p+q \ge N$); only certain subspaces do.   One of the main things we do in this paper is determine which subspaces admit such a discretization and, for those that do, we construct one by providing degrees of freedom for the finite element space.  We do this in depth for the piecewise constant finite element spaces and leave certain details about the higher-order spaces (like explicit degree of freedom counts) for future work.  Our main results are stated in Theorems~\ref{thm:geodecomp} and~\ref{thm:dualdecomp}, which provide geometric decompositions of the piecewise polynomial finite element spaces of double forms and their duals.  The proofs of these two theorems rely critically on the existence of a \emph{simplicial extension operator}~\cite[Definition 3.8]{berchenko2026extension} for double forms.  The existence of such an extension operator is proved in Theorem~\ref{thm:extension} for certain subspaces of double forms and ruled out for other subspaces in Proposition~\ref{prop:noextension}.

Our construction recovers several known finite element spaces as special cases.  The Regge finite element space is one example~\cite{christiansen2004characterization,christiansen2011linearization,li2018regge}.  In the lowest order setting, the members of this space are often described as piecewise constant symmetric matrices possessing tangential-tangential continuity, and the space has one degree of freedom per edge in the triangulation.  In our language, the Regge finite elements are symmetric $(1,1)$-forms with single-valued trace on lower-dimensional simplices.  The word ``symmetric'' is important here; our construction recovers the piecewise constant Regge finite element space when considering symmetric $(1,1)$-forms but fails to provide a piecewise constant finite element space when considering antisymmetric $(1,1)$-forms (except in dimension $N=2$).  This is consistent with the fact that antisymmetric $(1,1)$-forms are simply $2$-forms, and $2$-forms do not admit a piecewise constant discretization with tangential continuity in dimension $N>2$. (More precisely, they do not admit a tangentially continuous piecewise constant discretization that is geometrically decomposable in the sense of~\cite{arnold2009geometric}.)

In the same way that the space of $(1,1)$-forms decomposes naturally into two subspaces---symmetric $(1,1)$-forms and antisymmetric $(1,1)$-forms---the space of $(p,q)$-forms admits a natural decomposition as well.  This decomposition, which has its origins in representation theory~\cite[Exercises 6.13* and 15.32*]{fulton2004representation} and involves at most $\min\{p,q\}+1$ summands, can be characterized in several different ways~\cite{gawlik2025title}.  For our purposes it is convenient to regard this decomposition as an eigendecomposition of a certain self-adjoint operator on $(p,q)$-forms.  Relative to this decomposition, we show that all of the summands admit a discretization with piecewise polynomials of degree $\ge 1$, and all but one of the summands admits a piecewise constant discretization.  The exceptional summand consists of those $(p,q)$-forms that alternate in all $p+q$ arguments, i.e.~the $(p,q)$-forms that are actually $(p+q)$-forms.

When one considers $(2,1)$-forms, there are two summands in the aforementioned decomposition.  For one of those summands, our construction yields a finite element space that in 3D is isomorphic to the space of piecewise polynomial, trace-free matrices with normal-tangential continuity introduced by Gopalakrishnan, Lederer, and Sch\"oberl~\cite{gopalakrishnan2020mass}.  This space has two degrees of freedom per triangle in the lowest-order setting.

For $(2,2)$-forms in dimension $N=3$, our construction yields (for one of the summands in the decomposition) a finite element space that is isomorphic to the space of piecewise polynomial, symmetric matrices with normal-normal continuity introduced by Pechstein and Sch\"oberl~\cite{pechstein2011tangential}. This space has one degree of freedom per triangle and two degrees of freedom per tetrahedron in the lowest-order setting.

Our construction also gives rise to many new finite element spaces.  Of particular interest are $(2,2)$-forms in dimension $N \ge 3$ that satisfy the algebraic Bianchi identity $\phi(X,Y;Z,W) + \phi(Y,Z;X,W) + \phi(Z,X;Y,W) = 0$.  These so-called \emph{algebraic curvature tensors} possess precisely the same symmetries as the Riemann curvature tensor from differential geometry (including the symmetry $\phi(X,Y;Z,W) = \phi(Z,W;X,Y)$, which follows from the algebraic Bianchi identity and the fact that $\phi(X,Y;Z,W)$ alternates in $X$ and $Y$ and alternates in $Z$ and $W$).  Our construction yields a piecewise polynomial finite element space for such algebraic curvature tensors.  In dimension $N = 3$, the space is the same as the one mentioned above that can be identified with Pechstein and Sch\"oberl's space.  In dimension $N > 3$, the lowest-order version of this space has one degree of freedom per triangle and two degrees of freedom per tetrahedron, just like in dimension $N=3$.  Let us remark that in certain contexts, it may be preferable to work with a dual version of these double forms, namely $(N-2,N-2)$-forms whose double Hodge dual satisfies the algebraic Bianchi identity.  Our construction yields a finite element space for these double forms as well.  As discussed in~\cite{gopalakrishnan2023analysis}, such a finite element space may be useful for computations that involve the distributional Riemann curvature tensor.

\subsection{Organization}
We begin in Section~\ref{sec:doublemulti} by studying multilinear functionals that alternate in their first $p$ arguments and alternate in their last $q$ arguments.  We show that these  functionals, or $(p,q)$-covectors, admit a natural decomposition.  We then bring spatial dependence into the picture in Section~\ref{sec:doubleforms} and study $(p,q)$-forms on a manifold.  The tools developed in Sections~\ref{sec:doublemulti} and~\ref{sec:doubleforms} will be used to prove a key result in Section~\ref{sec:extension}: Given any $(p,q)$-form $\phi$ on the standard simplex $T^n = \{(\lambda_0,\dots,\lambda_n) \mid \lambda_i \ge 0, \, \sum_i \lambda_i = 1\}$ that has polynomial coefficients of degree at most $r \ge 0$ and vanishing trace on $\partial T^n$, one can nearly always extend $\phi$ to a $(p,q)$-form on $\R^{n+1}$ that has polynomial coefficients of degree $r$ and vanishing trace on the coordinate hyperplanes.  We show that this extension preserves the aforementioned decomposition, and that such an extension fails to exist for precisely one summand if $r=0$: the aforementioned exceptional summand of fully alternating tensors.  With the help of a recent paper by the first author~\cite{berchenko2026extension}, we use this result to prove the existence of piecewise polynomial finite element spaces for $(p,q)$-forms in Section~\ref{sec:finite}.  These finite element spaces exist for all polynomial degrees $r \ge 0$ and all subspaces in the decomposition except for the one that fails to admit extensions when $r=0$.  Using~\cite{berchenko2026extension}, we show that these finite element spaces admit a geometric decomposition in the sense of~\cite{arnold2009geometric}.  This decomposition gives rise to a local basis for the finite element spaces that resembles the Bernstein basis for scalar-valued Lagrange finite elements; see Table~\ref{tab:bases} for examples of these bases. We conclude Section~\ref{sec:finite} by providing formulas for the dimensions of the piecewise constant finite element spaces and for the number of degrees of freedom that one must assign to each simplex to ensure unisolvence.  We give examples that show how our construction recovers some known finite element spaces and discovers some new ones.  Finally, in Appendix~\ref{sec:rep}, we discuss how the decomposition of double forms and some of our results arise from representation theory.

For the reader's convenience, we provide a summary of the notation used in this paper in Table~\ref{tab:notation}.
	
	\begin{table}
		\centering\hspace{-0.2in}
		\begin{tabular}{cp{3.4in}}
			\renewcommand{\arraystretch}{1.2}
			$\Lambda^{p,q}$ & The space of $(p,q)$-covectors. \\
			$\Lambda^{p,q}(M)$ & The space of $(p,q)$-forms on a manifold $M$.  Sometimes abbreviated $\Lambda^{p,q}$. \\
			$s$ & The Bianchi sum $s : \Lambda^{p,q} \to \Lambda^{p+1,q-1}$. \\
			$s^*$ & The adjoint of $s$. \\
			$\Lambda^{p,q}_m$ & The eigenspace of $s^*s : \Lambda^{p,q} \to \Lambda^{p,q}$ corresponding to the eigenvalue $m(m+p-q+1)$. \\
			$\prescript{}{m^*}\Lambda^{p,q}$ & The eigenspace of $ss^*\colon\Lambda^{p,q}\to\Lambda^{p,q}$ corresponding to eigenvalue $m^*(m^*+q-p+1)$.\\
			$\tau$ & The transposition operator $\tau(\alpha \otimes \beta) = \beta \otimes \alpha$.\\
			$\owedge$ & The double wedge product $(\alpha \otimes \beta) \owedge (\gamma \otimes \delta) = (\alpha \wedge \gamma) \otimes (\beta \wedge \delta)$. \\
			$\ostar$ & The double Hodge star operator $\ostar(\alpha \otimes \beta) = \star \alpha \otimes \star \beta$.\\
			$X_{\rm id}$ & The tautological vector field $X_{\rm id} = \sum_i x^i \frac{\partial}{\partial x^i}$.\\
			$\kappa$ & The Koszul operator $\kappa \alpha = X_{\rm id} \contract \alpha$.\\
			$\kappa_L$, $\kappa_R$ & The maps $\kappa_L(\alpha \otimes \beta) = (\kappa \alpha) \otimes \beta$ and $\kappa_R(\alpha \otimes \beta) = \alpha \otimes (\kappa \beta)$.\\
			$d$ & The exterior derivative.\\
			$dx^I$ & $dx^{i_1} \wedge \dots \wedge dx^{i_k}$, where $I=(i_1,\dots,i_k)$.\\
			$d_L$, $d_R$ & The maps $d_L (f dx^I \otimes dx^J) = (df \wedge dx^I) \otimes dx^J$ and $d_R (f dx^I \otimes dx^J) = dx^I \otimes (df \wedge dx^J)$. \\
			$T^n$ & The standard $n$-simplex $\{(\lambda_0,\dots,\lambda_n) \mid \lambda_i \ge 0, \, \sum_i \lambda_i = 1 \}$. \\
			$S^n$ & The unit $n$-sphere $\{(u_0,\dots,u_n) \mid \sum_i u_i^2 = 1\}$.\\
			$Q^{n+1}$ & The $(n+1)$-simplex $\{(\lambda_0,\dots,\lambda_n) \mid \lambda_i \ge 0, \, \sum_i \lambda_i \le 1 \}$.\\
			$\mathcal{H}_r$ & The space of homogeneous polynomials of degree $r$.\\
			$\mathcal{H}_r \Lambda^k$, $\mathcal{H}_r \Lambda^{p,q}$, $\mathcal{H}_r \Lambda^{p,q}_m$ & The corresponding spaces of forms or double forms with coefficients in $\mathcal{H}_r$.\\
			$\mathring{\cH}_r \Lambda^k$, $\mathring{\cH}_r \Lambda^{p,q}$, $\mathring{\cH}_r \Lambda^{p,q}_m$ & The corresponding subspaces with vanishing trace on the coordinate hyperplanes.\\
			$\mathcal{H}_r^- \Lambda^{p,q}$, $\mathcal{H}_r^- \Lambda^{p,q}_m$ & The spaces $\kappa_L \kappa_R \mathcal{H}_{r-2} \Lambda^{p,q}$ and $\kappa_L \kappa_R \mathcal{H}_{r-2} \Lambda^{p,q}_m$. \\
			$\cP_r$ & The space of polynomials of degree at most $r$.\\
			$\cP_r\Lambda^k$, $\cP_r\Lambda^{p,q}$, $\cP_r\Lambda^{p,q}_m$ & The corresponding spaces of forms or double forms with coefficients in $\mathcal{P}_r$.\\
			$\oP_r\Lambda^k$, $\oP_r\Lambda^{p,q}$, $\oP_r\Lambda^{p,q}_m$ & The corresponding subspaces with vanishing trace on $\partial T^n$.\\
			$\Phi$ & The map $(u_0,\dots,u_n) \mapsto (u_0^2,\dots,u_n^2)$.\\
			$u_N$ & The function $u_0 u_1 \cdots u_n$. \\
			$\star_{S^n}$ & The operator $\star_{S^n} \alpha = \star_{\mathbb{R}^{n+1}}(\nu \wedge \alpha) = (-1)^k\kappa(\star_{\R^{n+1}}\alpha)$, where $\nu = X_{\rm id}^\flat$.\\
			$\ostar_{S^n}$ & The operator $\ostar_{S^n}(\alpha \otimes \beta) = \star_{S^n} \alpha \otimes \star_{S^n} \beta$. \\
			$\mathcal{T}$ & A simplicial triangulation. \\
			$\mathfrak{T}$ &  The partially ordered set consisting of the vertices, edges, faces, etc.~of $\mathcal{T}$, ordered by inclusion.\vspace{0.1in}
		\end{tabular}
		\caption{Summary of notation used in the paper.}
		\label{tab:notation}
	\end{table}

\subsection{Related work}  \label{sec:introrelated}
Recently, Hu and Lin~\cite{hu2025finite} have constructed finite element spaces for double forms using techniques that differ from ours.  As we explain in Section~\ref{sec:related}, some of the spaces we obtain coincide with theirs and some do not.  Briefly, the finite element spaces we treat in this paper are what Hu and Lin call $\iota^*\iota^*$-conforming spaces.  These spaces consist of double forms with single-valued trace on element interfaces.  Hu and Lin arrive at such spaces by tensoring Whitney forms with alternating forms and moving degrees of freedom around to weaken the continuity of the resulting spaces. Then, certain degrees of freedom are removed to obtain ``symmetry-reduced spaces''.  These symmetry-reduced, $\iota^*\iota^*$-conforming spaces are the ones which most closely match ours; see Section~\ref{sec:related} for details.  Hu and Lin also treat spaces with other continuity properties, while we do not.

Although some of the finite element spaces we obtain overlap with Hu and Lin's, the approach we take to arrive at them is rather different and leads us to new results.  Notably, we obtain a geometric decomposition of the finite element spaces in the sense of~\cite{arnold2009geometric}, complete with local bases and explicit extension operators.  Our approach also highlights the usefulness of the representation-theoretic decomposition of double forms in the construction of finite element spaces.  We suspect that such decompositions will be equally useful for constructing finite element spaces for tensors with other symmetries.  In fact, this has already happened in another context: When developing finite element exterior calculus, Arnold, Falk, and Winther~\cite[p. 37]{arnold2006finite} used a decomposition of polynomial differential forms into irreducible linear-invariant subspaces to identify affine-invariant subspaces thereof.

There are also commonalities between the tools that we use in this paper and tools that have been used to study differential complexes involving double forms~\cite{kupferman2022saint,kupferman2023elliptic,kupferman2024double,arnold2021complexes,cap2024bgg,hu2025finite}.  For example, the Bernstein-Gelfand-Gelfand (BGG) framework developed in~\cite{cap2024bgg,arnold2021complexes,hu2025finite} uses various maps between spaces of double forms to study differential complexes like the Hessian complex, elasticity complex, and div-div complex.  These maps, which include the exterior covariant derivative, Bianchi sum, and Koszul operator, possess several properties and interrelationships that end up being essential to both our work and theirs.

\section{Double multicovectors} \label{sec:doublemulti}
We begin by working in the linear algebra context, studying multilinear functionals that are antisymmetric in the first $p$ indices and antisymmetric in the last $q$ indices. We call these functionals \emph{double multicovectors}. As we will see, the space $\Lambda^{p,q}$ of double multicovectors has a much richer structure than the space $\Lambda^k$ of regular multicovectors. Ultimately, this rich structure arises because, unlike $\Lambda^k$, the space of double multicovectors $\Lambda^{p,q}$ is \emph{not} an irreducible representation with respect to the natural action of the general linear group. Consequently, $\Lambda^{p,q}$ has a natural decomposition into subspaces, and there are nontrivial natural maps between the $\Lambda^{p,q}$. We discuss the connection to representation theory in Appendix~\ref{sec:rep}. For now, we give a more elementary discussion of this structure.

We begin by defining double multicovectors. Then, in Section~\ref{sec:ssdual}, we define the \emph{Bianchi sum} $s$ and its adjoint $s^*$. We show that the eigendecomposition of $s^*s$ gives a natural decomposition $\Lambda^{p,q}=\bigoplus_m\Lambda^{p,q}_m$ in Section~\ref{sec:decomposition}. In Section~\ref{sec:properties}, we delve into the relationships between the summands of this decomposition and the $s$ and $s^*$ operators, including determining which summands are nonzero and proving the isomorphism $s^m\colon\Lambda^{p,q}_m\to\Lambda^{p+m,q-m}_0$. We discuss examples in Section~\ref{sec:decompeg}, including how this decomposition specializes to matrix decompositions in three dimensions. In Section~\ref{sec:additionaloperations}, we discuss the transposition operator $\tau\colon\Lambda^{p,q}\to\Lambda^{q,p}$ and the double Hodge star operator $\ostar\colon\Lambda^{p,q}\to\Lambda^{n-p,n-q}$.  Finally, in Section~\ref{sec:Trdecomp}, we point out a relationship between the decomposition $\Lambda^{p,q}=\bigoplus_m\Lambda^{p,q}_m$ and another decomposition that is metric-dependent.

In what follows, we fix a vector space $V$. We will assume that $V$ is finite-dimensional, though some definitions and results work equally well in the infinite-dimensional context.

\begin{definition}
  A \emph{$k$-covector} or \emph{multicovector} is an antisymmetric $k$-linear functional on $V$. The space of $k$-covectors is denoted $\bigwedge^kV^*$, which we will abbreviate as $\Lambda^k$ when there is no risk of confusion.
\end{definition}
\begin{definition}
  Let $\Lambda^{p,q}:=\Lambda^p\otimes\Lambda^q$. When we wish to emphasize the space $V$, we will use the notation $\bigwedge^{p,q}V^*$. A \emph{$(p,q)$-covector} or \emph{double multicovector} is an element of $\Lambda^{p,q}$. Letting $k=p+q$, double multicovectors are $k$-linear functionals on $V$ which are antisymmetric in the first $p$ indices and antisymmetric in the last $q$ indices. 
\end{definition}

\begin{notation} \label{not:basis}
  If $e_1,\dotsc,e_n$ is a basis for $V$, let $e^1,\dotsc,e^n$ be the corresponding dual basis of $V^*$. For a multi-index $I=(i_1,\dotsc,i_k)$, let $e^I:=e^{i_1}\wedge\dotsb\wedge e^{i_k}\in\Lambda^k$. Let $e^{I,J}:=e^I\otimes e^J\in\Lambda^{p,q}$, where $p=\abs I$ and $q=\abs J$.
\end{notation}

If we restrict $I$ and $J$ to each be in increasing order, then the $e^{I,J}$ form a basis of $\Lambda^{p,q}$. If $e_1,\dotsc,e_n$ is orthonormal with respect to an inner product on $V$, then this $e^{I,J}$ basis is orthonormal with respect to the induced inner product on $\Lambda^{p,q}$.

\subsection{The $s$ and $s^*$ operators}\label{sec:ssdual}
For any $p$ and $q$, there is a natural map $s\colon\Lambda^{p,q}\to\Lambda^{p+1,q-1}$. Up to a constant, this map is simply antisymmetrization in the first $p+1$ indices. There is likewise a corresponding map $s^*\colon\Lambda^{p,q}\to\Lambda^{p-1,q+1}$. The decomposition alluded to earlier is simply the eigendecomposition of $s^*s$. We now discuss these operators and this decomposition in more detail.

\begin{definition}\label{def:s}
  Let $s\colon\Lambda^{p,q}\to\Lambda^{p+1,q-1}$, sometimes called the \emph{Bianchi sum}, denote the map defined by
  \begin{multline*}
    (s\phi)(X_1,\dots,X_{p+1};Y_1,\dots,Y_{q-1}) \\
    = \sum_{a=1}^{p+1} (-1)^{a-1} \phi(X_1,\dots,\widehat{X_a},\dots,X_{p+1};X_a,Y_1,\dots,Y_{q-1}).
  \end{multline*}
  Equivalently, we can define $s$ on simple tensors by
  \begin{multline*}
    s\bigl((\alpha_1 \wedge \dots \wedge \alpha_p) \otimes (\beta_1 \wedge \dots \wedge \beta_q) \bigr) \\
    = \sum_{a=1}^q (-1)^{a-1} (\beta_a \wedge \alpha_1 \wedge \dots \wedge \alpha_p) \otimes (\beta_1 \wedge \dots \wedge \widehat{\beta_a} \wedge \dots \wedge \beta_q).
  \end{multline*}
  and then extending by linearity.
\end{definition}

See~\cite{kulkarni1972bianchi,gray1970some,calabi1961compact,kupferman2023elliptic,kupferman2024double,arnold2021complexes,hu2025finite}, as well as \cite[Exercise~15.30]{fulton2004representation} for more discussion of this map. We can likewise define the map going the other way.

\begin{definition}
  Let $s^*\colon\Lambda^{p,q}\to\Lambda^{p-1,q+1}$ denote the map defined by
  \begin{multline*}
    (s^*\phi)(X_1,\dots,X_{p-1};Y_1,\dots,Y_{q+1}) \\
    = \sum_{a=1}^{q+1} (-1)^{a-1} \phi(Y_a,X_1,\dots,X_{p-1};Y_1,\dots,\widehat{Y_a},\dots,Y_{q+1}),
  \end{multline*}
  or, equivalently, by
  \begin{multline*}
    s^*\bigl((\alpha_1 \wedge \dots \wedge \alpha_p) \otimes (\beta_1 \wedge \dots \wedge \beta_q) \bigr) \\
    = \sum_{a=1}^p (-1)^{a-1} (\alpha_1 \wedge \dots \wedge \widehat{\alpha_a} \wedge\dots\wedge \alpha_p) \otimes (\alpha_a\wedge\beta_1 \wedge \dots \wedge \beta_q).
  \end{multline*}
\end{definition}

As the notation suggests, $s$ and $s^*$ are adjoints of each other with respect to the natural inner product on double multicovectors induced by an arbitrary inner product on $V$. We will prove this result shortly. First, we give an alternate characterization of the $s$ operator as wedge-contraction with the identity linear transformation.

\begin{remark}
  Propositions~\ref{prop:wedgecontract},~\ref{prop:adjoint} and~\ref{prop:ssstar} also appear in~\cite[p.~49]{kupferman2023elliptic}.
\end{remark}

\begin{proposition} \label{prop:wedgecontract}
  We have
  \begin{align*}
    s(\alpha\otimes\beta)&=\sum_{i=1}^n(e^i\wedge\alpha)\otimes(e_i\contract\beta),\\
    s^*(\alpha\otimes\beta)&=\sum_{i=1}^n(e_i\contract\alpha)\otimes(e^i\wedge\beta),
  \end{align*}
  where $\alpha$ is a $p$-covector and $\beta$ is a $q$-covector.
\end{proposition}

\begin{proof}
  Checking on a basis, we must check that
  \begin{equation*}
    s(e^{I,J})=\sum_{i=1}^n(e^i\wedge e^I)\otimes(e_i\contract e^J).
  \end{equation*}
  By the above definition, the left-hand side is
  \begin{equation*}
    s(e^{I,J})=\sum_{a=1}^q(-1)^{a-1}(e^{j_a}\wedge e^{i_1}\wedge\dots\wedge e^{i_p})\otimes(e^{j_1}\wedge\dots\wedge\widehat{e^{j_a}}\wedge\dots\wedge e^{j_q}),
  \end{equation*}
  where $I=(i_1,\dotsc,i_p)$ and $J=(j_1,\dotsc,j_q)$.

  Moving on to the right-hand side, we have that $e_i\contract e^J=0$ unless $i=j_a$ for some $a$. Thus, we can instead sum over $a$, obtaining
  \begin{equation*}
    \sum_{i=1}^n(e^i\wedge e^I)\otimes(e_i\contract e^J)=\sum_{a=1}^q(e^{j_a}\wedge e^I)\otimes(e_{j_a}\contract e^J).
  \end{equation*}
  The claim follows because $e_{j_a}\contract e^J=(-1)^{a-1}(e^{j_1}\wedge\dots\wedge\widehat{e^{j_a}}\wedge\dots\wedge e^{j_q})$. The result for the $s^*$ operator is analogous.
\end{proof}

Assuming that the $e_1,\dotsc,e_n$ are orthonormal, the operators $e_i\contract$ and $e^i\wedge$ are adjoints of each other. We can use this property along with the previous proposition to show that the operators $s$ and $s^*$ are adjoints of each other.  See~\cite[Lemma A.1]{hu2025finite} for an alternative proof of this fact.

\begin{proposition}\label{prop:adjoint}
  The operators $s$ and $s^*$ are adjoints of each other.
\end{proposition}

\begin{proof}
  It suffices to check on simple tensors that
  \begin{equation*}
    \left\langle s(\alpha\otimes\beta),\gamma\otimes\delta\right\rangle=\left\langle\alpha\otimes\beta,s^*(\gamma\otimes\delta)\right\rangle,
  \end{equation*}
  where $\alpha\in\Lambda^p$, $\beta\in\Lambda^q$, $\gamma\in\Lambda^{p+1}$, and $\delta\in\Lambda^{q-1}$. We compute
  \begin{equation*}
    \begin{split}
      \left\langle s(\alpha\otimes\beta),\gamma\otimes\delta\right\rangle&=\sum_{i=1}^n\left\langle(e^i\wedge\alpha)\otimes(e_i\contract\beta),\gamma\otimes\delta\right\rangle\\
      &=\sum_{i=1}^n\left\langle e^i\wedge\alpha,\gamma\right\rangle\left\langle e_i\contract\beta,\delta\right\rangle\\
      &=\sum_{i=1}^n\left\langle\alpha,e_i\contract\gamma\right\rangle\left\langle\beta,e^i\wedge\delta\right\rangle\\
      &=\sum_{i=1}^n\left\langle\alpha\otimes\beta,(e_i\contract\gamma)\otimes(e^i\wedge\delta)\right\rangle\\
      &=\left\langle\alpha\otimes\beta,s^*(\gamma\otimes\delta)\right\rangle.\qedhere
    \end{split}
  \end{equation*}
\end{proof}

Therefore, $s^*s\colon\Lambda^{p,q}\to\Lambda^{p,q}$ is self-adjoint, so it is diagonalizable. Note that $s$ is nilpotent because $s^{q+1}\colon\Lambda^{p,q}\to\Lambda^{p+q+1,-1}=0$. Thus, for each eigenvector $\phi$ of $s^*s$, there exists an $m$ such that $s^{m+1}\phi=0$ but $s^m\phi\neq0$. As we will prove in the following propositions, the eigenvalue corresponding to $\phi$ is uniquely determined by $m$, so $m$ indexes the eigenspaces of $s^*s$.

\begin{lemma} \label{lemma:wedgecontract}
  If $\alpha$ is a $k$-covector, then
  \begin{equation*}
    \sum_{i=1}^ne^i\wedge(e_i\contract\alpha)=k\alpha.
  \end{equation*}
\end{lemma}
\begin{proof}
  It suffices to check on a basis. If $\alpha=e^I$, then $e_i\contract\alpha$ is nonzero if and only if $i\in I$. If so, then $e^i\wedge(e_i\contract\alpha)=\alpha$. Therefore,
  \begin{equation*}
    \sum_{i=1}^ne^i\wedge(e_i\contract\alpha)=\sum_{i\in I}\alpha=k\alpha.\qedhere
  \end{equation*}
\end{proof}

\begin{proposition} \label{prop:ssstar}
  On $\Lambda^{p,q}$, we have
  \begin{equation*}
    ss^*-s^*s=p-q.
  \end{equation*}
\end{proposition}

\begin{proof}
  It suffices to check on simple tensors. We compute that
  \begin{align*}
    ss^*(\alpha\otimes\beta)&=\sum_{i,j}\left(e^i\wedge(e_j\contract\alpha)\right)\otimes\left(e_i\contract(e^j\wedge\beta)\right)\\
                            &=\sum_{i,j}\left(e^i\wedge(e_j\contract\alpha)\right)\otimes\left(\delta_i^j\beta-e^j\wedge(e_i\contract\beta)\right),\\
    s^*s(\alpha\otimes\beta)&=\sum_{i,j}\left(e_j\contract(e^i\wedge\alpha)\right)\otimes\left(e^j\wedge(e_i\contract\beta)\right)\\
                            &=\sum_{i,j}\left(\delta^i_j\alpha-e^i\wedge(e_j\contract\alpha)\right)\otimes\left(e^j\wedge(e_i\contract\beta)\right),
  \end{align*}
  where $\delta$ denotes the Kronecker delta.
  Subtracting, we find that the second terms on each line cancel, leaving
  \begin{equation*}
    \begin{split}
      (ss^*-s^*s)(\alpha\otimes\beta)&=\sum_{i,j}\left(e^i\wedge(e_j\contract\alpha)\right)\otimes\delta_i^j\beta-\delta^i_j\alpha\otimes\left(e^j\wedge(e_i\contract\beta)\right)\\
      &=\sum_i(e^i\wedge(e_i\contract\alpha))\otimes\beta-\alpha\otimes(e^i\wedge(e_i\contract\beta))\\
      &=p\alpha\otimes\beta-\alpha\otimes q\beta.\qedhere
    \end{split}
  \end{equation*}
\end{proof}

\begin{remark}
A special case of Proposition~\ref{prop:ssstar} appears in~\cite[Equation (73)]{arnold2021complexes}.
\end{remark}

Proposition~\ref{prop:ssstar} gives us a quick way to determine when $s$ is injective or surjective, yielding a much shorter proof of \cite[Lemma~2]{arnold2021complexes}.

\begin{lemma}\label{lem:pq0}
  On $\Lambda^{p,q}$, if $0\le q\le p\le n$, then $\ker s$ has a nonzero element. Likewise, if $0\le p\le q\le n$, then $\ker s^*$ has a nonzero element.
\end{lemma}

\begin{proof}
  Assume $0\le q\le p\le n$. Let $\alpha=\alpha_1\wedge\dotsb\wedge\alpha_p$ be a nonzero $p$-covector, which is possible since $p\le n$. Since $q\le p$, we can let $\beta=\alpha_1\wedge\dotsb\wedge\alpha_q$, so $\beta$ is also nonzero and hence so is $\alpha\otimes\beta$. In the notation of Definition~\ref{def:s}, we have $\beta_a=\alpha_a$ for $1\le a\le q$, which implies that $\beta_a\wedge\alpha=0$ for all $a$, so $s(\alpha\otimes\beta)=0$. The second claim follows by symmetry.
\end{proof}

\begin{proposition}\label{prop:sinjsurj}
  Assume $0\le p,q\le n$, where $n=\dim V$. The operator $s\colon\Lambda^{p,q}\to\Lambda^{p+1,q-1}$ is injective if and only if $p<q$ and surjective if and only if $p\ge q-1$.
\end{proposition}

\begin{proof}
  Assume $p<q$. By Proposition~\ref{prop:ssstar}, we have $s^*s=ss^*+q-p$. Since $ss^*$ is positive semidefinite and $p<q$, we know that $s^*s$ is positive definite, and hence $s$ has zero kernel. Conversely, if $p\ge q$, then $s$ has nonzero kernel by Lemma~\ref{lem:pq0}.

  By symmetry, we have that $s^*\colon\Lambda^{p,q}\to\Lambda^{p-1,q+1}$ is injective if and only if $p>q$. Reindexing, we have that $s^*\colon\Lambda^{p+1,q-1}\to\Lambda^{p,q}$ is injective if and only if $p+1>q-1$. Hence its adjoint $s\colon\Lambda^{p,q}\to\Lambda^{p+1,q-1}$ is surjective if and only if $p+1>q-1$, which is equivalent to $p\ge q-1$.
\end{proof}

If we relax the constraint that $0\le p,q \le n$, then Proposition~\ref{prop:sinjsurj} continues to hold with the words ``if and only if'' replaced by ``if''.

\begin{proposition}\label{prop:sinjsurj2}
  The operator $s\colon\Lambda^{p,q}\to\Lambda^{p+1,q-1}$ is injective  if $p<q$ and surjective if $p\ge q-1$.
\end{proposition}
\begin{proof}
  The case $0\le p,q\le n$ is handled by Proposition~\ref{prop:sinjsurj}, so assume now that $p<0$ or $q<0$ or $p>n$ or $q>n$. In this case, $\Lambda^{p,q}=0$, so injectivity is automatic. If, additionally, $p\ge q-1$, then either $q-1\le p<0$ or $q-1<q<0$ or $p+1>p>n$ or $p+1\ge q>n$. In any case, $\Lambda^{p+1,q-1}=0$, so surjectivity is automatic.
\end{proof}

\subsection{The decomposition of double multicovectors}\label{sec:decomposition}
We can naturally decompose the space of double multicovectors into the eigenspaces of $s^*s$. We begin by investigating the eigenvalues.

\begin{lemma}\label{lem:sphieig}
  If $\phi\in\Lambda^{p,q}$ is an eigenvector of $s^*s$ with eigenvalue $\lambda$ then $s\phi$ is either zero or an eigenvector of $s^*s$ with eigenvalue $\lambda+q-p-2$.
\end{lemma}

\begin{proof}
  Since $s\phi\in\Lambda^{p+1,q-1}$, we have by Proposition~\ref{prop:ssstar} that
  \begin{equation*}
    (ss^*-s^*s)(s\phi)=((p+1)-(q-1))(s\phi)=(p-q+2)(s\phi).
  \end{equation*}
  On the other hand, since $s^*s\phi=\lambda\phi$, we have
  \begin{equation*}
    (ss^*-s^*s)(s\phi)=ss^*s\phi-s^*ss\phi=s\lambda\phi-s^*ss\phi=(\lambda-s^*s)(s\phi).
  \end{equation*}
  We conclude that
  \begin{equation*}
    s^*s(s\phi)=(\lambda+q-p-2)(s\phi)
  \end{equation*}
  as desired.
\end{proof}

\begin{proposition}\label{prop:eigenvalues}
  If $\phi\in\Lambda^{p,q}$ is an eigenvector of $s^*s$ and $m$ is the smallest integer such that $s^{m+1}\phi=0$, then $\phi$ has eigenvalue
  \begin{equation*}
    m(m+p-q+1).
  \end{equation*}
\end{proposition}

\begin{proof}
  We induct on $m$. If $m=0$ then $s\phi=0$ and so $\phi$ has eigenvalue $0$, as desired.

  Now assume that $m>0$ and that the claim holds for $m-1$ for all $p$ and $q$. In particular, we can apply the claim to $s\phi\in\Lambda^{p+1,q-1}$, because we know by the preceding lemma that $s\phi$ is an eigenvector. So then, we have that $s\phi$ has eigenvalue
  \begin{equation*}
    \begin{split}
      (m-1)((m-1)+(p+1)-(q-1)+1)&=(m-1)(m+p-q+2)\\
      &=m(m+p-q+1)+q-p-2.
    \end{split}
  \end{equation*}
  On the other hand, by the preceding lemma, if $\phi$ has eigenvalue $\lambda$, then $s\phi$ has eigenvalue $\lambda+q-p-2$, from which we conclude that $\lambda=m(m+p-q+1)$, as desired.
\end{proof}

\begin{corollary}\label{cor:mqp}
  If $\phi\in\Lambda^{p,q}$ is an eigenvector of $s^*s$ and $m$ is the smallest integer such that $s^{m+1}\phi=0$, then $m\ge q-p$.
\end{corollary}

\begin{proof}
  The claim is obvious if $p\ge q$ because $m\ge0$. On the other hand, if $p<q$, then $s$ is injective by Proposition~\ref{prop:sinjsurj}, so $s^*s$ is positive definite, so the eigenvalue $m(m+p-q+1)$ is positive. Since $m\ge0$, we conclude that $m+p-q+1$ is positive, which implies that $m\ge q-p$.
\end{proof}

\begin{corollary}
  For fixed $p$ and $q$, the eigenvalues $m(m+p-q+1)$ in the preceding proposition are strictly increasing in $m$. In particular, $m$ is determined by the eigenvalue $m(m+p-q+1)$.
\end{corollary}

\begin{proof}
  The claim follows because $m\ge0$ and $m+p-q+1>0$ and both are increasing in $m$.
\end{proof}

We are now ready to define the decomposition of double multicovectors.

\begin{definition}\label{def:decomposition}
  For an integer $m\ge0$, let $\Lambda^{p,q}_m$ be the eigenspace of $s^*s\colon\Lambda^{p,q}\to\Lambda^{p,q}$ corresponding to eigenvalue $m(m+p-q+1)$. We define these spaces to be zero if $m<0$.
\end{definition}

\begin{proposition} \label{prop:decomp}
  We have the decomposition
  \begin{equation*}
    \Lambda^{p,q}=\bigoplus_m\Lambda^{p,q}_m.
  \end{equation*}
\end{proposition}

\begin{proof}
  The operator $s^*s$ is self-adjoint and hence diagonalizable. Since $s$ is nilpotent, for each eigenvector $\phi$ there exists a smallest integer $m$ such that $s^{m+1}\phi=0$. We have shown that the corresponding eigenvalue is $m(m+p-q+1)$ and that this eigenvalue uniquely determines $m$.
\end{proof}

\begin{remark}
  Note that the definitions of the operators $s$ and $s^*$ do not require or depend on an inner product on $V$. Consequently, the eigendecomposition $\Lambda^{p,q}=\bigoplus\Lambda^{p,q}_m$ also does not depend on the inner product on $V$. (We did, however, use an arbitrary inner product to simplify the proofs of claims such as the diagonalizability of $s^*s$.) 
\end{remark}

By symmetry, the above discussion works equally well if we swap the roles of $p$ and $q$ and consider the eigendecomposition of $ss^*$ instead of $s^*s$. We call this decomposition the dual decomposition, but, as we show, it is actually the same as the decomposition above with shifted index.

\begin{definition}\label{def:dualdecomp}
  For an integer $m^*\ge0$, let $\prescript{}{m^*}\Lambda^{p,q}$ be the eigenspace of $ss^*\colon\Lambda^{p,q}\to\Lambda^{p,q}$ corresponding to eigenvalue $m^*(m^*+q-p+1)$. We define these spaces to be zero if $m^*<0$.
\end{definition}

\begin{proposition}\label{prop:dualdecomp}
  The dual decomposition is the same as the original decomposition with shifted index. Namely, $\prescript{}{m^*}\Lambda^{p,q}=\Lambda^{p,q}_m$ for $m^*=m+p-q$.
\end{proposition}

\begin{proof}
  Say $\phi\in\prescript{}{m^*}\Lambda^{p,q}$, so $\phi$ is an eigenvector of $ss^*$ with eigenvalue $m^*(m^*+q-p+1)$. By Proposition~\ref{prop:ssstar}, we have that then $\phi$ is also an eigenvector of $s^*s$ with eigenvalue $m^*(m^*+q-p+1)+q-p$. With $m^*=m+p-q$, we compute that
  \begin{equation*}
    \begin{split}
      m^*(m^*+q-p+1)+q-p&=(m+p-q)(m+1)+q-p\\
      &=m(m+p-q+1),
    \end{split}
  \end{equation*}
  so $\phi\in\Lambda^{p,q}_m$, as desired.
\end{proof}

\begin{remark}
  We caution the reader that \cite{gawlik2025title} uses the indexing for the dual decomposition. In other words, $\Lambda^{p,q}_m$ in \cite{gawlik2025title} refers to $\prescript{}m\Lambda^{p,q}$ in this paper.
\end{remark}

\subsection{Properties of the decomposition}\label{sec:properties}

Note that some terms of the decomposition may be zero. In the following propositions, we will determine exactly for which values of $m$ the space $\Lambda^{p,q}_m$ is nonzero, as well as how the operators $s$ and $s^*$ interact with the decomposition. In particular, we will prove the isomorphism $s^m\colon\Lambda^{p,q}_m\to\Lambda^{p+m,q-m}_0$.

\begin{proposition}
  The map $s$ sends $\Lambda^{p,q}_m$ to $\Lambda^{p+1,q-1}_{m-1}$. Likewise, the map $s^*$ sends $\Lambda^{p,q}_m$ to $\Lambda^{p-1,q+1}_{m+1}$.
\end{proposition}

\begin{proof}
  Let $\phi\in\Lambda^{p,q}_m$. If $s\phi$ is zero, the claim is tautological. Otherwise, $\phi$ is an eigenvector of $s^*s$ and $m$ is the smallest integer such that $s^{m+1}\phi=0$. By Lemma~\ref{lem:sphieig}, $s\phi$ is also an eigenvector of $s^*s$, and we have that $m-1$ is the smallest integer such that $s^{(m-1)+1}(s\phi)=0$. Hence, $s\phi\in\Lambda^{p+1,q-1}_{m-1}$. So, $s$ maps $\Lambda^{p,q}_m$ to $\Lambda^{p+1,q-1}_{m-1}$.

  By symmetry, $s^*$ maps $\prescript{}{m^*}\Lambda^{p,q}$ to $\prescript{}{m^*-1}\Lambda^{p-1,q+1}$. Letting $m^*=m+p-q$, we have $m^*-1=(m+1)+(p-1)-(q+1)$, so $\prescript{}{m^*}\Lambda^{p,q}=\Lambda^{p,q}_m$ and $\prescript{}{m^*-1}\Lambda^{p-1,q+1}=\Lambda^{p-1,q+1}_{m+1}$ by Proposition~\ref{prop:dualdecomp}.
\end{proof}

\begin{proposition}\label{prop:siso}
  The map $s\colon\Lambda^{p,q}_m\to\Lambda^{p+1,q-1}_{m-1}$ is injective if $m>0$ and surjective if $m^*=m+p-q\ge0$. Likewise, the map $s^*\colon\Lambda^{p,q}_m\to\Lambda^{p-1,q+1}_{m+1}$ is injective if $m^*>0$ and surjective if $m\ge0$.
\end{proposition}

\begin{proof}
  For $m>0$, the map $s^*s$ is a positive multiple of the identity on $\Lambda^{p,q}_m$, so $s\colon\Lambda^{p,q}_m\to\Lambda^{p+1,q-1}_{m-1}$ is injective and $s^*\colon\Lambda^{p+1,q-1}_{m-1}\to\Lambda^{p,q}_m$ is surjective. Reindexing, $s^*\colon\Lambda^{p,q}_m\to\Lambda^{p-1,q+1}_{m+1}$ is surjective for $m\ge0$. By symmetry, we have that the operator $s^*\colon\Lambda^{p,q}_m\to\Lambda^{p-1,q+1}_{m+1}$ is injective if $m^*>0$ and $s$ is surjective if $m^*\ge0$.
\end{proof}

We remark that the above proposition holds even if $\Lambda^{p,q}_m=0$; the bijectivity of $s^*s$ is then tautological. However, if we assume that $\Lambda^{p,q}_m$ is nonzero, then we can restate this proposition in a different way.

\begin{corollary}\label{cor:sisoalt}
  Assume $\Lambda^{p,q}_m$ is nonzero. Then $s\colon\Lambda^{p,q}_m\to\Lambda^{p+1,q-1}_{m-1}$ is an isomorphism if and only if $m>0$. Likewise, $s^*\colon\Lambda^{p,q}_m\to\Lambda^{p-1,q+1}_{m+1}$ is an isomorphism if and only if $m^*=m+p-q>0$.
\end{corollary}

\begin{proof}
  If $\Lambda^{p,q}_m$ is nonzero, then $m\ge0$ and $m^*\ge0$ (Corollary~\ref{cor:mqp}). So, if $m>0$, then $s$ is an isomorphism by Proposition~\ref{prop:siso}. If $m=0$, then $\Lambda^{p,q}_m$ is in the kernel of $s$ by definition, so $s$ is not an isomorphism. The argument for $s^*$ is analogous.
\end{proof}

\begin{corollary}\label{cor:siso}
  Each of the maps $s\colon\Lambda^{p,q}_m\to\Lambda^{p+1,q-1}_{m-1}$ and $s^*\colon\Lambda^{p,q}_m\to\Lambda^{p-1,q+1}_{m+1}$ is an isomorphism if and only if the map's domain and codomain are both zero or both nonzero.
\end{corollary}

\begin{proof}
  We prove the claim for $s$, and the claim for $s^*$ follows by symmetry. If both spaces are zero, the map is tautologically an isomorphism. If exactly one space is zero, the map cannot be an isomorphism. If both spaces are nonzero, then $\Lambda^{p+1,q-1}_{m-1}$ being nonzero implies $m-1\ge0$, so $s$ is an isomorphism by Corollary~\ref{cor:sisoalt}.
\end{proof}

Proposition~\ref{prop:siso} implies the following generalization to powers of $s$.
\begin{corollary}\label{cor:power}
  For a nonnegative integer $l$, the power $s^l\colon\Lambda^{p,q}_m\to\Lambda^{p+l,q-l}_{m-l}$ is injective if $m\ge l$ and surjective if $m^*=m+p-q\ge0$. Likewise, the power $(s^*)^l\colon\Lambda^{p,q}_m\to\Lambda^{p-l,q+l}_{m+l}$ is injective if $m^*\ge l$ and surjective if $m\ge0$.
\end{corollary}

\begin{proof}
  For the first claim, if $m\ge l$, then the numbers $m, m-1, \dots,m-l+1$ are all positive, so each map in the composition
  \begin{equation*}
    \Lambda^{p,q}_m\xrightarrow{s}\Lambda^{p+1,q-1}_{m-1}\xrightarrow{s}\dots\xrightarrow{s}\Lambda^{p+l-1,q-l+1}_{m-l+1}\xrightarrow{s}\Lambda^{p+l,q-l}_{m-l}
  \end{equation*}
  is injective, so the composition $s^l$ is injective as well.
  
  For the last claim, if $m\ge0$, then the numbers $m,m+1,\dots,m+l-1$ are all nonnegative, so each map in the composition
  \begin{equation*}
    \Lambda^{p,q}_m\xrightarrow{s^*}\Lambda^{p-1,q+1}_{m+1}\xrightarrow{s^*}\dots\xrightarrow{s^*}\Lambda^{p-(l-1),q+l-1}_{m+l-1}\xrightarrow{s^*}\Lambda^{p-l,q+l}_{m+l}
  \end{equation*}
  is surjective, so the composition $(s^*)^l$ is surjective as well. The remaining claims follow by symmetry.
\end{proof}

This corollary has an important consequence, allowing us to reduce to the $m=0$ case.

\begin{proposition}\label{prop:reducetom0}
  If $m\ge0$ and $m^*=m+p-q\ge0$, then $s^m\colon\Lambda^{p,q}_m\to\Lambda^{p+m,q-m}_0$ is an isomorphism. Otherwise, if $m<0$ or $m^*<0$, then $\Lambda^{p,q}_m$ is zero.
\end{proposition}

\begin{proof}
  We apply Corollary~\ref{cor:power} with $l=m$. For the second statement, see Definition~\ref{def:decomposition} and Corollary~\ref{cor:mqp}.
\end{proof}

We can now easily determine which decomposition summands are nonzero.

\begin{proposition}\label{prop:lpqmnonzero}
  The space $\Lambda^{p,q}_m$ is nonzero if and only if
  \begin{equation}\label{eq:mbounds}
    \max\{0,q-p\}\le m\le\min\{q,n-p\},
  \end{equation}
  where $n=\dim V$.
\end{proposition}

\begin{proof}
  We apply Proposition~\ref{prop:reducetom0} to reduce to the $m=0$ case, which we handle with Lemma~\ref{lem:pq0}. If $m<0$ or $m^*=m+p-q<0$, then $\Lambda^{p,q}_m=0$. So now assume $m\ge0$ and $m^*\ge0$, so $\Lambda^{p,q}_m$ is isomorphic to $\Lambda^{p+m,q-m}_0$, which is the kernel of $s$ on $\Lambda^{p+m,q-m}$.

  If, furthermore, $p+m>n$ or $q-m<0$, then $\Lambda^{p+m,q-m}$ is zero and hence so is $\Lambda^{p+m,q-m}_0$. So now assume furthermore that $p+m\le n$ and $q-m\ge0$. Then we have $0\le q-m\le p+m\le n$, where the middle inequality follows from $q-m=p-m^*\le p\le p+m$. Thus, $\Lambda^{p+m,q-m}_0$ is nonzero by Lemma~\ref{lem:pq0}.

  Putting everything together, we have nonzero $\Lambda^{p,q}_m$ when $m\ge0$, $m^*\ge0$, $p+m\le n$, and $q-m\ge0$, and otherwise $\Lambda^{p,q}_m=0$. Equation~\eqref{eq:mbounds} is equivalent to these four inequalities.
\end{proof}

\begin{remark}
  The preceding propositions yield a simple way to understand the injectivity and surjectivity of $s$ that we proved in Proposition~\ref{prop:sinjsurj}. On most summands of the decomposition, $s\colon\Lambda^{p,q}_m\to\Lambda^{p+1,q-1}_{m-1}$ is an isomorphism, but this fails exactly when one space is zero but the other is not.

  To see when it is possible for exactly one of the spaces to be zero, we compare Inequality~\eqref{eq:mbounds} with the corresponding inequality
  \begin{equation}\label{eq:mminusbounds}
    \max\{1,q-p-1\}\le m\le\min\{q,n-p\},
  \end{equation}
  for $\Lambda^{p+1,q-1}_{m-1}$ (adding $1$ after substituting).
  When can one inequality hold but the other fail? We see that \eqref{eq:mbounds} holds but \eqref{eq:mminusbounds} fails if and only if $m=0\ge q-p$, from which we conclude that $s$ fails to be injective if and only if $p\ge q$. On the other hand, \eqref{eq:mminusbounds} holds but \eqref{eq:mbounds} fails if and only if $m=q-p-1\ge1$, so $s$ fails to be surjective if and only if $q\ge p+2$.
\end{remark}

In addition to providing a decomposition of $\Lambda^{p,q}$, the spaces $\Lambda^{p,q}_m$ also provide a decomposition of other subspaces of $\Lambda^{p,q}$ like the kernel and image of various powers of $s$ and $s^*$.

\begin{proposition} \label{prop:kerimdecomp}
  For each nonnegative integer $l$, consider $s^l\colon\Lambda^{p,q}\to\Lambda^{p+l,q-l}$ and $(s^*)^l\colon\Lambda^{p+l,q-l}\to\Lambda^{p,q}$. We have
  \[
    \operatorname{ker} s^l = \bigoplus_{m=0}^{l-1} \Lambda^{p,q}_m, \quad \operatorname{im} (s^*)^l = \bigoplus_{m=l}^q \Lambda^{p,q}_m, \quad \Lambda^{p,q}_m = \operatorname{ker} s^{m+1} \cap \operatorname{im} (s^*)^m.
  \]
\end{proposition}
\begin{proof}
  Consider the space $\Lambda^{p+l,q-l}$, which we can decompose as
  \[
    \Lambda^{p+l,q-l} = \bigoplus_{m=l}^{q} \Lambda^{p+l,q-l}_{m-l}
  \]
  by Proposition~\ref{prop:lpqmnonzero}. We now apply $(s^*)^l$ to both sides. By Corollary~\ref{cor:power}, since $m-l\ge0$, we have
  \[
    (s^*)^l \Lambda^{p+l,q-l}_{m-l} = \Lambda^{p,q}_m.
  \]
  Thus,
  \[
    \operatorname{im} (s^*)^l = \bigoplus_{m=l}^q \Lambda^{p,q}_m.
  \]
  Taking the orthogonal complement of both sides yields $\ker s^l = \bigoplus_{m=0}^{l-1} \Lambda^{p,q}_m$, and from this it follows that $\Lambda^{p,q}_m = \operatorname{ker} s^{m+1} \cap \operatorname{im} (s^*)^m$.
\end{proof}

The preceding proposition can be used to give a simpler proof of \cite[Lemma 2.4]{hu2025finite}, which generalizes Proposition~\ref{prop:sinjsurj} to powers of $s$ and $s^*$.

\begin{proposition} \label{prop:sminjsurj}
  Assume $0 \le p,q \le n$, where $n = \dim V$.  Let $l$ be a positive integer.  The operator $s^l : \Lambda^{p,q} \to \Lambda^{p+l,q-l}$ is injective if and only if $p < q-l+1$ and surjective if and only if $p \ge q-l$.
\end{proposition}
\begin{proof}
  The operator $s^l$ fails to be injective if and only if there is at least one nonzero summand in the decomposition $\operatorname{ker} s^l = \bigoplus_{m=0}^{l-1} \Lambda^{p,q}_m$.  By Proposition~\ref{prop:lpqmnonzero}, this happens if and only if there exists an $m \in \{0,1,\dots,l-1\}$ for which
  \[
    \max\{0,q-p\} \le m \le \min\{q,n-p\},
  \]
  i.e. the intervals $[0,l-1]$ and $[\max\{0,q-p\},\min\{q,n-p\}]$ have nonempty intersection.  This happens if and only if $\max\{0,q-p\} \le l-1$, which is equivalent to $p \ge q-l+1$.  Therefore $s^l$ is injective if and only if $p < q-l+1$.

  To study surjectivity, we use the same strategy as in the proof of Proposition~\ref{prop:sinjsurj}.  By symmetry, we have that $(s^*)^l\colon\Lambda^{p,q}\to\Lambda^{p-l,q+l}$ is injective if and only if $q<p-l+1$. Reindexing, we have that $(s^*)^l\colon\Lambda^{p+l,q-l}\to\Lambda^{p,q}$ is injective if and only if $q-l<(p+l)-l+1$. Hence its adjoint $s^l\colon\Lambda^{p,q}\to\Lambda^{p+l,q-l}$ is surjective if and only if $q-l<p+1$, which is equivalent to $p \ge q-l$.
\end{proof}

\subsection{Decomposition examples}\label{sec:decompeg}
We now provide several examples of the above decomposition.  We begin by discussing the decomposition of $\Lambda^{p,q}$ for small values of $p$ and $q$ in Sections~\ref{sec:pq11}--\ref{sec:pq22}.  Then, in Section~\ref{sec:symmetric}, we discuss the special case $p=q$.  We conclude in Section~\ref{sec:mq} by discussing the $m=q$ summand in the decomposition, which, as we will eventually see, is somewhat exceptional.

\subsubsection{The decomposition for $(p,q)=(1,1)$.} \label{sec:pq11}
When $p=q=1$, the decomposition in Proposition~\ref{prop:decomp} reads
\[
  \Lambda^{1,1} = \Lambda^{1,1}_0 \oplus \Lambda^{1,1}_1.
\]
As we will show in Proposition~\ref{prop:imageipq}, $\Lambda^{1,1}_1$ is the image of $\Lambda^2$ under the natural inclusion $i^{1,1} : \Lambda^2 \hookrightarrow \Lambda^{1,1}$.  In other words, $\Lambda^{1,1}_1$ consists of antisymmetric bilinear forms.  Consequently, $\Lambda^{1,1}_0$ consists of symmetric bilinear forms.

\subsubsection{The decomposition for $(p,q)=(2,1)$.} \label{sec:pq21}
When $(p,q)=(2,1)$, the decomposition in Proposition~\ref{prop:decomp} reads
\[
  \Lambda^{2,1} = \Lambda^{2,1}_0 \oplus \Lambda^{2,1}_1.
\]

In dimension $n=3$, we can understand these spaces by identifying elements of $\Lambda^{2,1}$ with matrices.  Specifically, we can write any $\phi \in \Lambda^{2,1}$ in the form
\[
  \phi = \sum_{i,j=1}^3 a_{ij} \alpha^i \otimes e^j,
\]
where $e^1,e^2,e^3$ is a basis for $V^*$, $\alpha^1 = e^2 \wedge e^3$, $\alpha^2 = e^3 \wedge e^1$, and $\alpha^3 = e^1 \wedge e^2$.  In this basis,
\[
  s\phi = \sum_{i,j=1}^3 a_{ij} e^j \wedge \alpha^i = \left( \sum_{i=1}^3 a_{ii} \right) e^1 \wedge e^2 \wedge e^3,
\]
so $\phi$ belongs to $\Lambda^{2,1}_0 = \ker s$ if and only if the matrix of coefficients $[a_{ij}]_{i,j=1}^3$ is trace-free.  Thus, in three dimensions, the decomposition $\Lambda^{2,1} = \Lambda^{2,1}_0 \oplus \Lambda^{2,1}_1$ is simply the decomposition of a $3 \times 3$ matrix into its deviatoric part plus a multiple of the identity.

\subsubsection{The decomposition for $(p,q)=(2,2)$.} \label{sec:pq22}
When $(p,q)=(2,2)$, there are three summands in the decomposition:
\[
  \Lambda^{2,2} = \Lambda^{2,2}_0 \oplus \Lambda^{2,2}_1 \oplus \Lambda^{2,2}_2.
\]
We will first discuss the summands in any dimension $n$ and then specialize to $n=3$.

The space $\Lambda^{2,2}_0 = \ker s$ consists of $(2,2)$-forms that satisfy the \emph{Bianchi identity}
\begin{align*}
  (s\phi)(X,Y,Z;W) &= \phi(Y,Z;X,W) - \phi(X,Z;Y,W) + \phi(X,Y;Z,W) \\
                   &= \phi(Y,Z;X,W) + \phi(Z,X;Y,W) + \phi(X,Y;Z,W) = 0.
\end{align*}
Such a $(2,2)$-form is called an \emph{algebraic curvature tensor} because it possesses the same symmetries as the Riemann curvature tensor.  Namely, $\phi(X,Y;Z,W)$ alternates in $X$ and $Y$, alternates in $Z$ and $W$, and satisfies the Bianchi identity.  It can be shown~\cite[p.~204]{lee2018introduction} that such tensors automatically possess the additional symmetry 
\[
  \phi(X,Y;Z,W) = \phi(Z,W;X,Y).
\]

As we will show in Proposition~\ref{prop:imageipq}, the space $\Lambda^{2,2}_2$ is the image of $\Lambda^4$ under the natural inclusion $i^{2,2} : \Lambda^4 \hookrightarrow \Lambda^{2,2}$.  As such, it consists of tensors that alternate in all 4 arguments.  In particular, such tensors satisfy the symmetry $\phi(X,Y;Z,W) = \phi(Z,W;X,Y)$ as well.  This implies that any $(2,2)$-form satisfying the \emph{skew}-symmetry
\[
  \phi(X,Y;Z,W) = -\phi(Z,W;X,Y)
\]
must belong to the remaining space $\Lambda^{2,2}_1$.  In fact, $\Lambda^{2,2}_1$ consists precisely of those $(2,2)$-forms $\phi$ satisfying $\phi(X,Y;Z,W) = -\phi(Z,W;X,Y)$.  One way to show this is to count dimensions: By Proposition~\ref{prop:siso}, the dimension of $\Lambda^{2,2}_1$ matches the dimension of $\Lambda^{3,1}_0$, and this space is the kernel of the surjective map $s : \Lambda^{3,1} \to \Lambda^{4,0}$.  Hence it has dimension
\[
  \dim \Lambda^{3,1} - \dim \Lambda^{4,0} = n \binom{n}{3} - \binom{n}{4} = \frac{1}{2} \binom{n}{2} \left( \binom{n}{2} - 1 \right).
\]
Since this number matches the dimension of the space of $(2,2)$-forms satisfying the skew-symmetry $\phi(X,Y;Z,W) = -\phi(Z,W;X,Y)$, the claim follows.

In dimension $n=3$, the situation simplifies.  There are no 4-forms in 3 dimensions, so $\Lambda^{2,2}_2$ vanishes.  By the discussion above, the remaining spaces $\Lambda^{2,2}_0$ and $\Lambda^{2,2}_1$ must therefore consist of all symmetric $(2,2)$-forms and all skew-symmetric $(2,2)$-forms, respectively.  If, in the notation of Section~\ref{sec:pq21}, we identify a $(2,2)$-form $\phi = \sum_{i,j=1}^3 a_{ij} \alpha^i \otimes \alpha^j$ with a $3 \times 3$ matrix $A = [a_{ij}]_{i,j=1}^3$, then $\phi$ belongs to $\Lambda^{2,2}_0$ (respectively, $\Lambda^{2,2}_1$) if and only if $A$ is symmetric (respectively, skew-symmetric).

\subsubsection{Symmetric and skew-symmetric $(p,p)$-forms.} \label{sec:symmetric}
The discussion above about symmetric and skew-symmetric $(2,2)$-forms can be generalized to $(p,p)$-forms with $p$ a positive integer.  Let
\[\begin{split}
    \Lambda^{p,p}_{\rm sym} = \big\{ \phi \in \Lambda^{p,p} \mid \phi(X_1,\dots,X_p;Y_1,\dots,Y_p) = \phi(Y_1,\dots,Y_p;X_1,\dots,X_p) \\ \text{ for all } X_1,\dots,X_p,Y_1,\dots,Y_p \big\}
  \end{split}\]
and
\[\begin{split}
    \Lambda^{p,p}_{\rm skw} = \big\{ \phi \in \Lambda^{p,p} \mid \phi(X_1,\dots,X_p;Y_1,\dots,Y_p) = -\phi(Y_1,\dots,Y_p;X_1,\dots,X_p) \\ \text{ for all } X_1,\dots,X_p,Y_1,\dots,Y_p \big\}.
  \end{split}\]
In Section~\ref{sec:pq11} we pointed out that
\[
  \Lambda^{1,1}_{\rm sym} = \Lambda^{1,1}_0, \quad \Lambda^{1,1}_{\rm skw} = \Lambda^{1,1}_1,
\]
and the discussion in Section~\ref{sec:pq22} implies that
\[
  \Lambda^{2,2}_{\rm sym} = \Lambda^{2,2}_0 \oplus \Lambda^{2,2}_2, \quad \Lambda^{2,2}_{\rm skw} = \Lambda^{2,2}_1.
\]
More generally, it can be shown that
\[
  \Lambda^{p,p}_{\rm sym} = \bigoplus_{m \text{ even }} \Lambda^{p,p}_m, \quad\quad \Lambda^{p,p}_{\rm skw} = \bigoplus_{m \text{ odd }} \Lambda^{p,p}_m;
\]
see~\cite[Exercises 6.16* and 15.32*]{fulton2004representation} and~\cite{gawlik2025title}.

\subsubsection{The $m=q$ summand}\label{sec:mq}
It turns out that the case $m=q$ is quite special. In fact, we will eventually see in Sections~\ref{sec:extension} and \ref{sec:finite} that $\Lambda^{p,q}_q$ is the only summand that fails to admit a tangentially continuous piecewise constant finite element discretization in dimension greater than $p+q$.  It will therefore be important for us to characterize this space.

Let $k=p+q$. Observe that a $k$-covector, being antisymmetric in all indices, is, in particular, antisymmetric in the first $p$ indices and in the last $q$ indices. Thus, we have a natural inclusion $\Lambda^k\hookrightarrow\Lambda^{p,q}$. As we will see, $\Lambda^{p,q}_q$ is the image of this map. Conversely, the wedge product yields a natural map $\wedge\colon\Lambda^{p,q}\to\Lambda^k$. As we will see, $\Lambda^{p,q}_q$ is the orthogonal complement of $\ker\wedge$.

\begin{definition}\label{def:ipq}
	For $k=p+q$, let
	\begin{equation*}
		i^{p,q}\colon\Lambda^k\to\Lambda^{p,q}
	\end{equation*}
	denote the natural inclusion of antisymmetric $k$-tensors into the space of $(p,q)$-covectors given by
	\begin{equation*}
		(i^{p,q}\psi)(X_1,\dots,X_p;Y_1,\dots,Y_q):=\psi(X_1,\dots,X_p,Y_1,\dots,Y_q).
	\end{equation*}
	If either $p$ or $q$ are negative we define $i^{p,q}$ to be zero.
\end{definition}
\begin{definition}
	Let
	\begin{equation*}
		\wedge\colon\Lambda^{p,q}\to\Lambda^k
	\end{equation*}
	denote the wedge product map defined on simple tensors by
	\begin{equation*}
		\wedge(\alpha\otimes\beta):=\alpha\wedge\beta.
	\end{equation*}
\end{definition}

These operators have the following relationships with $s$.

\begin{proposition}\label{prop:si}
	We have
	\begin{align*}
		si^{p,q}&=(-1)^p(p+1)i^{p+1,q-1},\\
		s^*i^{p,q}&=(-1)^{p-1}(q+1)i^{p-1,q+1}.
	\end{align*}
\end{proposition}

\begin{proof}
	Let $\psi\in\Lambda^k$. Using the antisymmetry of $\psi$, we have
	\begin{equation*}
		\begin{split}
			&(si^{p,q}\psi)(X_1,\dots,X_{p+1};Y_1,\dots,Y_{q-1})\\
			&=\sum_{a=1}^{p+1}(-1)^{a-1}(i^{p,q}\psi)(X_1,\dots,\widehat{X_a},\dots,X_{p+1};X_a,Y_1,\dots,Y_{q-1})\\
			&=\sum_{a=1}^{p+1}(-1)^{a-1}\psi(X_1,\dots,\widehat{X_a},\dots,X_{p+1},X_a,Y_1,\dots,Y_{q-1})\\
			&=\sum_{a=1}^{p+1}(-1)^p\psi(X_1,\dots,X_a,\dots,X_{p+1},Y_1,\dots,Y_{q-1})\\
			&=(-1)^p(p+1)\psi(X_1,\dots,X_a,\dots,X_{p+1},Y_1,\dots,Y_{q-1}).
		\end{split}
	\end{equation*}
	The second claim follows similarly, with care taken about the signs.
\end{proof}

\begin{proposition}\label{prop:wedges}
	On $(p,q)$-forms, we have
	\begin{align*}
		\mathop\wedge s&=(-1)^pq\mathop\wedge,\\
		\mathop\wedge s^*&=(-1)^{p-1}p\mathop\wedge.
	\end{align*}
\end{proposition}

\begin{proof}
	On simple tensors, we have
	\begin{equation*}
		\begin{split}
			&\mathop\wedge s\bigl((\alpha_1 \wedge \dots \wedge \alpha_p) \otimes (\beta_1 \wedge \dots \wedge \beta_q) \bigr) \\
			&= \sum_{a=1}^q (-1)^{a-1}\mathop\wedge\bigl((\beta_a \wedge \alpha_1 \wedge \dots \wedge \alpha_p) \otimes (\beta_1 \wedge \dots \wedge \widehat{\beta_a} \wedge \dots \wedge \beta_q)\bigr)\\
			&=\sum_{a=1}^q(-1)^{a-1}\beta_a \wedge \alpha_1 \wedge \dots \wedge \alpha_p\wedge\beta_1 \wedge \dots \wedge \widehat{\beta_a} \wedge \dots \wedge \beta_q\\
			&=\sum_{a=1}^q(-1)^p\alpha_1\wedge\dots\wedge\alpha_p\wedge\beta_1\wedge\dots\wedge\beta_a\wedge\dots\wedge\beta_q\\
			&=(-1)^pq(\alpha_1 \wedge \dots \wedge \alpha_p) \wedge (\beta_1 \wedge \dots \wedge \beta_q).
		\end{split}
	\end{equation*}
	The second claim follows similarly, with care taken about the signs.
\end{proof}

\begin{proposition}\label{prop:imageipq}
	The image of $i^{p,q}\colon\Lambda^k\to\Lambda^{p,q}$ is $\Lambda^{p,q}_q$.
\end{proposition}

\begin{proof}
	Assume $0\le p,q$; otherwise, the claim is tautological because $\Lambda^{p,q}=0$.
	
	We induct on $q$. If $q=0$, then $p=k$, and $m=0$ is the only decomposition summand, so $i^{p,q}$ is just the isomorphism $\Lambda^k\to\Lambda^{k,0}=\Lambda^{k,0}_0$.
	
	Now let $q\ge1$ and assume that the claim holds for $q-1$. Then $i^{p,q}=(-1)^pq^{-1}s^*i^{p+1,q-1}$. By the inductive hypothesis, the image of $i^{p+1,q-1}$ is $\Lambda^{p+1,q-1}_{q-1}$. Checking that $q-1\ge0$ and $(q-1)+(p+1)-(q-1)=p+1>0$, Proposition~\ref{prop:siso} tells us that $s^*$ is an isomorphism from $\Lambda^{p+1,q-1}_{q-1}$ to $\Lambda^{p,q}_q$. Hence, the image of $i^{p,q}=(-1)^pq^{-1}s^*i^{p+1,q-1}$ is $\Lambda^{p,q}_q$.
\end{proof}

\begin{proposition}\label{prop:kerwedge}
	The space $\Lambda^{p,q}_q$ is the orthogonal complement of the kernel of $\wedge\colon\Lambda^{p,q}\to\Lambda^k$.
\end{proposition}

\begin{proof}
	Again, assume $0\le p,q$; otherwise, the claim is tautological.
	
	We induct on $q$ on the statement that $\wedge\colon\Lambda^{p,q}_m\to\Lambda^k$ is zero if $m<q$ and is an isomorphism if $m=q$. As before, if $q=0$, then $p=k$, and $m=0$ is the only decomposition summand, so we see that $\wedge\colon\Lambda^{k,0}_0=\Lambda^{k,0}\to\Lambda^k$ is the obvious isomorphism.
	
	Now assume $q\ge1$ and that the proposition holds for $q-1$. On $(p,q)$-forms, we have $\wedge=(-1)^pq^{-1}\mathop\wedge s$.
	
	Consider first the case $m=q$. Since $m>0$ and $m^*=m+p-q=p\ge0$, Proposition~\ref{prop:siso} tells us that $s$ is an isomorphism from $\Lambda^{p,q}_q$ to $\Lambda^{p+1,q-1}_{q-1}$, and then $\wedge$ is an isomorphism from $\Lambda^{p+1,q-1}_{q-1}$ to $\Lambda^k$ by the inductive hypothesis. Hence the composition $\wedge=(-1)^pq^{-1}\mathop\wedge s$ is an isomorphism from $\Lambda^{p,q}_q$ to $\Lambda^k$.
	
	Now consider the case $m<q$. If the space $\Lambda^{p,q}_m$ is zero, then the claim is tautological, so we may assume $m\ge0$ and $m^*\ge0$. If $m=0$, then $s$ sends $\Lambda^{p,q}_m$ to zero. If $m>0$, then $s$ is an isomorphism from $\Lambda^{p,q}_m$ to $\Lambda^{p+1,q-1}_{m-1}$ by Proposition~\ref{prop:siso}, and then $\wedge$ sends $\Lambda^{p+1,q-1}_{m-1}$ to zero by the inductive hypothesis. Either way, the composition $\wedge=(-1)^pq^{-1}\mathop\wedge s$ is zero, as desired.
\end{proof}

\subsection{Additional operations on double multicovectors}\label{sec:additionaloperations}
In addition to $s$ and $s^*$, there are several other natural operations on double multicovectors. As before, we work on a vector space $V$ of dimension $n$.
\begin{definition}
  Let the \emph{transposition operator}
  \[
    \tau : \Lambda^{p,q} \to \Lambda^{q,p}
  \]
  be the involution that swaps the two factors, that is, on simple tensors, we have
  \[
    \tau(\alpha \otimes \beta) = \beta \otimes \alpha.
  \]
\end{definition}

\begin{definition} \label{def:owedge}
  Let the \emph{double wedge product}, sometimes called the Kulkarni--Nomizu product, be the binary operation
  \[
    \owedge : \Lambda^{p,q} \times \Lambda^{p',q'} \to \Lambda^{p+p',q+q'},
  \]  
  that is defined on simple tensors by
  \[
    (\alpha \otimes \beta) \owedge (\gamma \otimes \delta) = (\alpha \wedge \gamma) \otimes (\beta \wedge \delta).
  \]
\end{definition}

\begin{example} \label{ex:KN}
If $p=q=p'=q'=1$, then the $(2,2)$-covector $(\alpha \otimes \beta) \owedge (\gamma \otimes \delta) =: \varphi$ satisfies 
\begin{align*}
\varphi(X,Y;Z,W)
&= (\alpha \wedge \gamma)(X,Y) (\beta \wedge \delta)(Z,W) \\
&= (\alpha(X)\gamma(Y) - \alpha(Y)\gamma(X)) (\beta(Z)\delta(W)-\beta(W)\delta(Z)) \\
&= \alpha(X)\beta(Z)\gamma(Y)\delta(W) - \alpha(X)\beta(W)\gamma(Y)\delta(Z) \\
&\;\; - \alpha(Y)\beta(Z)\gamma(X)\delta(W) + \alpha(Y)\beta(W)\gamma(X)\delta(Z) \\
&= (\alpha \otimes \beta)(X;Z) (\gamma \otimes \delta)(Y;W) - (\alpha \otimes \beta)(X;W) (\gamma \otimes \delta)(Y;Z) \\
&\;\; - (\alpha \otimes \beta)(Y;Z) (\gamma \otimes \delta)(X;W) + (\alpha \otimes \beta)(Y;W) (\gamma \otimes \delta)(X;Z)
\end{align*}
for all vectors $X,Y,Z,W$.  By linearity, it follows that the double wedge product of any pair of $(1,1)$-covectors $\sigma$ and $\rho$ satisfies
\begin{align*}
(\sigma \owedge \rho)(X,Y;Z,W) 
&= \sigma(X;Z)\rho(Y;W) - \sigma(X;W)\rho(Y;Z) \\
&\;\; - \sigma(Y;Z)\rho(X;W) + \sigma(Y;W)\rho(X;Z).
\end{align*}
\end{example}
In the differential geometry literature, the double wedge product of two double multicovectors of arbitrary degree has been used in a variety of contexts~\cite{de2012differentiable,calabi1961compact,gray1970some,kulkarni1972bianchi,labbi2005double,kupferman2024double}.  In these references it is denoted $\wedge$, $\cdot$, or with concatenation.

\begin{definition} \label{def:ostar}
  Let the \emph{double Hodge star} be the operator
  \[
    \ostar : \Lambda^{p,q} \to \Lambda^{n-p,n-q}
  \]
  that is defined on simple tensors by
  \begin{equation*}
    \ostar(\alpha \otimes \beta)=\star \alpha \otimes \star \beta,
  \end{equation*}
  where $\star$ is the Hodge star.
\end{definition}

\begin{remark}
  Similarly to $s$ and $s^*$, the definitions of the operators $\tau$ and $\owedge$ do not require or depend on an inner product on $V$. In contrast, because $\star$ does depend on an inner product on $V$, so does $\ostar$.
\end{remark}

Our goal now is to prove the compatibility of the double Hodge star with the above decomposition of double multicovectors. We begin with some basic properties of these operators.

\begin{proposition}
  We have
  \begin{align*}
    \tau s&= s^*\tau,\\
    \tau s^*&=s\tau,\\
    \tau\ostar&=\ostar\tau.
  \end{align*}
\end{proposition}

\begin{proof}
  The claims follow from the symmetry between the definitions of $s$ and $s^*$, and from the symmetry in the definition of $\ostar$.
\end{proof}

By symmetry, $\tau$ sends the decomposition to the dual decomposition.
\begin{proposition}\label{prop:taudecomp}
  The transposition $\tau$ is an isomorphism between $\Lambda^{p,q}_m$ and $\Lambda^{q,p}_{m^*}$, where $m^*=m+p-q$.
\end{proposition}

\begin{proof}
  Say $\phi\in\Lambda^{q,p}_{m^*}$. Then, by Definition~\ref{def:decomposition}, $\phi$ is an eigenvector of $s^*s$ with eigenvalue $m^*(m^*+q-p+1)$. Since $\tau s^*s=ss^*\tau$, we have that $\tau\phi\in\Lambda^{p,q}$ is an eigenvector of $ss^*$ with the same eigenvalue, $m^*(m^*+q-p+1)$. Therefore, by Definition~\ref{def:dualdecomp} and Proposition~\ref{prop:dualdecomp}, $\tau\phi\in\prescript{}{m^*}\Lambda^{p,q}=\Lambda^{p,q}_m$, as desired.
\end{proof}

\begin{proposition}
  On $\Lambda^{p,q}$, we have
  \begin{equation*}
    \ostar^2=(-1)^{p(n-p)+q(n-q)}.
  \end{equation*}
\end{proposition}

\begin{proof}
  The claim follows from the fact that, on $k$-covectors, $\star^2=(-1)^{k(n-k)}$.
\end{proof}

\begin{proposition}
  For any $\phi,\psi \in \Lambda^{p,q}$, we have
  \[
    \langle \varphi, \psi \rangle = \ostar^{-1} (\varphi \owedge \ostar \psi).
  \]
\end{proposition}

\begin{proof}
  It suffices to prove the claim for simple tensors $\phi=\alpha\otimes\beta$ and $\psi=\gamma\otimes\delta$. By properties of the Hodge star, we have
  \begin{multline*}
    \langle\phi,\psi\rangle=\langle\alpha,\gamma\rangle\langle\beta,\delta\rangle=\left(\star^{-1}(\alpha\wedge\star\gamma)\right)\left(\star^{-1}(\beta\wedge\star\delta)\right)\\
    =\ostar^{-1}\left((\alpha\otimes\beta)\owedge\ostar(\gamma\otimes\delta)\right).\qedhere
  \end{multline*}
\end{proof}

\begin{lemma}\label{lemma:tubes}
  For any $\varphi \in \Lambda^{p,q}$ and $\psi \in \Lambda^{p',q'}$, we have
  \[
    s(\varphi \owedge \psi) = (s\varphi) \owedge \psi + (-1)^k \varphi \owedge (s\psi),
  \]
  where $k=p+q$.
\end{lemma}

\begin{proof}
  One can verify this claim using Proposition~\ref{prop:wedgecontract} and properties of the interior product. See also~\cite[Proposition 2.1]{gray1970some}.
\end{proof}

\begin{example} \label{ex:tubes}
Let $\varphi = \alpha \otimes \beta \in \Lambda^{1,1}$ and $\psi = \gamma \otimes \delta \in \Lambda^{1,1}$.  Using Definition~\ref{def:owedge} to compute $\owedge$ and the second part of Definition~\ref{def:s} to compute $s$, we obtain
\begin{align*}
s(\varphi \owedge \psi)
&= s( (\alpha \wedge \gamma) \otimes (\beta \wedge \delta) ) \\
&= (\beta \wedge \alpha \wedge \gamma) \otimes \delta - (\delta \wedge \alpha \wedge \gamma) \otimes \beta.
\end{align*}
Meanwhile,
\begin{align*}
(s\varphi) \owedge \psi + \varphi \owedge (s\psi) 
&= ((\beta \wedge \alpha) \otimes 1) \owedge (\gamma \otimes \delta) + (\alpha \otimes \beta) \owedge ((\delta \wedge \gamma) \otimes 1) \\
&= (\beta \wedge \alpha \wedge \gamma) \otimes \delta + (\alpha \wedge \delta \wedge \gamma) \otimes \beta \\
&= s(\varphi \owedge \psi).
\end{align*}
\end{example}

\begin{proposition}
  On $\Lambda^{p,q}$, we have
  \begin{align*}
    \ostar s^* &= (-1)^{k+1} s \ostar,\\
    \ostar s &= (-1)^{k+1}s^* \ostar,  
  \end{align*}
  where $k=p+q$.
\end{proposition}

\begin{proof}
  Let $\varphi \in \Lambda^{p-1,q+1}$ and $\psi \in \Lambda^{p,q}$.  Notice that $\varphi \owedge \ostar \psi$ belongs to $\Lambda^{n-1,n+1}=0$, so $s(\varphi \owedge \ostar \psi) = 0$.  Note also that $(p-1)+(q+1)=k$. Thus, Lemma~\ref{lemma:tubes} implies that
  \[
    (s \varphi) \owedge \ostar \psi = (-1)^{k+1} \varphi \owedge (s \ostar \psi).
  \]
  Equivalently,
  \[
    \langle s \varphi, \psi \rangle = (-1)^{k+1}  \langle \varphi, \ostar^{-1} s \ostar \psi \rangle.
  \]
  Since $\varphi$ and $\psi$ are arbitrary and $s$ and $s^*$ are adjoints, we conclude that
  \[
    s^* = (-1)^{k+1} \ostar^{-1} s \ostar,
  \]
  from which the first claim follows. Conjugating by $\tau$ and using the fact that $\tau$ preserves $k$ and commutes with $\ostar$ and hence $\ostar^{-1}$, we obtain
  \begin{equation*}
    \tau s^*\tau = (-1)^{k+1} \ostar^{-1} \tau s \tau \ostar,
  \end{equation*}
  and so the second claim follows by $\tau s^*\tau =s$ and $\tau s\tau=s^*$.
\end{proof}

We are now ready to prove that the double Hodge star respects the decomposition.

\begin{proposition}
  The operators $s^*s$ and $\tau\ostar$ commute.
\end{proposition}

\begin{proof}
  On $\Lambda^{p,q}$, with $k=p+q$, noting that $s$ preserves $k$, we compute that
  \begin{equation*}
    s^*s\tau\ostar=s^*\tau s^*\ostar=\tau ss^*\ostar=(-1)^{k+1}\tau s\ostar s=\tau \ostar s^*s.\qedhere
  \end{equation*}
\end{proof}

\begin{proposition}\label{prop:tauostardecomp}
  The isomorphism $\tau\ostar$ preserves the decomposition of double multicovectors, sending $\Lambda^{p,q}_m$ to $\Lambda^{n-q,n-p}_m$.
\end{proposition}

\begin{proof}
  Let $\phi\in\Lambda^{p,q}$ be an eigenvector of $s^*s$. Then $\tau\ostar\phi\in\Lambda^{n-q,n-p}$ is an eigenvector of $s^*s$ with the same eigenvalue. We then observe that
  \begin{equation*}
    m(m+p-q+1)=m(m+(n-q)-(n-p)+1),
  \end{equation*}
  so this eigenvalue corresponds to the same value of $m$ in both $\Lambda^{p,q}$ and $\Lambda^{n-q,n-p}$.
\end{proof}

\begin{example} \label{ex:tauostar1}
Let $n=3$, $p=2$, $q=1$, and $m=0$.  In the notation of Section~\ref{sec:pq21}, consider a $(2,1)$-covector
\[
\varphi = \sum_{i,j=1}^3 a_{ij} \alpha^i \otimes e^j.
\]
As discussed in Section~\ref{sec:pq21}, $\varphi$ belongs to $\Lambda^{2,1}_0$ if and only if $\sum_i a_{ii}=0$.  Meanwhile,
\[
\tau \ostar \varphi = \tau \sum_{i,j=1}^3 a_{ij} \, {\star}\alpha^i \otimes \star e^j =  \tau \sum_{i,j=1}^3 a_{ij} e^i \otimes \alpha^j = \sum_{i,j=1}^3 a_{ij} \alpha^j \otimes e^i,
\]
so $\tau \ostar \varphi$ belongs to $\Lambda^{2,1}_0$ if and only if $\sum_i a_{ii}=0$ as well.
\end{example}
\begin{example} \label{ex:tauostar2}
Let $n=3$, $p=1$, $q=1$, and $m=1$.  Using the same notation as above, consider a $(1,1)$-covector
\[
\varphi = \sum_{i,j=1}^3 a_{ij} e^i \otimes e^j.
\]
As discussed in Section~\ref{sec:pq11}, $\varphi$ belongs to $\Lambda^{1,1}_1 = \Lambda^{1,1}_{\rm skw}$ if and only if $a_{ij}=-a_{ji}$ for each $i,j$.  Meanwhile,
\[
\tau\ostar\varphi = \tau \sum_{i,j=1}^3 a_{ij} \, {\star} e^i \otimes \star e^j =  \tau \sum_{i,j=1}^3 a_{ij} \alpha^i \otimes \alpha^j = \sum_{i,j=1}^3 a_{ij} \alpha^j \otimes \alpha^i,
\]
so $\tau\ostar\varphi$ belongs to $\Lambda^{2,2}_1 = \Lambda^{2,2}_{\rm skw}$ if and only if $a_{ij}=-a_{ji}$ for each $i,j$ as well.
\end{example}

We can conclude that $\ostar$ by itself sends the decomposition to the dual decomposition.

\begin{proposition}\label{prop:ostardecomp}
  The isomorphism $\ostar$ sends $\Lambda^{p,q}_m$ to $\Lambda^{n-p,n-q}_{m^*}$, where $m^*=m+p-q$.
\end{proposition}

\begin{proof}
  Since $\tau$ is an involution that commutes with $\ostar$, we have $\ostar=(\tau\ostar)\tau$. By Proposition~\ref{prop:taudecomp}, $\tau$ sends $\Lambda^{p,q}_m$ to $\Lambda^{q,p}_{m^*}$. By Proposition~\ref{prop:tauostardecomp}, $\tau\ostar$ sends $\Lambda^{q,p}_{m^*}$ to $\Lambda^{n-p,n-q}_{m^*}$.
\end{proof}

\subsection{Relationship with other decompositions} \label{sec:Trdecomp}
The decomposition of $\Lambda^{p,q}$ in Proposition~\ref{prop:decomp} turns out to be equivalent to another decomposition that has been studied by other authors~\cite{kulkarni1972bianchi,labbi2005double}.   To explain this correspondence, suppose $\{e_i\}_i$ is orthonormal, and let $g = \sum_i e^i \otimes e^i$.  We will abuse notation and also use $g$ to denote the map $g \owedge {} : \Lambda^{p,q} \to \Lambda^{p+1,q+1}$, so that on simple tensors,
\[
g (\alpha \otimes \beta) = \sum_i (e^i \wedge \alpha) \otimes (e^i \wedge \beta).
\]	
The adjoint of this map is the trace operator $\Tr : \Lambda^{p+1,q+1} \to \Lambda^{p,q}$, which is given on simple tensors by
\[
\Tr (\alpha \otimes \beta) = \sum_i (e_i \contract \alpha) \otimes (e_i \contract \beta).
\]
\begin{remark}
We caution the reader that the word ``trace'' is used elsewhere in this paper to refer to the pullback under the inclusion of one manifold in another.  To distinguish between the trace operator above and the pullback, we use uppercase $\Tr$ for the former and lowercase $\tr$ for the latter.
\end{remark}
Below we will show that the maps $\Tr$ and $s$ are related to one another via conjugation by a single (as opposed to double) Hodge star operator
\[
\star_L(\alpha \otimes \beta) = (\star \alpha) \otimes \beta.
\]
We begin with a lemma whose proof can also be found in~\cite[Lemma 8.2.1]{hirani2003discrete}.
\begin{lemma} \label{lemma:contractwedge}
	For any multicovector $\alpha \in \Lambda^p$, we have
	\[
	e_i \contract \star \alpha = \star (\alpha \wedge e^i) = (-1)^p \star(e^i \wedge \alpha).
	\]
\end{lemma}
\begin{proof}
	Let $\alpha \in \Lambda^p$ and $\beta \in \Lambda^{n-p-1}$.  Then
	\begin{align*}
		\langle e_i \contract \star \alpha, \beta \rangle
		&= \langle \star \alpha, e^i \wedge \beta \rangle \\
		&= \alpha \wedge e^i \wedge \beta \\
		&= \langle \star(\alpha \wedge e^i), \beta \rangle.
	\end{align*}
	Since $\beta$ was arbitrary, the claim follows.
\end{proof}
\begin{proposition} \label{prop:Trstar}
	On $\Lambda^{p,q}$, we have
	\[
	\Tr {\star}_L  = (-1)^p {\star}_L s.
	\]
\end{proposition}
\begin{proof}
	It suffices to consider a simple tensor $\alpha \otimes \beta \in \Lambda^{p,q}$.   Using Lemma~\ref{lemma:contractwedge}, we compute
	\begin{align*}
		\Tr {\star}_L (\alpha \otimes \beta)
		&= \sum_i (e_i \contract \star \alpha) \otimes (e_i \contract \beta) \\
		&= \sum_i (-1)^p \star(e^i \wedge \alpha) \otimes (e_i \contract \beta) \\
		&= (-1)^p \star_L s(\alpha \otimes \beta).\qedhere
	\end{align*}
\end{proof}

\begin{definition}
For each nonnegative integer $m$, consider $\Tr^{m+1} : \Lambda^{p,q} \to \Lambda^{p-m-1,q-m-1}$ and $g^m : \Lambda^{p-m,q-m} \to \Lambda^{p,q}$.  We define
\[
\Sigma^{p,q}_m = \ker \Tr^{m+1} \cap \operatorname{im} g^m.
\]
\end{definition}
\begin{proposition}
The isomorphism $\star_L$ sends $\Lambda^{p,q}_m$ to $\Sigma^{n-p,q}_m$.
\end{proposition}
\begin{proof}
Recall from Proposition~\ref{prop:kerimdecomp} that
\[
\Lambda^{p,q}_m = \ker s^{m+1} \cap \operatorname{im} (s^*)^m.
\]
By Proposition~\ref{prop:Trstar}, a $(p,q)$-covector $\varphi$ belongs to $\ker s^{m+1}$ if and only if 
\[
0 = \star_L s^{m+1} \varphi = (-1)^{(m/2+p)(m+1)} \Tr^{m+1} \star_L \varphi, 
\]
i.e.~$\star_L \varphi$ belongs to $\ker \Tr^{m+1}$.  Likewise, $\varphi$ belongs to $\operatorname{im}(s^*)^m$ if and only if it is orthogonal to $\ker s^m = \star_L \ker \Tr^m$, i.e.~$\star_L \varphi$ is orthogonal to $\ker \Tr^m$.  Since the operators $g$ and $\Tr$ are adjoints, this is equivalent to the statement that $\star_L \varphi$ belongs to $\operatorname{im}g^m$.  We conclude that $\varphi$ belongs to $\ker s^{m+1} \cap \operatorname{im} (s^*)^m = \Lambda^{p,q}_m$ if and only if $\star_L \varphi$ belongs to $\ker \Tr^{m+1} \cap \operatorname{im} g^m = \Sigma^{n-p,q}_m$.
\end{proof}
The preceding correspondence between $\Lambda^{p,q}_m$ and $\Sigma^{n-p,q}_m$ implies the following decomposition that appears, for example, in~\cite[p. 185]{kulkarni1972bianchi} and~\cite[Equation (10)]{labbi2005double}.
\begin{proposition}
We have the decomposition
\[
\Lambda^{p,q} = \bigoplus_m \Sigma^{p,q}_m.
\]
The space $\Sigma^{p,q}_m$ is nonzero if and only if
\begin{equation} \label{Sigmanonzero}
\max\{0,p+q-n\} \le m \le \min\{p,q\},
\end{equation}
where $n = \dim V$.
\end{proposition}
\begin{proof}
This follows from Propositions~\ref{prop:decomp} and~\ref{prop:lpqmnonzero} and the isomorphism $\Sigma^{p,q}_m \simeq \Lambda^{n-p,q}_m$.
\end{proof}
\begin{example}
The space $\Lambda^{1,1}$ admits the decomposition
\begin{align*}
\Lambda^{1,1}
&= \Sigma^{1,1}_0 \oplus \Sigma^{1,1}_1 \\
&= (\ker \Tr^1 \cap \operatorname{im} g^0) \oplus (\ker \Tr^2 \cap \operatorname{im} g^1) \\
&= \ker \Tr \oplus \operatorname{im} g.
\end{align*}
If one identifies elements of $\Lambda^{1,1}$ with matrices as in Example~\ref{ex:tauostar2}, then this decomposition says that any $n \times n$ matrix can be written as a trace-free matrix plus a multiple of the identity.
\end{example}
\begin{example}
The space $\Lambda^{2,2}$ admits the decomposition
\begin{align*}
\Lambda^{2,2}
&= \Sigma^{2,2}_0 \oplus \Sigma^{1,1}_1 \oplus \Sigma^{2,2}_2 \\
&= (\ker \Tr^1 \cap \operatorname{im} g^0) \oplus (\ker \Tr^2 \cap \operatorname{im} g^1) \oplus (\ker \Tr^3 \cap \operatorname{im} g^2)  \\
&= \ker \Tr \oplus (\ker \Tr^2 \cap \operatorname{im} g^1) \oplus \operatorname{im} g^2.
\end{align*}
This decomposition is well known in Riemannian geometry: If one regards the Riemann curvature tensor $\mathrm{Riem}$ as an element of $\Lambda^{2,2}$, then one has~\cite[Proposition 7.28]{lee2018introduction}
\[
\mathrm{Riem} = \underbrace{\mathrm{Weyl}}_{\in \Sigma^{2,2}_0} + \underbrace{\frac{1}{n-2} g \owedge \left( \mathrm{Ric} - \frac{R}{n} g \right)}_{\in \Sigma^{2,2}_1} + \underbrace{\frac{R}{2n(n-1)} g \owedge g}_{\in \Sigma^{2,2}_2},
\]
where $\mathrm{Ric} = \Tr \mathrm{Riem}$ is the Ricci tensor, $R = \Tr \mathrm{Ric}$ is the scalar curvature, and
\[
\mathrm{Weyl} = \mathrm{Riem} - \frac{1}{n-2} g \owedge \left( \mathrm{Ric} - \frac{R}{2(n-1)} g \right)
\]
is the Weyl tensor.  We can also see from the inequality~\eqref{Sigmanonzero} that the Weyl tensor vanishes in dimension $n \le 3$ (since then $\Sigma^{2,2}_0=0$), and $\mathrm{Riem} = \frac{R}{4}g \owedge g$ in dimension $n=2$ (since then $\Sigma^{2,2}_0=\Sigma^{2,2}_1=0$).
\end{example}

  \begin{remark}
    Despite the equivalence of the decompositions $\Lambda^{p,q}=\bigoplus_m\Lambda^{p,q}_m$ and $\Lambda^{p,q}=\bigoplus_m\Sigma^{p,q}_m$, we remark that a key difference is the following: The decomposition $\Lambda^{p,q}=\bigoplus_m\Lambda^{p,q}_m$ is invariant under the action of the general linear group on $V$. In other words, if we consider a linear transformation acting on $V$, then the induced action on the multilinear forms $\Lambda^{p,q}$ will send the subspaces $\Lambda^{p,q}_m$ to themselves. In contrast, the decomposition $\Lambda^{p,q}=\bigoplus_m\Sigma^{p,q}_m$ relies on a choice of inner product on $V$, and hence is only invariant under orthogonal linear transformations.
  \end{remark}

\section{Double forms}\label{sec:doubleforms}
Now that we have understood double multicovectors, we move on to double multicovector \emph{fields}, also known as double forms. We first briefly discuss the implications of Section~\ref{sec:doublemulti} in the general manifold setting. Then, in Section~\ref{sec:doubleeuclidean}, we specialize to double forms on Euclidean space and discuss the left and right exterior derivative and Koszul operators. Finally, in Section~\ref{sec:polynomialdouble}, we specialize further to double forms with polynomial coefficients and prove an important result characterizing the kernel of the Koszul operators.

\begin{definition}
  Let $M$ be a smooth manifold. For $p+q=k$, let the space of $(p,q)$-\emph{forms} or \emph{double forms}, denoted $\Lambda^{p,q}(M)$, be the space of smooth covariant $k$-tensor fields on $M$ that are antisymmetric in the first $p$ indices and antisymmetric in the last $q$ indices. In other words, $\Lambda^{p,q}(M)$ is the space of smooth sections of the bundle $\bigwedge^pT^*M\otimes\bigwedge^qT^*M$.
\end{definition}

Note that, at each point $x\in M$, this bundle gives the vector space $\bigwedge^pT^*_xM\otimes\bigwedge^qT^*_xM$, so we just have the constructions from the previous subsections with $V=T_xM$. Consequently, the operators on double multicovectors from the previous section can be applied pointwise to yield operators on double forms.

\begin{definition}\label{def:operatorsM}
  We define the operators
  \begin{align*}
    s&\colon\Lambda^{p,q}(M)\to\Lambda^{p+1,q-1}(M),\\
    s^*&\colon\Lambda^{p,q}(M)\to\Lambda^{p-1,q+1}(M),\\
    \tau&\colon\Lambda^{p,q}(M)\to\Lambda^{q,p}(M),\\
    \ostar&\colon\Lambda^{p,q}(M)\to\Lambda^{n-p,n-q}(M),
  \end{align*}
  by applying the corresponding double multicovector operators pointwise. Here, $n=\dim M$, and $\ostar$ requires and depends on a Riemannian metric on $M$.
\end{definition}

The preceding formulas relating these operators on double multicovectors apply equally well to double forms, and, likewise, double forms have the same eigendecomposition.

\begin{proposition}
  We have the decomposition
  \begin{equation*}
    \Lambda^{p,q}(M)=\bigoplus_m\Lambda^{p,q}_m(M),
  \end{equation*}
  where $\max\{0,q-p\}\le m\le\min\{q,n-p\}$ and $\Lambda^{p,q}_m(M)$ is the space of eigenfunctions of $s^*s$ with eigenvalue $m(m+p-q+1)$.
\end{proposition}

Except for $\ostar$, which depends on a Riemannian metric, these operators commute with pullback by smooth maps.
\begin{proposition}\label{prop:pullbackdecomposition}
  Let $\Phi\colon M\to N$ be a smooth map between smooth manifolds. Then the pullback map $\Phi^*\colon\Lambda^{p,q}(N)\to\Lambda^{p,q}(M)$ commutes with $s$, $s^*$, and $\tau$, and $\Phi^*$ respects the decomposition, sending $\Lambda^{p,q}_m(N)$ to $\Lambda^{p,q}_m(M)$.
\end{proposition}

\begin{proof}
  Because the wedge product commutes with pullback, both $s\Phi^*$ and $\Phi^*s$, when applied to $(\alpha_1 \wedge \dots \wedge \alpha_p) \otimes (\beta_1 \wedge \dots \wedge \beta_q)$, are equal to
  \begin{equation*}
    \sum_{a=1}^q (-1)^{a-1} (\Phi^*\beta_a \wedge \Phi^*\alpha_1 \wedge \dots \wedge \Phi^*\alpha_p) \otimes (\Phi^*\beta_1 \wedge \dots \wedge \widehat{\Phi^*\beta_a} \wedge \dots \wedge \Phi^*\beta_q).
  \end{equation*}
  We can similarly show that $\Phi^*$ commutes with $s^*$ and $\tau$. Consequently, $\Phi^*$ commutes with $s^*s$, and so $\Phi^*$ respects the eigendecomposition of $s^*s$, with the same eigenvalues (possibly sending some eigenvectors to zero).
\end{proof}

If $M$ has a Riemannian metric (or simply a connection $\nabla$ on the tangent bundle), then we can define the exterior covariant derivative on $\Lambda^{p,q}(M)$ in two different ways, since we can view $(p,q)$-forms as $\Lambda^p$-valued $q$-forms or $\Lambda^q$-valued $p$-forms. However, we will only need these operators when $M$ is simply Euclidean space, so we instead present the definition in this specialized context.

\subsection{Double forms on Euclidean space}\label{sec:doubleeuclidean}
In this subsection, we will have $M$ be $\R^{n+1}$, with coordinates $(x^0,\dots,x^n)$. Note that the dimension here is $n+1$, which differs from the convention in Section~\ref{sec:doublemulti}. In this context, we will define natural operators $d_L$, $d_R$, $\kappa_L$, and $\kappa_R$, and discuss how they relate to the decomposition of $\Lambda^{p,q}$.  Just like the $s$ and $s^*$ operators, these operators play a fundamental role in the study of differential complexes involving double forms~\cite{kupferman2022saint,kupferman2023elliptic,kupferman2024double,arnold2021complexes,cap2024bgg,hu2025finite}, so we will point out links to relevant literature along the way.

\begin{notation}
  When there is no risk of confusion, we will let $\Lambda^k$ and $\Lambda^{p,q}$ denote $\Lambda^k(\R^{n+1})$ and $\Lambda^{p,q}(\R^{n+1})$, respectively. For a multi-index $I=(i_1,\dotsc,i_k)$, let $\dx^I=\dx^{i_1}\wedge\dots\wedge\dx^{i_k}\in\Lambda^k$, and let $\dx^{I,J}=\dx^I\otimes\dx^J\in\Lambda^{p,q}$.
\end{notation}

\begin{definition}
  We define natural operators
  \begin{equation*}
    \D_L\colon\Lambda^{p,q}\to\Lambda^{p+1,q},\qquad\D_R\colon\Lambda^{p,q}\to\Lambda^{p,q+1},
  \end{equation*}
  by
  \begin{equation*}
    \D_L(f\dx^{I,J})=\left(\D f\wedge\dx^I\right)\otimes\dx^J,\qquad\D_R(f\dx^{I,J})=\dx^I\otimes\left(\D f\wedge\dx^J\right).
  \end{equation*}
  Here, $f$ is an arbitrary smooth function, and we extend these definitions by linearity.
\end{definition}

\begin{proposition}\label{prop:ddcommute}
  The operators $\D_L$ and $\D_R$ commute.
\end{proposition}

\begin{proof}
  Applying both $d_Ld_R$ and $d_Rd_L$ to $f\dx^{I,J}$, by the symmetry of the Hessian, we obtain
  \begin{equation*}
    \sum_{i,j}\frac{\partial f}{\partial x^i\partial x^j}\left(\dx^i\wedge\dx^I\right)\otimes\left(\dx^j\wedge\dx^J\right).\qedhere
  \end{equation*}
\end{proof}

\begin{definition}
  The \emph{tautological vector field} is
  \begin{equation*}
    X_{\id}:=\sum_{i=0}^nx^i\pp{x^i}.
  \end{equation*}
  If $\alpha$ is a $k$-form, we let the \emph{Koszul operator} $\kappa$ denote contraction with $X_{\id}$, that is,
  \begin{equation*}
    \kappa\alpha:=X_{\id}\contract\alpha.
  \end{equation*}
  For a double form, we can apply $\kappa$ to either the left factor or the right factor; we denote these operators by $\kappa_L$ and $\kappa_R$, respectively. Namely, we have,
  \begin{align*}
    \kappa_L\colon\Lambda^{p,q}&\to\Lambda^{p-1,q},&\kappa_R\colon\Lambda^{p,q}&\to\Lambda^{p,q-1},\\
    \alpha\otimes\beta&\mapsto\left(\kappa\alpha\right)\otimes\beta,&\alpha\otimes\beta&\mapsto\alpha\otimes\left(\kappa\beta\right).
  \end{align*}
\end{definition}

\begin{proposition}
  The operators $\kappa_L$ and $\kappa_R$ commute.
\end{proposition}

\begin{proof}
  Applying both $\kappa_L\kappa_R$ and $\kappa_R\kappa_L$ to $\alpha\otimes\beta$, we obtain
  \begin{equation*}
    \left(X_{\id}\contract\alpha\right)\otimes\left(X_{\id}\contract\beta\right).\qedhere
  \end{equation*}
\end{proof}

There are also several nontrivial commutator relationships between our operators. The first proposition that we prove below is a special case of a more general relationship that appears in~\cite[p.~49]{kupferman2023elliptic} and~\cite[p. 259]{gray1970some}.  See also~\cite[Equation (2.13)]{hu2025finite} and~\cite[Equation (3)]{cap2024bgg}.

\begin{proposition} \label{prop:kappas}
  We have
  \begin{align*}
    \kappa_Ls+s\kappa_L&=\kappa_R,&\kappa_Rs^*+s^*\kappa_R&=\kappa_L\\
    \kappa_Ls^*+s^*\kappa_L&=0,&\kappa_Rs+s\kappa_R&=0.
  \end{align*}
\end{proposition}

\begin{proof}
  Using Proposition~\ref{prop:wedgecontract}, we have
  \begin{align*}
    \kappa_Ls(\alpha\otimes\beta)&=\sum_{i,j}x^j\left(\pp{x^j}\contract(dx^i\wedge\alpha)\right)\otimes\left(\pp{x^i}\contract\beta\right),\\
    s\kappa_L(\alpha\otimes\beta)&=\sum_{i,j}x^j\left(dx^i\wedge\left(\pp{x^j}\contract\alpha\right)\right)\otimes\left(\pp{x^i}\contract\beta\right).
  \end{align*}
  Adding, we obtain
  \begin{equation*}
    \begin{split}
      (\kappa_Ls+s\kappa_L)(\alpha\otimes\beta)&=\sum_{i,j}x^j\left(\pp[x^i]{x^j}\alpha\right)\otimes\left(\pp{x^i}\contract\beta\right)\\
      &=\sum_ix^i\alpha\otimes\left(\pp{x^i}\contract\beta\right)\\
      &=\alpha\otimes\left(X_{\id}\contract\beta\right),
    \end{split}
  \end{equation*}
  as desired.

  Meanwhile,
  \begin{align*}
    \kappa_Ls^*(\alpha\otimes\beta)&=\sum_{i,j}x^j\left(\pp{x^j}\contract\left(\pp{x^i}\contract\alpha\right)\right)\otimes\left(\dx^i\wedge\beta\right),\\
    s^*\kappa_L(\alpha\otimes\beta)&=\sum_{i,j}x^j\left(\pp{x^i}\contract\left(\pp{x^j}\contract\alpha\right)\right)\otimes\left(\dx^i\wedge\beta\right).
  \end{align*}
  The sum is zero by antisymmetry of contraction.

  The remaining two claims follows by symmetry.
\end{proof}

\begin{example} \label{ex:kappas}
Let $\varphi = \sum_{i,j} a_{ij} dx^i \otimes dx^j \in \Lambda^{1,1}$, where each $a_{ij}$ is a smooth function.  Then 
\begin{align*}
\kappa_L s \varphi 
&= \kappa_L \sum_{i,j} a_{ij} (dx^j \wedge dx^i) \otimes 1 \\
&= \sum_{i,j} a_{ij} (x^j dx^i - x^i dx^j) \otimes 1 
\end{align*}
and
\begin{align*}
s \kappa_L \varphi
&= s \sum_{i,j} a_{ij} x^i \otimes dx^j \\
&= \sum_{i,j} a_{ij} x^i dx^j \otimes 1,
\end{align*}
so
\[
(\kappa_L s + s \kappa_L) \varphi = \sum_{i,j} a_{ij} x^j dx^i \otimes 1 = \kappa_R \varphi.
\]
\end{example}

We can now show that the operator $\kappa_L\kappa_R=\kappa_R\kappa_L$ respects the decomposition $\Lambda^{p,q}=\bigoplus_m\Lambda^{p,q}_m$.

\begin{proposition} \label{prop:kks}
  The operator $\kappa_L\kappa_R=\kappa_R\kappa_L$ commutes with both $s$ and $s^*$.
\end{proposition}

\begin{proof}
  We have
  \begin{equation*}
    \kappa_L\kappa_Rs=-\kappa_Ls\kappa_R=(s\kappa_L-\kappa_R)\kappa_R=s\kappa_L\kappa_R
  \end{equation*}
  because $\kappa_R^2=0$. The claim for $s^*$ follows by symmetry.
\end{proof}

\begin{proposition}\label{prop:kkdecomposition}
  The operator $\kappa_L\kappa_R\colon\Lambda^{p,q}\to\Lambda^{p-1,q-1}$ sends $\Lambda^{p,q}_m$ to $\Lambda^{p-1,q-1}_m$.
\end{proposition}

\begin{proof}
  By the above proposition, $\kappa_L\kappa_R$ commutes with $s^*s$, so if $\phi$ is an eigenvalue of $s^*s$ with eigenvalue $m(m+p-q+1)$, we have
  \begin{equation*}
    s^*s(\kappa_L\kappa_R\phi)=m(m+p-q+1)(\kappa_L\kappa_R\phi).
  \end{equation*}
  Thus, $\kappa_L\kappa_R\phi$ is either zero or an eigenvalue of $s^*s$ with eigenvalue $m(m+p-q+1)$. If $\kappa_L\kappa_R\phi$ is zero, then it is in $\Lambda^{p-1,q-1}_m$ tautologically; if it is nonzero, then it is in $\Lambda^{p-1,q-1}_m$ because $m(m+(p-1)-(q-1)+1)=m(m+p-q+1)$.
\end{proof}

\begin{example} \label{ex:kkdecomposition}
Let 
\[
\varphi = \sum_{i,j=1}^3 a_{ij} \, {\star} dx^i \otimes \star dx^j \in \Lambda^{2,2}(\mathbb{R}^3),
\]
where each $a_{ij}$ is a smooth function.  Notice that $\kappa {\star} dx^i = \beta^i$ for each $i$, where $\beta^1 = x^2 dx^3 - x^3 dx^2$, $\beta^2 = x^3 dx^1 - x^1 dx^3$, and $\beta^3 = x^1 dx^2 - x^2 dx^1$.  Thus,
\[
\kappa_L \kappa_R \varphi = \sum_{i,j=1}^3 a_{ij} \beta^i \otimes \beta^j.
\]
This double form belongs to $\Lambda^{1,1}_0(\mathbb{R}^3) = \Lambda^{1,1}_{\rm sym}(\mathbb{R}^3)$ whenever $\varphi \in \Lambda^{2,2}_0(\mathbb{R}^3) = \Lambda^{2,2}_{\rm sym}(\mathbb{R}^3)$, and it belongs to $\Lambda^{1,1}_1(\mathbb{R}^3) = \Lambda^{1,1}_{\rm skw}(\mathbb{R}^3)$ whenever $\varphi \in \Lambda^{2,2}_1(\mathbb{R}^3) = \Lambda^{2,2}_{\rm skw}(\mathbb{R}^3)$.
\end{example}

We can prove an analogous result for the operator $d_Ld_R$. First, we prove some relations between $d_L,d_R,s$ and $s^*$.  These relations also appear in~\cite[p. 51]{kupferman2023elliptic}.  One of them, namely $d_Ls+sd_L=0$, plays an important role in the BGG framework; see~\cite[Equations (22) and (40)] {arnold2021complexes},~\cite[Equation (5)]{cap2024bgg}, and~\cite[Equation (2.11)]{hu2025finite}.

\begin{proposition} \label{prop:dssd}
  We have
  \begin{align*}
    d_Ls+sd_L&=0,&d_Rs^*+s^*d_R&=0\\
    d_Ls^*+s^*d_L&=d_R,&d_Rs+sd_R&=d_L.
  \end{align*}
\end{proposition}

\begin{proof}
  Using Proposition~\ref{prop:wedgecontract}, we have
  \begin{align*}
    d_Ls(f\dx^{I,J})&=\sum_{i,j}\pp[f]{x^j}(\dx^j\wedge\dx^i\wedge\dx^I)\otimes\left(\pp{x^i}\contract\dx^J\right),\\
    sd_L(f\dx^{I,J})&=\sum_{i,j}\pp[f]{x^j}(\dx^i\wedge\dx^j\wedge\dx^I)\otimes\left(\pp{x^i}\contract\dx^J\right).
  \end{align*}
  The sum is zero by the antisymmetry of wedge.

  Meanwhile,
  \begin{align*}
    d_Ls^*(f\dx^{I,J})&=\sum_{i,j}\pp[f]{x^j}\left(\dx^j\wedge\left(\pp{x^i}\contract\dx^I\right)\right)\otimes(\dx^i\wedge\dx^J),\\
    s^*d_L(f\dx^{I,J})&=\sum_{i,j}\pp[f]{x^j}\left(\pp{x^i}\contract(\dx^j\wedge\dx^I)\right)\otimes(\dx^i\wedge\dx^J).
  \end{align*}
  Adding, we obtain
  \begin{equation*}
    \begin{split}
      (d_Ls^*+s^*d_L)(f\dx^{I,J})&=\sum_{i,j}\pp[f]{x^j}\left(\pp[x^j]{x^i}\dx^I\right)\otimes(dx^i\wedge\dx^J)\\
      &=\sum_i\pp[f]{x^i}\dx^I\otimes(\dx^i\wedge\dx^J)\\
      &=d_R(f\dx^{I,J}).
    \end{split}
  \end{equation*}
  The remaining claims follow by symmetry.
\end{proof}

The following proposition also appears in a more general context in \cite[Proposition~3.10]{kupferman2024double}.

\begin{proposition}\label{prop:ddscommute}
  The operator $d_Ld_R=d_Rd_L$ commutes with both $s$ and $s^*$.
\end{proposition}

\begin{proof}
  We have
  \begin{equation*}
    d_Ld_Rs=d_L(d_L-sd_R)=-d_Lsd_R=sd_Ld_R.
  \end{equation*}
  The claim for $s^*$ follows by symmetry.
\end{proof}

\begin{proposition}\label{prop:dddecomposition}
  The operator $d_Ld_R\colon\Lambda^{p,q}\to\Lambda^{p+1,q+1}$ sends $\Lambda^{p,q}_m$ to $\Lambda^{p+1,q+1}_m$.
\end{proposition}

\begin{proof}
  The proof is analogous to the proof of Proposition~\ref{prop:kkdecomposition}.
\end{proof}

\begin{example} \label{ex:dddecomposition}
Let
\[
\varphi = \sum_{i,j=1}^3 a_{ij} dx^i \otimes dx^j \in \Lambda^{1,1}(\mathbb{R}^3),
\]
where each $a_{ij}$ is a smooth function, and let
\begin{align*}
b_{1j} &= \frac{\partial a_{3j}}{\partial x^2} -  \frac{\partial a_{2j}}{\partial x^3}, \\
b_{2j} &= \frac{\partial a_{1j}}{\partial x^3} -  \frac{\partial a_{3j}}{\partial x^1}, \\
b_{3j} &= \frac{\partial a_{2j}}{\partial x^1} -  \frac{\partial a_{1j}}{\partial x^2}
\end{align*}  
for each $j=1,2,3$.
Using the fact that
\begin{align*}
d\left( \sum_{i=1}^3 a_{ij} dx^i \right) 
&= d(a_{1j} dx^1 + a_{2j} dx^2 + a_{3j} dx^3) \\
&= b_{1j} dx^2 \wedge dx^3 + b_{2j} dx^3 \wedge dx^1 + b_{3j} dx^1 \wedge dx^2 \\
&= \sum_{i=1}^3 b_{ij} \, {\star} dx^i,
\end{align*}
we compute
\[
d_L \varphi = \sum_{i,j=1}^3 b_{ij} \, {\star} dx^i \otimes dx^j = \sum_{i,j=1}^3  {\star} dx^i \otimes (b_{ij} dx^j).
\]
Applying $d_R$, we get
\[
d_R d_L \varphi = \sum_{i,j=1}^3 {\star} dx^i \otimes (c_{ij} \, {\star} dx^j) = \sum_{i,j=1}^3 c_{ij} \, {\star} dx^i \otimes \star dx^j,
\]
where
\begin{align*}
c_{i1} &= \frac{\partial b_{i3}}{\partial x^2} -  \frac{\partial b_{i2}}{\partial x^3}, \\
c_{i2} &= \frac{\partial b_{i1}}{\partial x^3} -  \frac{\partial b_{i3}}{\partial x^1}, \\
c_{i3} &= \frac{\partial b_{i2}}{\partial x^1} -  \frac{\partial b_{i1}}{\partial x^2}.
\end{align*}
This calculation shows that the map $d_L d_R = d_R d_L$ has the effect of sending the matrix $A = [a_{ij}]_{i,j=1}^3$ to the matrix $C = [c_{ij}]_{i,j=1}^3 = \top \operatorname{curl} \top \operatorname{curl} A$, where $\operatorname{curl}$ acts column-wise and $\top$ denotes transposition.  In particular, if $\varphi$ belongs to $\Lambda^{1,1}_0(\mathbb{R}^3) = \Lambda^{1,1}_{\rm sym}(\mathbb{R}^3)$, then $A$ is a symmetric matrix, and one can check that $C$ is then symmetric as well.  Hence, $d_L d_R \varphi \in \Lambda^{2,2}_0(\mathbb{R}^3) = \Lambda^{2,2}_{\rm sym}(\mathbb{R}^3)$ whenever $\varphi \in \Lambda^{1,1}_0(\mathbb{R}^3) = \Lambda^{1,1}_{\rm sym}(\mathbb{R}^3)$.

Similarly, $C = \top \operatorname{curl} \top \operatorname{curl} A$ is skew-symmetric whenever $A$ is skew-symmetric, so $d_L d_R \varphi \in \Lambda^{2,2}_1(\mathbb{R}^3) = \Lambda^{2,2}_{\rm skw}(\mathbb{R}^3)$ whenever $\varphi \in \Lambda^{1,1}_1(\mathbb{R}^3) = \Lambda^{1,1}_{\rm skw}(\mathbb{R}^3)$.
\end{example}

Finally, we have commutation relations between the exterior derivatives and the Koszul operators. This result also appears in \cite[Lemma 3]{arnold2021complexes} and~\cite[Equations (2.10) and (2.12)]{hu2025finite}, and it is used as a way of defining $s$ in~\cite[Section 2]{cap2024bgg}.

\begin{proposition} \label{prop:dkkd}
  We have
  \begin{equation*}
    d_L\kappa_R-\kappa_Rd_L=s,\qquad
    d_R\kappa_L-\kappa_Ld_R=s^*.
  \end{equation*}
\end{proposition}

\begin{proof}
  We have
  \begin{align*}
    \begin{split}
      d_L\kappa_R\left(f\dx^{I,J}\right)&=d_L\sum_ix^if\dx^I\otimes\left(\pp{x^i}\contract\dx^J\right)\\
      &=\sum_i\left(\left(f\dx^i+x^i\D f\right)\wedge\dx^I\right)\otimes\left(\pp{x^i}\contract\dx^I\right).
    \end{split}\\
    \kappa_Rd_L\left(f\dx^{I,J}\right)&=\sum_i\left(x^i\D f\wedge\dx^I\right)\otimes\left(\pp{x^i}\contract\dx^J\right).
  \end{align*}
  Subtracting, we obtain $s\left(f\dx^{I,J}\right)$ using Proposition~\ref{prop:wedgecontract}. The second equation follows by symmetry.
\end{proof}

\subsection{Polynomial double forms on Euclidean space}\label{sec:polynomialdouble}
We now specialize to double forms with polynomial coefficients. The operators we have discussed all send polynomial double forms to polynomial double forms, with the Koszul operators raising polynomial degree, the exterior derivative operators lowering polynomial degree, and the $s$ and $s^*$ operators keeping the polynomial degree the same. As we will see, in order to prove our results, we will need to understand the image of $\kappa_L\kappa_R$. We begin with definitions.
\begin{definition}
  Let $\cH_r\Lambda^{p,q}(\R^{n+1})$ or simply $\cH_r\Lambda^{p,q}$ denote the space of double forms with homogeneous polynomial coefficients of degree $r$. In other words, $\cH_r\Lambda^{p,q}$ is spanned by $f\dx^{I,J}$, where $f$ is a homogeneous polynomial of degree $r$. We will define the decomposition component $\cH_r\Lambda^{p,q}_m$ similarly, and we will also occasionally need analogously defined spaces $\cH_r\Lambda^k$ of $k$-forms, as well as the space $\cH_r$ of scalar fields, which is simply the space of homogeneous polynomials of degree $r$.
\end{definition}

Observe that $\kappa_L\kappa_R\colon\cH_{r-2}\Lambda^{p+1,q+1}\to\cH_r\Lambda^{p,q}$. The image of this map will be important enough to merit a definition.

\begin{definition}
  Let
  \begin{equation*}
    \cH_r^-\Lambda^{p,q}:=\kappa_L\kappa_R\cH_{r-2}\Lambda^{p+1,q+1}.
  \end{equation*}
  We likewise let
  \begin{equation*}
    \cH_r^-\Lambda^{p,q}_m:=\kappa_L\kappa_R\cH_{r-2}\Lambda^{p+1,q+1}_m.
  \end{equation*}
\end{definition}
Note that $\cH_r^-\Lambda^{p,q}=0$ if $r<2$. As the notation suggests, $\cH_r^-\Lambda^{p,q}_m$ is a subspace of $\cH_r\Lambda^{p,q}_m$ by Proposition~\ref{prop:kkdecomposition}. More specifically, we have the following.

\begin{proposition}\label{prop:minusm}
  We have
  \begin{equation*}
    \cH_r^-\Lambda^{p,q}_m=\cH_r^-\Lambda^{p,q}\cap\Lambda^{p,q}_m.
  \end{equation*}
\end{proposition}

\begin{proof}
  If $\phi\in\kappa_L\kappa_R\cH_{r-2}\Lambda^{p+1,q+1}_m$, then it is in $\cH_r^-\Lambda^{p,q}$ by definition and in $\Lambda^{p,q}_m$ by Proposition~\ref{prop:kkdecomposition}.

  Conversely, assume that $\phi\in\cH_r^-\Lambda^{p,q}\cap\Lambda^{p,q}_m$. By definition, $\phi=\kappa_L\kappa_R\psi$ for some $\psi\in\cH_{r-2}\Lambda^{p+1,q+1}$, but $\psi$ might not be in the decomposition summand $\Lambda^{p+1,q+1}_m$. However, we can decompose $\psi=\sum_{m'}\psi_{m'}$ where each $\psi_{m'}\in\Lambda^{p+1,q+1}_{m'}$. The polynomial coefficients are unaffected by the decomposition, so, in fact, $\psi_{m'}\in\cH_{r-2}\Lambda^{p+1,q+1}_{m'}$. Letting $\phi_{m'}=\kappa_L\kappa_R\psi_{m'}$, we have that $\phi=\sum_{m'}\phi_{m'}$. By Proposition~\ref{prop:kkdecomposition}, $\phi_{m'}\in\Lambda^{p,q}_{m'}$. Since $\phi\in\Lambda^{p,q}_m$, we conclude that $\phi_{m'}=0$ unless $m'=m$, so $\phi=\phi_m=\kappa_L\kappa_R\psi_m$, so $\phi\in\cH_r^-\Lambda^{p,q}_m$ by definition.
\end{proof}

Since $\kappa_L$ commutes with $\kappa_R$ and $\kappa_L^2=\kappa_R^2=0$, we see that anything in $\cH_r^-\Lambda^{p,q}$ is in the kernel of both $\kappa_L$ and $\kappa_R$. Through the next few propositions, we will see that this condition almost characterizes $\cH_r^-\Lambda^{p,q}$.

\begin{proposition}\label{prop:cartan}
  On $\cH_r\Lambda^{p,q}$, we have
  \begin{equation*}
    d_L\kappa_L+\kappa_Ld_L=r+p,\qquad d_R\kappa_R+\kappa_Rd_R=r+q.
  \end{equation*}
\end{proposition}

\begin{proof}
  Checking on a basis and applying Cartan's formula, we have
  \begin{equation*}
    \begin{split}
      (d_L\kappa_L+\kappa_Ld_L)(f\dx^{I,J})&=\left((d\kappa+\kappa d)\left(f\dx^I\right)\right)\otimes\dx^J\\
      &=\left(\Lie_{X_{\id}}(f\dx^I)\right)\otimes\dx^J\\
      &=\left((r+p)(f\dx^I)\right)\otimes\dx^J.
    \end{split}
  \end{equation*}
  In the last step, we used $\Lie_{X_{\id}}x^i=x^i$ and hence $\Lie_{X_{\id}}\dx^i=\dx^i$, so, using the Leibniz rule, the Lie derivative applied to a differential form with homogeneous polynomial coefficients simply multiplies the form by the total degree, that is, the sum of the polynomial degree and the form degree.

  The claim for the operators on the right factor is analogous.
\end{proof}

\begin{proposition}\label{prop:kkdd}
  If $\phi\in\cH_r\Lambda^{p,q}_m$ and $\kappa_L\phi=\kappa_R\phi=0$, then
  \begin{equation*}
    \begin{split}
      \kappa_L\kappa_Rd_Ld_R\phi&=\left((r+p)(r+q-1)-m(m+p-q+1)\right)\phi\\
      &=(r+p+m)(r+q-m-1)\phi.
    \end{split}
  \end{equation*}
  \begin{proof}
    The idea is to use the commutation relations to move the $\kappa_L$ and $\kappa_R$ operators to the right to get zero. We compute
    \begin{equation*}
      \begin{split}
        \kappa_L\kappa_Rd_Ld_R\phi&=\kappa_R\kappa_Ld_Rd_L\phi\\
        &=(\kappa_Rd_R\kappa_Ld_L-\kappa_Rs^*d_L)\phi\\
        &=((\kappa_Rd_R(r+p)-\kappa_Rd_Rd_L\kappa_L)-(\kappa_Ld_L-s^*\kappa_Rd_L))\phi\\
        &=(((r+p)(r+q)-(r+p)d_R\kappa_R-0)\\
        &\qquad{}-(((r+p)-d_L\kappa_L)-(s^*d_L\kappa_R-s^*s)))\phi\\
        &=((r+p)(r+q)-(r+p)-s^*s)\phi\\
        &=((r+p)(r+q-1)-m(m+p-q+1))\phi\\
        &=(r+p+m)(r+q-m-1)\phi.\qedhere
      \end{split}
    \end{equation*}
  \end{proof}
\end{proposition}

An alternative derivation of this constant via representation theory is given in Appendix~\ref{sec:rep}. We are now ready to understand the relationship between $\ker\kappa_L\cap\ker\kappa_R$ and $\cH_r^-\Lambda^{p,q}$.

\begin{proposition}\label{prop:kerkoszul}
  Let $\phi$ be a nonzero element of $\cH_r\Lambda^{p,q}_m$. We have that $\kappa_L\phi=\kappa_R\phi=0$ if and only if exactly one of the following holds:
  \begin{itemize}
  \item $r=p=q=m=0$, so $\phi$ is a constant scalar field.
  \item $r=1$, $m=q$, and $\phi=i^{p,q}\kappa\psi$ for some $\psi\in\cH_{r-1}\Lambda^{k+1}$, where $k=p+q$ and $i^{p,q}$ is defined in Definition~\ref{def:ipq}.
  \item $r\ge2$, and $\phi\in\cH_r^-\Lambda^{p,q}_m$.
  \end{itemize}
  
  In the last case, we have
  \begin{equation*}
    \phi=\kappa_L\kappa_R\left(C^{-1}d_Ld_R\phi\right),
  \end{equation*}
  where
  \begin{equation*}
    C=(r+p+m)(r+q-m-1).
  \end{equation*}
\end{proposition}

\begin{proof}
  It is easy to check that, in any of these three cases, $\kappa_L\phi=\kappa_R\phi=0$. In the first case, $\phi$ is a $(0,0)$-form, so $\kappa_L\phi=\kappa_R\phi=0$. In the second case, it is easy to check from the definition of $i^{p,q}$ that $\kappa_Li^{p,q}=i^{p-1,q}\kappa$ and $\kappa_Ri^{p,q}=(-1)^pi^{p,q-1}\kappa$, so $\kappa_L\phi=i^{p-1,q}\kappa^2\psi=0$ and $\kappa_R\phi=(-1)^pi^{p,q-1}\kappa^2\psi=0$. Finally, in the third case, by definition, $\phi=\kappa_L\kappa_R\psi$ for some $\psi$, so $\phi$ is in the kernel of $\kappa_L$ and $\kappa_R$ because the two operators commute and square to zero.
  
  Assume now that $\phi$ is in the kernel of both $\kappa_L$ and $\kappa_R$. We must prove that we are in one of the three cases. If $r=0$, then $\phi$ is constant, and so $d_L\phi=d_R\phi=0$. Along with the assumption that $\kappa_L\phi=\kappa_R\phi=0$, Proposition~\ref{prop:cartan} tells us that $r+p=r+q=0$, from which we conclude that $p=q=r=0$, so $\cH_r\Lambda^{p,q}$ is simply the space of constant scalar fields. We also have $m=0$ since $0\le m\le q$.

  Assume henceforth that $r\ge1$. Since $r\ge1$, the factor $r+p+m$ of $C$ must be positive. Recall that, because $\Lambda^{p,q}_m$ is nonempty, we have $m\le q$. So, the second factor $r+q-m-1$ is positive except when $r=1$ and $m=q$. So, apart from the case $r=1$ and $m=q$, we have $C>0$.
  
  If $C>0$, then Proposition~\ref{prop:kkdd} tells us that
  \begin{equation*}
    \phi=\kappa_L\kappa_R\left(C^{-1}d_Ld_R\phi\right).
  \end{equation*}
  Since $d_L$ and $d_R$ lower polynomial degree by one, we have that $C^{-1}d_Ld_R\phi\in\cH_{r-2}\Lambda^{p+1,q+1}$, so $\phi\in\cH_r^-\Lambda^{p,q}$, as desired. In particular, $r\ge2$.

  So then it remains to consider the case $r=1$. In this case, $d_Ld_R\phi=0$ because $d_Ld_R$ lowers polynomial degree by two, so $C=0$ by Proposition~\ref{prop:kkdd}. As discussed, $C=0$ implies $r=1$ and $m=q$. Since $m=q$, by Proposition~\ref{prop:imageipq}, we have that $\phi=i^{p,q}\phi'$ for some $k$-form $\phi'$. Since $\phi$ and $\phi'$ are equal as $k$-tensors, $\phi'$ likewise has homogeneous polynomial coefficients of degree $r$. We claim that $\kappa\phi'=0$. This claim is trivial if $k=0$. Otherwise, $p\ge1$ or $q\ge1$. If $p\ge1$, then we use $0=\kappa_L\phi=i^{p-1,q}\kappa\phi'$, which implies that $\kappa\phi'=0$ because $i^{p-1,q}$ is an inclusion. If $q\ge1$, we reason similarly using $\kappa_R$. By Cartan's formula, we have $(d\kappa+\kappa d)\phi'=(r+k)\phi'$, so, using $r\ge1$ and $\kappa\phi'=0$, we have $\phi'=\kappa\psi$, where $\psi=(r+k)^{-1}\D\phi'$. Since $d$ lowers polynomial degree, we have that $\psi\in\cH_{r-1}\Lambda^{k+1}$, as desired.
\end{proof}

\section{Extending double forms on the simplex} \label{sec:extension}
For it to be possible to construct geometrically decomposable finite element spaces of double forms, a key requirement is that we be able to extend a double form with vanishing trace on the standard simplex $T^n$ to a double form on $\R^{n+1}$ with vanishing trace on the coordinate hyperplanes \cite{berchenko2026extension}. As we will see, doing so is possible except when $r=0$ and $m=q$.

We begin with definitions. Then, in Section~\ref{sec:simplexspherehodge}, we review some results from \cite{berchenkokogan2021duality} regarding the coordinate transformation $\lambda_i=u_i^2$ between the simplex $T^n$ and the sphere $S^n$. As we will see, the extension problem becomes easier after passing to the sphere. In Section~\ref{sec:overview}, we give a high-level overview of the extension construction and provide examples. In Sections~\ref{sec:pullback} and \ref{sec:doublehodgestar}, we delve into the details, adapting the results in \cite{berchenkokogan2021duality} to the double form setting, and setting up everything we need to prove the extension theorem in Section~\ref{sec:extensionconstruction}.

Much of the content in this section is important for proving the aforementioned extension theorem but not crucial for understanding the remainder of the paper.  Readers who wish to delve into fewer details and reach Section~\ref{sec:finite} more quickly may consider reading Definitions~\ref{def:simplex}-\ref{def:extension} and skimming the statements of Theorem~\ref{thm:extension} and Proposition~\ref{prop:noextension}.  

\begin{definition} \label{def:simplex}
  Let $T^n$ denote the standard simplex in $\R^{n+1}$. Specifically,
  \begin{equation*}
    T^n=\{(\lambda_0,\dots,\lambda_n)\mid\lambda_i\ge0,\sum_i\lambda_i=1\}.
  \end{equation*}
\end{definition}

\begin{definition} \label{def:Pr}
  Let $\cP_r(T^n)$ denote the space of polynomials on $T^n$ of degree at most $r$. We define the spaces $\cP_r\Lambda^k(T^n)$, $\cP_r\Lambda^{p,q}(T^n)$, and $\cP_r\Lambda^{p,q}_m(T^n)$ to be the corresponding spaces of forms or double forms with polynomial coefficients of degree at most $r$.
\end{definition}

\begin{definition} \label{def:oH}
  We have a natural inclusion of the boundary $\partial T^n\hookrightarrow T^n$. We say that a form or double form has \emph{vanishing trace} if it vanishes when pulled back under this inclusion, or, equivalently, that the tensor vanishes at $\partial T^n$ on vectors tangent to $\partial T^n$. We let $\oP_r\Lambda^k(T^n)$, $\oP_r\Lambda^{p,q}(T^n)$, and $\oP_r\Lambda^{p,q}_m(T^n)$ be the vanishing trace subspaces of the corresponding space.
\end{definition}

Note that the boundary of $T^n$ is the set of points $(\lambda_0,\dotsc,\lambda_n)$ in $T^n$ such that $\lambda_i=0$ for some $i$. This observation motivates a corresponding definition for forms on $\R^{n+1}$.

\begin{definition} \label{def:oP}
  For each $i$, we have a natural inclusion of the coordinate hyperplanes $\{\lambda_i=0\}\hookrightarrow\R^{n+1}$. We say that a form or double form on $\R^{n+1}$ has \emph{vanishing trace} if it vanishes when pulled back under this inclusion for all $i$. We let $\oH_r\Lambda^k(\R^{n+1})$, $\oH_r\Lambda^{p,q}(\R^{n+1})$, and $\oH_r\Lambda^{p,q}_m(\R^{n+1})$ be the vanishing trace subspaces of the corresponding space.
\end{definition}

Pulling back via the inclusion $T^n\hookrightarrow\R^{n+1}$, we can restrict a double form on $\R^{n+1}$ to a double form on $T^n$. Extension is the inverse of this operation.

\begin{definition} \label{def:extension}
  We say that a double form $\phi$ on $\R^{n+1}$ is an \emph{extension} of a double form $\bar\phi$ on $T^n$ if $\bar\phi$ is the pullback of $\phi$ via the inclusion $T^n\hookrightarrow\R^{n+1}$.
\end{definition}

Without the vanishing trace condition, extension is easy.

\begin{proposition}
  Every form in $\cP_r\Lambda^{p,q}_m(T^n)$ can be extended to a form in $\cH_r\Lambda^{p,q}_m(\R^{n+1})$.
\end{proposition}

\begin{proof}
  Observe that $\cH_r\Lambda^{p,q}_m(\R^{n+1})=\cH_r(\R^{n+1})\otimes\bigwedge^{p,q}_mV^*$, where $V=T_x\R^{n+1}$. Note that $V$ is itself just $\R^{n+1}$, and hence independent of $x$, but we use the notation $V$ to maintain the distinction between $\R^{n+1}$ as a vector space and $\R^{n+1}$ as a manifold. Likewise, $\cP_r\Lambda^{p,q}_m(T^n)=\cP_r(T^n)\otimes\bigwedge^{p,q}_mH^*$, where $H=T_xT^n$, a hyperplane of $V$. As a result, we can prove the proposition by proving two independent claims: The first claim is that polynomials on $T^n$ can be extended to \emph{homogeneous} polynomials on $\R^{n+1}$. The second claim is that the vector space map $\bigwedge^{p,q}_mV^*\to\bigwedge^{p,q}_mH^*$ is surjective.

  The extension of polynomials is the standard homogenization procedure. Given a polynomial $\bar f\in\cP_r(T^n)$, we can write it as a sum of monomials in the variables $\lambda_1,\dots,\lambda_n$ of degrees varying from $0$ to $r$. We obtain $f\in\cH_r(\R^{n+1})$ by multiplying each term by an appropriate power of $\lambda_0+\dots+\lambda_n$ so that the resulting term has degree exactly $r$. Since $\lambda_0+\dots+\lambda_n=1$ on $T^n$, the polynomial $f$ has the same values on $T^n$ as $\bar f$.

  For the linear algebra problem, since $H\hookrightarrow V$ is injective, we have that $V^*\to H^*$ is surjective, and hence so is $\bigwedge^{p,q}V^*\to\bigwedge^{p,q}H^*$. The compatibility with the decomposition follows from the fact that the pullback operation respects the decomposition in Proposition~\ref{prop:pullbackdecomposition}.
\end{proof}

With the vanishing trace condition, the question is more complicated. Certainly, double forms on $\R^{n+1}$ with vanishing trace restrict to double forms on $T^n$ with vanishing trace. However, this map need not be surjective. As we will see, if $r=0$ and $m=q$, then it is generally not possible to extend a double form in $\oP_r\Lambda^{p,q}_m(T^n)$ to a double form in $\oH_r\Lambda^{p,q}_m(\R^{n+1})$. However, as we will also see, apart from this exceptional case, extension is always possible, via an explicit construction.

This construction relies on some ideas from \cite{berchenkokogan2021duality}; we briefly review the key ideas we will need.

\subsection{The simplex, the sphere, and the Hodge star}\label{sec:simplexspherehodge}
One of the key ideas from \cite{berchenkokogan2021duality} is a coordinate transformation between the simplex and the sphere:
\begin{definition}
  Let $\Phi\colon\R^{n+1}\to\R^{n+1}$ be defined by
  \begin{equation*}
    (\lambda_0,\dots,\lambda_n)=\Phi(u_0,\dots,u_n)=(u_0^2,\dotsc,u_n^2).
  \end{equation*}
\end{definition}
Noting that $\lambda_i=u_i^2\ge0$ and that $u_0^2+\dots+u_n^2=1$ is equivalent to $\lambda_0+\dots+\lambda_n=1$, we see that $\Phi$ maps the unit sphere $S^n$ to the standard simplex $T^n$.

\begin{notation}
  Because of the presence of squares, we will henceforth use subscript notation for coordinates, rather than the Einstein notation of superscripts and subscripts.
\end{notation}

As we will see, one of the key benefits of this coordinate transformation is that it turns vanishing trace into full vanishing on the coordinate hyperplanes. To illustrate, observe that $\dl_i$ has vanishing trace on the hyperplane $\{\lambda_i=0\}$. Indeed, $\dl_i$ vanishes on any vector tangent to the hyperplane. However, it does not vanish on vectors that are not tangent to the hyperplane, such as $\pp{\lambda_i}$. In contrast, the pullback of $\dl_i$ under the transformation $\lambda_i=u_i^2$ is $2u_i\du_i$, which is identically zero on the hyperplane $\{u_i=0\}$, vanishing on all vectors, not just those tangent to $\{u_i=0\}$.

Another key idea from \cite{berchenkokogan2021duality} is the relationship between the Hodge star on the sphere and the Koszul operator. To illustrate, observe that the Hodge star on one-forms on the two-sphere is just $90^\circ$ rotation, which can be realized on vector proxies by taking the cross product with the normal vector. The normal vector on the sphere, however, is just the tautological vector field $X_{\id}$ in the definition of the Koszul operator.

\begin{definition}
  We define the \emph{tautological covector field}
  \begin{equation*}
    \nu:=\sum_{i=0}^nu_i\du_i.
  \end{equation*}
\end{definition}
As the name suggests, $\nu=X_{\id}^\flat$ with respect to the standard metric on the $(u_0,\dots,u_n)$ coordinate system. As the notation suggests, restricted to the unit sphere, $\nu$ is the unit conormal.

\begin{definition}\label{def:starsn}
  We define an operator $\starop_{S^n}\colon\Lambda^k(\R^{n+1})\to\Lambda^{n-k}(\R^{n+1})$ by
  \begin{equation*}
    \starop_{S^n}\alpha:=\starop_{\R^{n+1}}(\nu\wedge\alpha),
  \end{equation*}
  where $\starop_{\R^{n+1}}$ is the usual Hodge star operator on $\R^{n+1}$, with the subscript $\R^{n+1}$ added for clarity.
\end{definition}

As the notation suggests, if we restrict to the sphere, then $\starop_{S^n}$ is the Hodge star on the sphere.

\begin{proposition}[{\cite[Proposition 2.16]{berchenkokogan2021duality}}]\label{prop:starsn}
  Let $\alpha$ be a $k$-form on $\R^{n+1}$, and let $\bar\alpha\in\Lambda^k(S^n)$ be the pullback of $\alpha$ under the inclusion $S^n\hookrightarrow\R^{n+1}$. Then $\starop_{S^n}\bar\alpha$ is the pullback of $\starop_{S^n}\alpha$, where $\starop_{S^n}\bar\alpha$ refers to the Hodge star operator on the sphere, and $\starop_{S^n}\alpha$ refers to Definition~\ref{def:starsn}.
\end{proposition}

The Koszul operator $\kappa$ is the contraction with $X_{\id}$, which is adjoint to wedging with $\nu$, yielding the following relationship.
\begin{proposition}\label{prop:starkoszul}
  For $\alpha\in\Lambda^k(\R^{n+1})$, we have
  \begin{equation*}
    \starop_{S^n}\alpha=(-1)^k\kappa(\starop_{\R^{n+1}}\alpha).
  \end{equation*}
\end{proposition}
\begin{proof}
  Since $\nu=X_{\id}^\flat$, we have $X_{\id}\contract(\starop_{\R^{n+1}}\alpha)=\starop_{\R^{n+1}}(\alpha\wedge\nu)$; see, for example, \cite[Proposition B.1]{berchenkokogan2021duality}. We then compute  
  \begin{equation*}
    \kappa(\starop_{\R^{n+1}}\alpha)=\starop_{\R^{n+1}}(\alpha\wedge\nu)=(-1)^k\starop_{\R^{n+1}}(\nu\wedge\alpha)=(-1)^k\starop_{S^n}\alpha.\qedhere
  \end{equation*}
\end{proof}

\subsection{An overview of the extension construction}\label{sec:overview}
Before we proceed with the extension construction, we give an overview of how it will work, along with some examples.

Given a polynomial double form $\bar\phi$ on $T^n$ with vanishing trace, we can extend it to a double form $\phi$ with homogeneous coefficients on $\R^{n+1}$. Note that, by homogeneity, the fact that $\phi$ vanishes when pulled back to $\partial T^n$ implies that it also vanishes when pulled back to any dilation $c\partial T^n$, where $c\in\R$. The union of these dilations is the union of the hyperplanes $\{\lambda_i=0\}$, so one might ask why $\phi$ does not automatically vanish when pulled back to the hyperplanes $\{\lambda_i=0\}$, which is the vanishing trace condition for $\R^{n+1}$. The answer is that the vanishing trace condition on $\R^{n+1}$ requires that $\phi$ vanish at the hyperplane for all vectors tangent to the hyperplane. On the other hand, we only have vanishing on vectors tangent to the dilates $c\partial T^n$, so, as illustrated in Figure~\ref{fig:dilates}, $\phi$ does not have to vanish if we input a vector that is tangent to $\{\lambda_i=0\}$ but not tangent to $c\partial T^n$. As we will see, this issue is the key issue that needs to be resolved to construct an extension with vanishing trace.

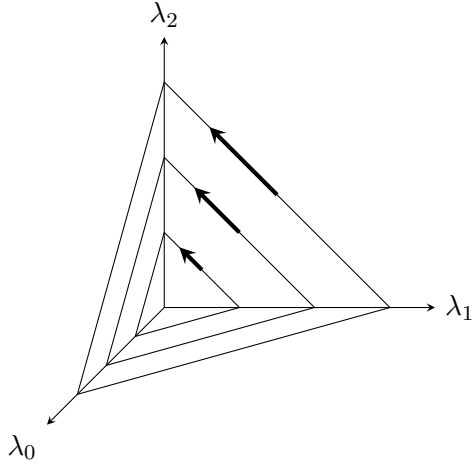
\begin{figure}
  \begin{tikzpicture}[scale=3]
    \coordinate (O) at (0,0,0);
    
    \coordinate (X) at (1,0,0);
    \coordinate (Y) at (0,1,0);
    \coordinate (Z) at (0,0,1);
    
    \coordinate (X2) at (0.666,0,0);
    \coordinate (Y2) at (0,0.666,0);
    \coordinate (Z2) at (0,0,0.666);
    
    \coordinate (X3) at (0.3333,0,0);
    \coordinate (Y3) at (0,0.333,0);
    \coordinate (Z3) at (0,0,0.333);

    \draw[ ->,>=stealth] (O) -- (1.2,0,0) node[anchor=west]{$\lambda_1$};
    \draw[ ->,>=stealth] (O) -- (0,1.2,0) node[anchor=south]{$\lambda_2$};
    \draw[ ->,>=stealth] (O) -- (0,0,1.35) node[anchor=north east]{$\lambda_0$};

    \draw[] (X) -- (Y) -- (Z) -- cycle;
    \draw[] (X2) -- (Y2) -- (Z2) -- cycle;
    \draw[] (X3) -- (Y3) -- (Z3) -- cycle;
    
    \coordinate (M) at (0.5,0.5,0); 
    \coordinate (M2) at (0.333,0.333,0); 
    \coordinate (M3) at (0.167,0.167,0); 
    \draw[->,ultra thick,>=stealth] (M) -- +(-0.3,0.3,0);
    \draw[->,ultra thick,>=stealth] (M2) -- +(-0.2,0.2,0);
    \draw[->,ultra thick,>=stealth] (M3) -- +(-0.1,0.1,0);

  \end{tikzpicture}
  \caption{A form which vanishes on vectors tangent to dilates of $\partial T^n$ need not vanish on all vectors tangent to the coordinate hyperplanes.}
  \label{fig:dilates}
\end{figure}

One way to think about the construction is that we will bring $\bar\phi$ to a venue where extension to $\R^{n+1}$ is automatic, and then return it back once we have extended it. As we will see in Appendix~\ref{sec:rep}, this ``venue'' is more natural not only for solving this extension problem but also for understanding the representation theory of the vanishing trace spaces, yielding, for example, the dimensions of these spaces.

So, we proceed by computing the pullback $\psi=\Phi^*\phi$. Letting $\bar\psi=\Phi^*\bar\phi$ we have that $\bar\psi$ is the pullback of $\psi$ to $S^n$ via $S^n\hookrightarrow\R^{n+1}$. As we discussed, $\bar\psi$ has vanishing trace in a stronger sense. Specifically, at $S^n\cap\{u_i=0\}$, $\bar\psi$ vanishes on all vectors tangent to $S^n$, not just those tangent to $S^n\cap\{u_i=0\}$. By homogeneity, we conclude that $\psi$ vanishes at $\{u_i=0\}$ on all vectors tangent to the dilates $cS^n$. However, as before, in general, $\psi$ will not vanish if we input a vector that is not tangent to $cS^n$, such as $X_{\id}$.

So, now we apply $\ostar_{S^n}$, defined by applying $\star_{S^n}$ to both factors of the double form. At points in $S^n\cap\{u_i=0\}$, since $\psi$ vanishes on all vectors tangent to $S^n$, so does $\ostar_{S^n}\psi$. However, by Proposition~\ref{prop:starkoszul}, we see that $\ostar_{S^n}\psi=(-1)^k\kappa_L\kappa_R\ostar_{\R^{n+1}}\psi$, where $k=p+q$. Therefore, $\ostar_{S^n}\psi$ is in the image of both $\kappa_L$ and $\kappa_R$, and hence in the kernel of both $\kappa_L$ and $\kappa_R$, which we recall are contraction with $X_{\id}$. Thus, unlike $\psi$, we have that $\ostar_{S^n}\psi$ vanishes if we input $X_{\id}$. Since $X_{\id}$ along with the vectors tangent to $S^n$ span the entire tangent space to $\R^{n+1}$, we conclude that $\ostar_{S^n}\psi$ is zero on all vectors at points in $S^n\cap\{u_i=0\}$. By homogeneity, we conclude that $\ostar_{S^n}\psi$ is zero on all vectors at all points in the hyperplane $\{u_i=0\}$. Consequently, all of the polynomial coefficients of $\ostar_{S^n}\psi$ are divisible by $u_i$, so $\ostar_{S^n}\psi$ is divisible by $u_N:=u_0\dotsm u_n$.

So now we divide by $u_N$ and consider $u_N^{-1}\ostar_{S^n}\psi$. By this point, our double form on $S^n$ is very different from what we started with, so our task now is to undo this whole process as far as $S^n$ is concerned. Since $\ostar_{S^n}=(-1)^k\kappa_L\kappa_R\ostar_{\R^{n+1}}$ and $\ostar_{\R^{n+1}}$ is invertible, the task amounts to inverting $\kappa_L\kappa_R$. Proposition~\ref{prop:kerkoszul} is exactly the tool for the job. Since $\ostar_{S^n}\psi$ is in the kernel of both $\kappa_L$ and $\kappa_R$, so is $u_N^{-1}\ostar_{S^n}\psi$, so Proposition~\ref{prop:kerkoszul} applies, and we have $u_N^{-1}\ostar_{S^n}\psi=\kappa_L\kappa_RC^{-1}d_Ld_R(u_N^{-1}\ostar_{S^n}\psi)$, so $(-1)^k\ostar_{\mathbb{R}^{n+1}}^{-1}C^{-1}d_Ld_R\psi$ is the desired inverse image of $u_N^{-1}\ostar_{S^n}\psi$ under $\ostar_{S^n}$. Note that $\ostar_{S^n}$ is not injective on double forms on $\R^{n+1}$, so we do not simply get $u_N^{-1}\psi$. On the other hand, $\ostar_{S^n}$ is certainly bijective on double forms on $S^n$, so the restriction to $S^n$ is indeed simply $u_N^{-1}\bar\psi$.

Our penultimate step is simply to multiply back by $u_N$. Then, restricted to the sphere, we have $\bar\psi$. Meanwhile, on $\R^{n+1}$, we have something that, being a multiple of $u_N$, manifestly vanishes on the hyperplanes. Pushing forward via $\Phi$, we obtain an extension of $\bar\phi$ that has vanishing trace on the hyperplanes, as desired.

In the remainder of this section, we will prove that each step works as described in this overview, but we first provide an example and a counterexample.

\begin{example}
  Let $n=1$, and let $\D \ell$ be the length element of $T^1$, normalized so that the length of $T^1$ is one. Let $\bar\phi=\D \ell\otimes\D \ell$. Since we have a $(1,1)$-form and the boundary of $T^1$ is zero-dimensional, we know that $\bar\phi$ has vanishing trace. Since $\bar\phi$ is symmetric, we have $m=0$. Our goal is to construct an extension of $\bar\phi$ to $\R^2$ that has vanishing trace to the hyperplanes $\{\lambda_0=0\}$ and $\{\lambda_1=0\}$.
  \begin{enumerate}
  \item We first construct an arbitrary extension of $\bar\phi$ to $\cH_0\Lambda^{1,1}_0(\R^2)$. In this case, $\phi=\dl_1\otimes\dl_1$ suffices. Note that, while $\phi$ vanishes on $\{\lambda_1=0\}$, it does \emph{not} vanish on $\{\lambda_0=0\}$. Our goal is to find an extension that does.
  \item We pull back via $\Phi$. Since $\dl_1=2u_1\du_1$, we obtain $4u_1^2\du_1\otimes\du_1$.
  \item We apply $\ostar_{S^n}$.
    \begin{enumerate}
    \item Applying $(\nu\otimes\nu)\owedge$, we obtain
      \begin{equation*}
        4u_0^2u_1^2(\du_0\wedge\du_1)\otimes(\du_0\wedge\du_1).
      \end{equation*}
    \item Applying $\ostar_{\R^{n+1}}$, we obtain $4u_0^2u_1^2$.
    \end{enumerate}
    Note that we could also compute using $\ostar_{S^n}=(-1)^k\kappa_L\kappa_R\ostar_{\R^{n+1}}$.
  \item We divide by $u_N=u_0u_1$, obtaining $4u_0u_1$.
  \item We divide by $C$. In the formula for $C$, we have $r=2$ and $p=q=m=0$, so $C=2$, so we obtain $2u_0u_1$.
  \item We apply $d_Ld_R$, obtaining $2(\du_0\otimes\du_1+\du_1\otimes\du_0)$.
  \item We apply $(-1)^k\ostar_{\mathbb{R}^{n+1}}^{-1}$. We obtain $-2(\du_1\otimes\du_0+\du_0\otimes\du_1)$.
  \item We multiply by $u_N$, obtaining
    \begin{equation*}
      -2u_0u_1(\du_1\otimes\du_0+\du_0\otimes\du_1).
    \end{equation*}
  \item We push forward via $\Phi$, obtaining
    \begin{equation*}
      -\tfrac12(\dl_1\otimes\dl_0+\dl_0\otimes\dl_1).
    \end{equation*}
  \end{enumerate}
  It is clear that the result vanishes if we pull back to $\{\lambda_0=0\}$, so $\dl_0=0$, and likewise if we pull back to $\{\lambda_1=0\}$, so $\dl_1=0$. Note that the result exactly matches the Regge basis $-\dl_i\odot\dl_j$, where $\odot$ denotes the symmetrized tensor product $\alpha \odot \beta = \frac{1}{2}(\alpha \otimes \beta + \beta \otimes \alpha)$~\cite[Proposition 2.7]{li2018regge}.
\end{example}

\begin{example}\label{eg:counterexample}
  Let $n=2$, and consider the area form on $T^2$, normalized so that $T^2$ has area one, interpreted as an antisymmetric $(1,1)$-form. This tensor has vanishing trace, but the construction \emph{fails} because we are in the exceptional case $r=0$ and $m=q$; no extension exists. It is illustrative to see what goes wrong.
  \begin{enumerate}
  \item We begin with an arbitrary extension; $2(\dl_1\otimes\dl_2-\dl_2\otimes\dl_1)$ suffices.
  \item We pull back via $\Phi$, obtaining
    \begin{equation*}
      8u_1u_2(\du_1\otimes\du_2-\du_2\otimes\du_1).
    \end{equation*}
  \item We apply $\ostar_{S^n}$.
    \begin{enumerate}
    \item Applying $(\nu\otimes\nu)\owedge$, we obtain
      \begin{multline*}
        8u_1u_2((u_0\du_0\wedge\du_1+u_2\du_2\wedge\du_1)\otimes(u_0\du_0\wedge\du_2+u_1\du_1\wedge\du_2)\\
        -(u_0\du_0\wedge\du_2+u_1\du_1\wedge\du_2)\otimes(u_0\du_0\wedge\du_1+u_2\du_2\wedge\du_1)).
      \end{multline*}
    \item Applying $\ostar_{\R^{n+1}}$, we obtain
      \begin{multline*}
        8u_1u_2((u_0\du_2-u_2\du_0)\otimes(-u_0\du_1+u_1\du_0)\\
        -(-u_0\du_1+u_1\du_0)\otimes(u_0\du_2-u_2\du_0)),
      \end{multline*}
      which, with cancellation, simplifies to
      \begin{equation*}
        \begin{split}
          8u_0u_1u_2(&u_0(\du_1\otimes\du_2-\du_2\otimes\du_1)\\
          {}+{}&u_1(\du_2\otimes\du_0-\du_0\otimes\du_2)\\
          {}+{}&u_2(\du_0\otimes\du_1-\du_1\otimes\du_0)).
        \end{split}
      \end{equation*}
    \end{enumerate}
  \item Dividing by $u_N$ yields
    \begin{equation*}
      \begin{split}
        8(&u_0(\du_1\otimes\du_2-\du_2\otimes\du_1)\\
        {}+{}&u_1(\du_2\otimes\du_0-\du_0\otimes\du_2)\\
        {}+{}&u_2(\du_0\otimes\du_1-\du_1\otimes\du_0)).
      \end{split}
    \end{equation*}
  \item In the formula for $C$, we have $r=p=q=m=1$, so $C=0$, so we cannot divide by $C$. Indeed, we are in the exceptional case of Proposition~\ref{prop:kerkoszul}, where we are in the kernel of $\kappa_L$ and $\kappa_R$ but fail to be in the image of $\kappa_L\kappa_R$. Alternatively, we can see that we will fail because $d_Ld_R$ will yield zero because our expression has polynomial degree one and $d_Ld_R$ lowers polynomial degree by two.
  \end{enumerate}
\end{example}

For more examples of extensions obtained from this construction, see Table~\ref{tab:bases} in Section~\ref{sec:finite}. We now proceed with discussing each of the operations in the construction in detail.

\subsection{The pullback, vanishing trace, and even double forms}\label{sec:pullback}
We begin by investigating the pullback operation $\Phi^*$ given by $\lambda_i=u_i^2$, $\dl_i=2u_i\du_i$.

\begin{proposition}\label{prop:pullbackTS}
  The pullback $\Phi^*$ is an injective map from $\Lambda^{p,q}(T^n)$ to $\Lambda^{p,q}(S^n)$.
\end{proposition}

\begin{proof}
  Observe that $\Phi$ is a diffeomorphism from the part of $S^n$ in the positive orthant to the interior of $T^n$. So, therefore, if $\bar\psi\in\Lambda^{p,q}(T^n)$ and $\Phi^*\bar\psi=0$, then $\bar\psi$ is zero on the interior of $T^n$. Since $\bar\psi$ is smooth, it must therefore be zero on the boundary of $T^n$, too.
\end{proof}

\begin{proposition}\label{prop:pullback}
  The pullback $\Phi^*$ is an injective map from $\cH_r\Lambda^{p,q}_m(\R^{n+1})$ to $\cH_{2r+k}\Lambda^{p,q}_m(\R^{n+1})$, where $k=p+q$.
\end{proposition}

\begin{proof}
  Let $\phi\in\cH_r\Lambda^{p,q}_m(\R^{n+1})$, and $\psi=\Phi^*\phi$. Because $\lambda_i=u_i^2$, the pullback $\psi$ gets two polynomial degrees per polynomial degree of $\phi$; additionally, from $\dl_i=2u_i\du_i$, $\psi$ acquires one polynomial degree for every form degree of $\phi$. The pullback respects the decomposition by Proposition~\ref{prop:pullbackdecomposition}.

  The proof of injectivity is similar to above. Observe that $\Phi$ is a diffeomorphism if we restrict the domain and codomain to the strictly positive orthant of $\R^{n+1}$. Therefore, if $\psi=0$, we can conclude that $\phi=0$ on the strictly positive orthant. Since $\phi$ has polynomial coefficients, the fact that $\phi$ vanishes on an open set implies that it vanishes on all of $\R^{n+1}$.
\end{proof}

Our construction also requires that we invert the pullback operation, but doing so is not always possible, even for scalar fields. For example, $u_0u_1$ gets pushed forward to $\sqrt{\lambda_0\lambda_1}$, which is not a polynomial. As such, we need additional conditions.

\begin{definition}
  Let $R_i$ be the reflection across the coordinate plane $\{u_i=0\}$, so $u_i\mapsto-u_i$. We say that a double form $\psi$ is \emph{even} if $R_i^*\psi=\psi$ for all $i$.
\end{definition}

Since $u_i\mapsto-u_i$ yields $\du_i\mapsto-\du_i$, to check if a polynomial double form is even, in each term, for each $i$, we count the total number of times $u_i$ or $\du_i$ appears; this total must be even.

\begin{proposition}\label{prop:psieven}
  If $\phi\in\Lambda^{p,q}(\R^{n+1})$ and $\psi=\Phi^*\phi\in\Lambda^{p,q}(\R^{n+1})$, then $\psi$ is even.
\end{proposition}

\begin{proof}
  Since $\lambda_i=u_i^2=(-u_i)^2$, we have that $\Phi\circ R_i=\Phi$, so $R_i^*\Phi^*\phi=\Phi^*\phi$, so $R_i^*\psi=\psi$.
\end{proof}

We might guess that perhaps we find a preimage of $\Phi^*$ for \emph{even} double forms. However, unlike the case of simple forms in \cite{berchenkokogan2021duality}, even the even condition is not enough. For example $u_1^2\du_0\otimes\du_0$ is even, but it gets pushed forward to $\frac{\lambda_1}{4\lambda_0}\dl_0\otimes\dl_0$, which is not a polynomial. However, for our construction, we will only need to push forward double forms that are not only even but also divisible by $u_N:=u_0\dotsm u_n$. We will see that not only is the pushforward a polynomial double form, it also has vanishing trace.

\begin{notation}
  Let $u_N:=\prod_{i=0}^nu_i$ denote the product of the coordinate functions.
\end{notation}

\begin{proposition}\label{prop:pushforward}
  Assume that $\psi\in\cH_{2r+k}\Lambda^{p,q}_m(\R^{n+1})$ is even and divisible by $u_N$. Then $\psi=\Phi^*\phi$ for a unique $\phi\in\oH_r\Lambda^{p,q}_m(\R^{n+1})$ with vanishing trace.
\end{proposition}

\begin{proof}
  Let $\psi=u_N\psi'$. Then $\psi'$ is odd in the sense that $R_i^*\psi'=-\psi'$. Consider a term $f\du_{I,J}$ of  $\psi'$, where $f$ is a monomial. Consider the case where this term contains zero or two copies of $\du_i$, that is, $i$ is in both $I$ and $J$ or in neither of them. Then, for $\psi'$ to be odd, this term must also contain an odd power of $u_i$ in the polynomial factor $f$; in particular, this power must be at least one. Consequently, the corresponding term $u_Nf\du_{I,J}$ in $\psi$ has an even power of $u_i$ that is at least two. As a result, if we had two copies of $\du_i$, we can match up each $\du_i$ with a $u_i$, and we can push forward each $u_i\du_i$ to $\frac12\dl_i$, leaving behind an even power of $u_i$, which pushes forward to an integer power of $\lambda_i$. If we had zero copies of $\du_i$, then we just have a positive even power of $u_i$, which pushes forward to a positive integer power of $\lambda_i$. In particular, the pushforward must have at least one $\lambda_i$ or $\dl_i$.

  Meanwhile, a term $u_Nf\du_{I,J}$ of $\psi$ that contains one copy of $\du_i$ must also have an odd power of $u_i$ in the polynomial factor $u_Nf$ because $\psi$ is even. We likewise have that $u_i\du_i$ pushes forward to $\frac12\dl_i$, leaving behind an even power of $u_i$, which pushes forward to an integer power of $\lambda_i$.

  Note that, in either case, for every $i$, the term of the pushforward has at least one $\lambda_i$ or $\dl_i$, so it vanishes when pulled back to the hyperplane $\{\lambda_i=0\}$, as required by the definition of vanishing trace on $\R^{n+1}$.

  Thus, there exists a pushforward $\phi\in\oH_r\Lambda^{p,q}(\R^{n+1})$. The pushforward is unique because $\Phi^*$ is injective. With regards to the decomposition, to check that $\phi\in\Lambda^{p,q}_m$, we can let $\phi_m$ be the projection of $\phi\in\Lambda^{p,q}$ onto the $\Lambda^{p,q}_m$ summand. By Proposition~\ref{prop:pullbackdecomposition}, since $\psi\in\Lambda^{p,q}_m$, we have that $\Phi^*\phi_m=\psi$. By the uniqueness of $\phi$, we have $\phi=\phi_m$.
\end{proof}

The converse holds as well.

\begin{proposition}
  If $\phi\in\oH_r\Lambda^{p,q}_m(\R^{n+1})$, then $\psi=\Phi^*\phi\in\cH_{2r+k}\Lambda^{p,q}_m(\R^{n+1})$ is even and divisible by $u_N$.
\end{proposition}

\begin{proof}
  We have already shown that $\psi$ is even in Proposition~\ref{prop:psieven}. To show that it is divisible by each $u_i$, we will show that $\psi$ fully vanishes on each of the coordinate hyperplanes $\{u_i=0\}$, in the sense that it vanishes on all vectors, not just those vectors tangent to the hyperplane.

  Let $u=(u_0,\dots,u_n)$ be a point on the hyperplane $u_i=0$, and let $\lambda=\Phi(u)$. Let $e_i$ denote the $i$th coordinate basis vector at $u$, that is, $e_i=\pp{u_i}\Bigr\rvert_u$. Any vector at $u$ can be written in the form $be_i+X$, where $b$ is a real number and $X$ is tangent to the hyperplane.

  Observe that the push forward $\Phi_*e_i$ is zero. Indeed, $\pp{u_i}=\pp[\lambda_i]{u_i}\pp{\lambda_i}=2u_i\pp{\lambda_i}$, which is zero at $u$. Meanwhile, vectors tangent to the hyperplane $\{u_i=0\}$ get pushed forward to vectors tangent to the hyperplane $\{\lambda_i=0\}$. So, at $u$, applying $\psi$ to vectors written in the form $be_i+X$, we obtain
  \begin{equation}\label{eq:boundaryvanish}
    \begin{split}
      &\psi\rvert_u(b_1e_i+X_1,\dots,b_pe_i+X_p;c_1e_i+Y_1,\dots,c_qe_i+Y_q)\\
      &=\phi\rvert_\lambda(\Phi_*(b_1e_i+X_1),\dots,\Phi_*(b_pe_i+X_p);\\
      &\qquad\qquad\Phi_*(c_1e_i+Y_1),\dots,\Phi_*(c_qe_i+Y_q))\\
      &=\phi\rvert_\lambda(\Phi_*X_1,\dots,\Phi_*X_p;\Phi_*Y_1,\dots,\Phi_*Y_q),\\
    \end{split}
  \end{equation}
  which is zero because $\phi$ has vanishing trace and the $\Phi_*X_a$ and $\Phi_*Y_a$ are tangent to the hyperplane $\{\lambda_i=0\}$.

  We conclude that every polynomial coefficient of $\psi$ vanishes on the hyperplane $\{u_i=0\}$, from which we conclude that every polynomial coefficient is divisible by $u_i$. This argument holds for every $i$, so $\psi$ is divisible by $u_N$.
\end{proof}

We can similarly prove the corresponding statement for $T^n$ and $S^n$, namely that vanishing trace on $T^n$ (vanishing of tangential components on $\partial T^n$) yields full vanishing (all components) on the great circles of $S^n$.

\begin{proposition}\label{prop:vanishgreatcircle}
  Let $\bar\phi\in\Lambda^{p,q}(T^n)$ and let $\bar\psi=\Phi^*\bar\phi\in\Lambda^{p,q}(S^n)$. If $\bar\phi$ has vanishing trace, then $\bar\psi$ fully vanishes on the great circles $\{u_i=0\}$, in the sense that it vanishes on all vectors tangent to $S^n$, not just those vectors tangent to the great circle.
\end{proposition}

\begin{proof}
  Let $u=(u_0,\dotsc,u_n)$ be a point on $S^n$, and assume that $u_i=0$. Let $\lambda=\Phi(u)$, a point on the boundary component $\{\lambda_i=0\}$ of $T^n$.

  As before, let $e_i$ denote the $i$th coordinate basis vector at $u$, that is, $e_i=\pp{u_i}\Bigr\rvert_u$. Because $u_i=0$, we have $e_i\in T_uS^n$. Note that $e_i$ is normal to the great circle $\{u_i=0\}$, so any vector in $T_uS^n$ can be written in the form $be_i+X$, where $b$ is a real number and $X$ is tangent to the great circle $\{u_i=0\}$.

  As before, the push forward $\Phi_*e_i$ is zero, and vectors tangent to the great circle $\{u_i=0\}$ get pushed forward to vectors tangent to the boundary component $\{\lambda_i=0\}$ of $T^n$. So, we have Equation~\eqref{eq:boundaryvanish} except with $\bar\psi$ and $\bar\phi$ instead of $\psi$ and $\phi$, and the expression is zero because $\bar\phi$ has vanishing trace and the $\Phi_*X_a$ and $\Phi_*Y_a$ are tangent to the boundary component $\{\lambda_i=0\}$.
\end{proof}

\subsection{The double Hodge star on the sphere}\label{sec:doublehodgestar}
Recall from Definition~\ref{def:operatorsM} that we have double Hodge star operations $\ostar_{S^n}\colon\Lambda^{p,q}(S^n)\to\Lambda^{n-p,n-q}(S^n)$ and $\ostar_{\R^{n+1}}\colon\Lambda^{p,q}(\R^{n+1})\to\Lambda^{n+1-p,n+1-q}(\R^{n+1})$ by applying $\star$ to each factor of the double form. Recall from Definition~\ref{def:starsn} that we defined $\star_{S^n}$ on differential forms on $\R^{n+1}$, so we can analogously define  $\ostar_{S^n}$ on double forms on $\R^{n+1}$ as well.

\begin{definition}\label{def:ostarsn}
  Define $\ostar_{S^n}\colon\Lambda^{p,q}(\R^{n+1})\to\Lambda^{n-p,n-q}(\R^{n+1})$ on simple tensors by
  \begin{equation*}
    \ostar_{S^n}(\alpha\otimes\beta):=(\starop_{S^n}\alpha)\otimes(\starop_{S^n}\beta).
  \end{equation*}
\end{definition}

Each proposition in Section~\ref{sec:simplexspherehodge} yields analogous propositions for double forms.

\begin{proposition}\label{prop:ostarrestrict}
  Let $\psi$ be a $(p,q)$-form on $\R^{n+1}$ and let $\bar\psi$ be its pullback under the inclusion $S^n\hookrightarrow\R^{n+1}$. Then $\ostar_{S^n}\bar\psi$ is the pullback of $\ostar_{S^n}\psi$, where $\ostar_{S^n}\bar\psi$ refers to Definition~\ref{def:operatorsM} and $\ostar_{S^n}\psi$ refers to Definition~\ref{def:ostarsn}.
\end{proposition}

\begin{proof}
  On simple tensors, the claim follows by applying Proposition~\ref{prop:starsn} to each factor, and then we extend by linearity.
\end{proof}

\begin{proposition}\label{prop:ostarkoszul}
  For $\psi\in\Lambda^{p,q}(\R^{n+1})$, we have
  \begin{equation*}
    \ostar_{S^n}\psi=(-1)^k\kappa_L\kappa_R(\ostar_{\R^{n+1}}\psi),
  \end{equation*}
  where $k=p+q$.
\end{proposition}

\begin{proof}
  As before, on simple tensors, the claim follows by applying Proposition~\ref{prop:starkoszul} to each factor.
\end{proof}

In general, forms $\bar\psi$ on the sphere do not have unique extensions to forms $\psi$ on $\R^{n+1}$, not even if we require $\psi$ to be homogeneous. However, as we will see, we do get uniqueness if we additionally require $\psi$ to be in the image of $\ostar_{S^n}$.

\begin{lemma}\label{lem:ostarvanish}
  Let $\psi\in\cH_r\Lambda^{p,q}(\R^{n+1})$, and let $\bar\psi$ be its pullback under the inclusion $S^n\hookrightarrow\R^{n+1}$. Let $u$ be a point on the sphere, and assume that $\bar\psi\rvert_u$ is the zero double multicovector on $T_uS^n$. Then, for any real number $h$, $\ostar_{S^n}\psi\rvert_{hu}$ is the zero double multicovector on $T_{hu}\R^{n+1}$.
\end{lemma}

\begin{proof}
  By Proposition~\ref{prop:ostarrestrict}, the pullback of $\ostar_{S^n}\psi$ to the sphere is $\ostar_{S^n}\bar\psi$, which is zero at $u$. In other words, $\ostar_{S^n}\psi\rvert_u(X_1,\dotsc,X_p;Y_1,\dots,Y_q)=0$ for all vectors $X_a,Y_a\in T_uS^n$.

  Observe that the tautological vector field $X_{\id}=\sum_{i=0}^nu_i\pp{u_i}$ is normal to the sphere, so any vector in $X\in T_u\R^{n+1}$ can be written as $X=bX_{\id}+\bar X$, where $b$ is a real number and $\bar X$ is tangent to the sphere. By Proposition~\ref{prop:ostarkoszul}, we have that $\ostar_{S^n}\psi$ is in the image of $\kappa_L\kappa_R$, so it is in the kernel of both $\kappa_L$ and $\kappa_R$. Consequently, by antisymmetry, the expression $\ostar_{S^n}\psi\rvert_u(X_1,\dots,X_p;Y_1,\dots,Y_q)$ vanishes if any of the $X_a$ or $Y_a$ are the tautological vector field $X_{\id}$. By multilinearity, writing each $X_a$ and $Y_a$ in the above form, we obtain
  \begin{equation*}
    \begin{split}
      &\ostar_{S^n}\psi\rvert_u(X_1,\dots,X_p;Y_1,\dots,Y_q)\\
      &=\ostar_{S^n}\psi\rvert_u(b_1X_{\id}+\bar X_1,\dots,b_pX_{\id}+\bar X_p;c_1X_{\id}+\bar Y_1,\dots,c_qX_{\id}+\bar Y_q)\\
      &=\ostar_{S^n}\psi\rvert_u(\bar X_1,\dots,\bar X_p;\bar Y_1,\dots,\bar Y_q),
    \end{split}
  \end{equation*}
  which is zero because $\ostar_{S^n}\psi\rvert_u$ vanishes on vectors tangent to the sphere.
  
  We have shown that $\ostar_{S^n}\psi\rvert_u$ is the zero double multicovector. In other words, in the standard form, all of the polynomial coefficients of $\ostar_{S^n}\psi$ vanish at this point $u$. Because the polynomial coefficients are homogeneous, they must also vanish at any scalar multiple of this point, so $\ostar_{S^n}\psi\rvert_{hu}$ is zero for any real number $h$.
\end{proof}

This lemma immediately implies that extensions become unique after applying the Hodge star.

\begin{proposition}\label{prop:ostarunique}
  Let $\psi,\psi'\in\cH_r\Lambda^{p,q}(\R^{n+1})$, and assume that the pullbacks of $\psi$ and $\psi'$ under the inclusion $S^n\hookrightarrow\R^{n+1}$ are equal. Then $\ostar_{S^n}\psi=\ostar_{S^n}\psi'$.
\end{proposition}

\begin{proof}
  The pullback of $\psi-\psi'$ is zero at every point on the sphere. Since any point of $\R^{n+1}$ is a scalar multiple of a point on the sphere, we conclude by Lemma~\ref{lem:ostarvanish} that $\ostar_{S^n}(\psi-\psi')=0$ at every point of $\R^{n+1}$.
\end{proof}

The lemma also yields an essential relationship between forms on $S^n$ that fully vanish on great circles and forms on $\R^{n+1}$ that are divisible by $u_N$.

\begin{proposition}\label{prop:ostardivisible}
  Let $\psi\in\cH_r\Lambda^{p,q}(\R^{n+1})$, and let $\bar\psi$ be the pullback of $\psi$ under the inclusion $S^n\hookrightarrow\R^{n+1}$. Assume that $\bar\psi$ fully vanishes on the great circles of $S^n$, that is, at every point $u\in S^n$ with some $u_i=0$, we have that $\psi\rvert_u(X_1,\dots,X_p;Y_1,\dots,Y_q)=0$ for all vectors $X_a,Y_a$ in $T_uS^n$. Then $\ostar_{S^n}\psi$ is divisible by $u_N=\prod_{i=0}^nu_i$.
\end{proposition}

\begin{proof}
  Any point on the hyperplane $\{u_i=0\}$ in $\R^{n+1}$ is a scalar multiple of a point on the great circle $\{u_i=0\}$ in $S^n$, so Lemma~\ref{lem:ostarvanish} tells us that $\ostar_{S^n}\psi\rvert_u$ is the zero double multicovector for any point $u$ with $u_i=0$. Hence, in standard form, the polynomial coefficients of $\ostar_{S^n}\psi$ vanish on the entire hyperplane $\{u_i=0\}$. Consequently, these polynomial coefficients must be divisible by $u_i$.
\end{proof}

\subsection{The extension construction}\label{sec:extensionconstruction}
We now have the tools to follow the steps outlined in Section~\ref{sec:overview} to construct extensions of double forms on $T^n$ with vanishing trace, and to understand when the construction fails and the extension does not exist.

\begin{theorem}\label{thm:extension}
  Let $\bar\phi\in\oP_r\Lambda^{p,q}_m(T^n)$ be nonzero. Let $\phi\in\cH_r\Lambda^{p,q}_m(\R^{n+1})$ be an arbitrary extension of $\bar\phi$ to $\R^{n+1}$. Provided we are not in the case $r=0$ and $m=q$, then $\bar\phi$ also has an extension $\phi'\in\oH_r\Lambda^{p,q}_m(\R^{n+1})$ with vanishing trace, given by the formula
  \begin{equation*}
    \phi'=(-1)^kC^{-1}(\Phi^*)^{-1}u_N\ostar_{\R^{n+1}}^{-1}d_Ld_R(u_N^{-1}\ostar_{S^n}\Phi^*\phi),
  \end{equation*}
  where $k=p+q$, $C=(2r+p+m+1)(2r+q-m)$, and $u_N=\prod u_i$.

  Moreover, the extension $\phi'$ given by this formula depends only on $\bar\phi$, not on the arbitrary extension $\phi$.
\end{theorem}

\begin{proof}
  We have all the ingredients, so now we just apply each operator step by step. Let $\psi=\Phi^*\phi$ and $\bar\psi$ be its restriction to the sphere $S^n$, so we also have $\bar\psi=\Phi^*\bar\phi$. By Proposition~\ref{prop:pullback}, we have $\psi\in\cH_{2r+k}\Lambda^{p,q}_m(\R^{n+1})$ and $\bar\psi$ fully vanishes on the great circles $\{u_i=0\}$ by Proposition~\ref{prop:vanishgreatcircle}.

  We now consider $\ostar_{S^n}\psi$. We first remark that Proposition~\ref{prop:ostarunique} implies that $\ostar_{S^n}\psi$ only depends on $\bar\psi$, which in turns depends only on $\bar\phi$, not on the arbitrary extension $\phi$. Hence, $\phi'$ only depends on $\bar\phi$, not on $\phi$, as desired.

  Next, we have that $\ostar_{S^n}\psi=(-1)^k\kappa_L\kappa_R\ostar_{\R^{n+1}}\psi$ by Proposition~\ref{prop:ostarkoszul}. Noting that $\ostar_{\R^{n+1}}$ does not change polynomial degree, by Proposition~\ref{prop:ostardecomp}, we have that $\ostar_{\R^{n+1}}\psi\in\cH_{2r+k}\Lambda^{n+1-p,n+1-q}_{m^*}(\R^{n+1})$, where $m^*=m+p-q$. Then, noting that $\kappa_L$ and $\kappa_R$ raise polynomial degree and lower form degree, we have by Proposition~\ref{prop:kkdecomposition} that $\ostar_{S^n}\psi\in\cH_{2r+k+2}\Lambda^{n-p,n-q}_{m^*}(\R^{n+1})$. Its restriction to the sphere is $\ostar_{S^n}\bar\psi$ by Proposition~\ref{prop:ostarrestrict}.

  By Proposition~\ref{prop:ostardivisible}, $\ostar_{S^n}\psi$ is divisible by $u_N=\prod_{i=0}^nu_i$. So, $u_N^{-1}\ostar_{S^n}\psi\in\cH_{2r+k-n+1}\Lambda^{n-p,n-q}_{m^*}(\R^{n+1})$. Recalling that $\ostar_{S^n}\psi=(-1)^k\kappa_L\kappa_R\ostar_{\R^{n+1}}\psi$, we have that $\ostar_{S^n}\psi$ is in the kernel of both $\kappa_L$ and $\kappa_R$. Since multiplication by $u_N^{-1}$ commutes with $\kappa_L$ and $\kappa_R$, we conclude that $u_N^{-1}\ostar_{S^n}\psi$ is in the kernel of $\kappa_L$ and $\kappa_R$, too. So now we would like to apply Proposition~\ref{prop:kerkoszul} to show that $u_N^{-1}\ostar_{S^n}\psi$ is in the image of $\kappa_L\kappa_R$.

  To do so, we must first deal with the exceptional cases of Proposition~\ref{prop:kerkoszul}. First, Proposition~\ref{prop:kerkoszul} requires that $u_N^{-1}\ostar_{S^n}\psi$ be nonzero. Assume for the sake of contradiction that $u_N^{-1}\ostar_{S^n}\psi=0$, so then $\ostar_{S^n}\psi$ would be zero. Although $\ostar_{S^n}$ is not injective on forms on $\R^{n+1}$, it is bijective on forms on $S^n$, so then we could conclude that $\bar\psi$ is zero, from which it would follow by Proposition~\ref{prop:pullbackTS} that $\bar\phi$ is zero, which we assumed is not the case.

  Since $u_N^{-1}\ostar_{S^n}\psi\in\cH_{2r+k-n+1}\Lambda^{n-p,n-q}_{m^*}(\R^{n+1})$, the first case of Proposition~\ref{prop:kerkoszul} reads $2r+k-n+1=n-p=n-q=m^*=0$. In particular $p=n$, $q=n$, $k=p+q=2n$, and so $2r+k-n+1=2r+n+1$, which cannot be zero.

  The second case of Proposition~\ref{prop:kerkoszul} reads $2r+k-n+1=1$ and $m^*=n-q$. The first equation gives $n=2r+k$, which implies $n\ge k$. Recalling that $m^*=m+p-q$, the second equation gives $n=m+p$. Recalling that $m\le q$, we have $n\le q+p=k$. We conclude that $n=k$, so $r=0$, and $m=q$, which is the exceptional case excluded in the theorem statement. As we saw in Example~\ref{eg:counterexample} and will see more generally below, vanishing trace extension is not possible in this case.

  So, we are in the general case of Proposition~\ref{prop:kerkoszul}, so
  \begin{equation}\label{eq:ustarpsi}
    u_N^{-1}\ostar_{S^n}\psi=C^{-1}\kappa_L\kappa_Rd_Ld_R(u_N^{-1}\ostar_{S^n}\psi),
  \end{equation}
  where $C$ is
  \begin{multline*}
    \bigl((2r+k-n+1)+(n-p)+(m+p-q)\bigr)\bigl((2r+k-n+1)+(n-q)-(m+p-q)-1\bigr)\\
    =(2r+p+m+1)(2r+q-m).
  \end{multline*}
  So now let
  \begin{equation}\label{eq:psiprime}
    \psi'=(-1)^kC^{-1}u_N\ostar_{\mathbb{R}^{n+1}}^{-1}d_Ld_R(u_N^{-1}\ostar_{S^n}\psi).
  \end{equation}
  Per Proposition~\ref{prop:dddecomposition}, $d_Ld_R(u_N^{-1}\ostar_{S^n}\psi)\in\cH_{2r+k-n-1}\Lambda^{n-p+1,n-q+1}_{m^*}$, so then $\ostar_{\R^{n+1}}^{-1}d_Ld_R(u_N^{-1}\ostar_{S^n}\psi)\in\cH_{2r+k-n-1}\Lambda^{p,q}_m$, and so $\psi'\in\cH_{2r+k}\Lambda^{p,q}_m$. So then, applying Proposition~\ref{prop:ostarkoszul}, using the fact that multiplication by $u_N$ commutes with pointwise operations $\kappa_L$, $\kappa_R$, and $\ostar_{\R^{n+1}}$, and using Equation~\eqref{eq:ustarpsi}, we obtain
  \begin{multline*}
    \ostar_{S^n}\psi'=(-1)^k\kappa_L\kappa_R\ostar_{\R^{n+1}}\psi'=C^{-1}u_N\kappa_L\kappa_Rd_Ld_R(u_N^{-1}\ostar_{S^n}\psi)=\ostar_{S^n}\psi.
  \end{multline*}
  Letting $\bar\psi'$ be the restriction of $\psi'$ to the sphere, we conclude by Proposition~\ref{prop:ostarrestrict} that $\ostar_{S^n}\bar\psi'=\ostar_{S^n}\bar\psi$, so $\bar\psi'=\bar\psi$ because $\ostar_{S^n}$ is bijective on the sphere.

  The final step is to let $\phi'=(\Phi^*)^{-1}\psi'$, but to do so we must verify that the push forward exists using Proposition~\ref{prop:pushforward}. We know that $\psi$ is even by Proposition~\ref{prop:psieven}. We recall the notation that $R_i$ is the reflection across the hyperplane $\{u_i=0\}$ given by $u_i\mapsto-u_i$. The Hodge star $\star_{\R^{n+1}}$ anticommutes with pullback under reflections, so then the double Hodge star $\ostar_{\R^{n+1}}$ commutes with pullback under reflections. The operations $d_L$ and $d_R$ commute with any pullback. Since the vector field $X_{\id}$ is invariant under reflection, $\kappa_L$ and $\kappa_R$ commute with $R_i^*$. Finally, $R_i^*u_N=-u_N$, so multiplication or division by $u_N$ anticommutes with $R_i^*$. So, all of the operations in Equation~\eqref{eq:psiprime} commute with $R_i^*$, with the exception of $u_N$ and $u_N^{-1}$, each of which anticommutes with $R_i^*$. We conclude that $\psi'$ is even. We have that $\psi'$ is divisible by $u_N$ by construction. So, by Proposition~\ref{prop:pushforward}, there exists a unique $\phi'\in\oH_r\Lambda^{p,q}_m(\R^{n+1})$ with $\Phi^*\phi'=\psi'$.

  Letting $\bar\phi'$ be the pullback of $\phi'$ to $T^n$, we have $\Phi^*\bar\phi'=\bar\psi'=\bar\psi=\Phi^*\bar\phi$, so $\bar\phi'=\bar\phi$ by Proposition~\ref{prop:pullbackTS}. We conclude that $\phi'$ is the desired vanishing trace extension of $\bar\phi$.
\end{proof}

We also show that extension fails in the exceptional case $r=0$ and $m=q$.

\begin{proposition} \label{prop:noextension}
  Let $\bar\phi\in\oP_0\Lambda^{p,q}_q(T^n)$ be nonzero. Then there does not exist a vanishing trace extension $\phi'\in\oH_0\Lambda^{p,q}_q(\R^{n+1})$ of $\bar\phi$.
\end{proposition}

\begin{proof}
  The initial part of the proof of Theorem~\ref{thm:extension} proceeds as before, with $\phi$ an arbitrary extension of $\bar\phi$, then setting $\psi:=\Phi^*\phi$, and then finding that $u_N^{-1}\ostar_{S^n}\psi$ is a nonzero element of $\cH_{2r+k-n+1}\Lambda^{n-p,n-q}_{m^*}$ that is in the kernel of $\kappa_L$ and $\kappa_R$. The first case of Proposition~\ref{prop:kerkoszul} likewise yields a contradiction.

  So, we are in the second or third case of Proposition~\ref{prop:kerkoszul}. Plugging in $r=0$ and $m=q$, we find that $u_N^{-1}\ostar_{S^n}\psi\in\cH_{k-n+1}\Lambda^{n-p,n-q}_p$. Recalling that $\Lambda^{p,q}_m$ being nonzero implies $m\le q$, we have that $\Lambda^{n-p,n-q}_p$ being nonzero implies $p\le n-q$, so $k=p+q\le n$, and so the polynomial degree $k-n+1$ is at most $1$. We conclude that we cannot be in the third case of Proposition~\ref{prop:kerkoszul}, so we must be in the second case, and so $k-n+1=1$ and $p=n-q$, both of which tell us that $k=n$.

  We claim that, when $k=n$, the space $\oH_0\Lambda^{p,q}_q(\R^{n+1})$ is zero and hence cannot contain an extension of a nonzero double form on $T^n$. By Proposition~\ref{prop:imageipq}, the space $\oH_0\Lambda^{p,q}_q(\R^{n+1})$ is the image of $\oH_0\Lambda^k(\R^{n+1})$ under the inclusion $i^{p,q}$ of $k$-forms into $(p,q)$-forms, so it suffices to show that $\oH_0\Lambda^k(\R^{n+1})=0$ when $k=n$.

  For any $\alpha\in\oH_0\Lambda^n(\R^{n+1})$, we can write it as $\alpha=\sum_i a_i\starop_{\R^{n+1}}\dl_i$, where the $a_i$ are constants. By assumption $\alpha$ vanishes when pulled back to every hyperplane $\{\lambda_i=0\}$, so we investigate what happens to the terms in the right-hand side under this restriction. The restriction of $\starop_{\R^{n+1}}\dl_i$ is nonzero; it is just the volume form on this hyperplane, which we can denote $\mu_i$. On the other hand, for $j\neq i$, the restriction of $\starop_{\R^{n+1}}\dl_j$ is zero because $\starop_{\R^{n+1}}\dl_j$ is a wedge product of $n$ factors including $\dl_i$. So, since $\alpha$ has vanishing trace, pulling back the equation $\alpha=\sum_ia_i\starop_{\R^{n+1}}\dl_i$ to the hyperplane yields $0=a_i\mu_i$, so $a_i=0$. This argument holds for all $i$, so $\alpha=0$.
\end{proof}

\section{Finite element spaces} \label{sec:finite}

Thanks to Theorem~\ref{thm:extension} and a recent paper by the first author~\cite{berchenko2026extension}, we now have the ingredients that are needed to construct global finite spaces for $(p,q)$-forms on simplicial triangulations.  We begin in Section~\ref{sec:geo} by constructing a piecewise polynomial finite element space for each summand $\Lambda^{p,q}_m$ in the decomposition $\Lambda^{p,q} = \bigoplus_m \Lambda^{p,q}_m$, with the exception of $\Lambda^{p,q}_q$ which fails to admit a piecewise constant discretization.  Our construction yields a geometric decomposition of each finite element space and its dual.  The degrees of freedom for the finite element spaces are given by integration against $(p,q)$-forms on a simplex that have vanishing trace on the boundary of the simplex.   In Section~\ref{sec:dim}, we investigate the piecewise constant finite element spaces in more detail by computing their dimensions along with the dimensions of the vanishing-trace subspaces.  In Section~\ref{sec:examples}, we discuss some of the piecewise constant finite element spaces produced by our construction, paying special attention to the three-dimensional setting and pointing out some connections with known finite element spaces.  We also present bases for a few of the piecewise constant finite element spaces in Table~\ref{tab:bases}.  In Section~\ref{sec:related}, we discuss the relationship between our work and that of Hu and Lin~\cite{hu2025finite}.

\subsection{Geometric decomposition} \label{sec:geo}

Following \cite{berchenko2026extension}, we show that the extension operator in Theorem~\ref{thm:extension} yields finite element spaces for double forms on simplicial triangulations.

Let $p,q,m,n,r$ be integers satisfying $0 \le p,q \le n$, $\max\{0,q-p\} \le m \le \min\{q,n-p\}$, and $r \ge 0$.  Theorem~\ref{thm:extension} tells us that as long as $(r,m) \neq (0,q)$, every double form $\phi \in \mathring{\mathcal{P}}_r\Lambda^{p,q}_m(T^n)$ admits an extension $\phi' \in \mathring{\mathcal{H}}_r\Lambda^{p,q}_m(\mathbb{R}^{n+1})$.  By restricting the domain of $\phi'$ to the $(n+1)$-simplex
\[
  Q^{n+1} := \left\{(\lambda_0,\dots,\lambda_n) \mid \lambda_i \ge 0, \, \sum_i \lambda_i \le 1 \right\},
\]
we can view $\phi'$ as a member of $\mathcal{P}_r\Lambda^{p,q}_m(Q^{n+1})$, the space of $(p,q)$-forms on $Q^{n+1}$ that belong pointwise to $\Lambda^{p,q}_m$ and have polynomial coefficients of degree at most $r$. The fact that $\phi'$ vanishes on the hyperplanes of $\R^{n+1}$ implies that it vanishes on all proper faces of $Q^{n+1}$ other than $T^n$.

\begin{definition} \label{def:simplicialextension}
  Let 
  \[
    E_n : \mathring{\mathcal{P}}_r\Lambda^{p,q}_m(T^n) \to \mathcal{P}_r\Lambda^{p,q}_m(Q^{n+1})
  \]
  denote the extension operator described above, which, per Theorem~\ref{thm:extension}, exists whenever $(r,m) \neq (0,q)$.  
\end{definition}

Because the extension $E_n\phi$ vanishes on all proper faces of $Q^{n+1}$ other than $T^n$, the extension operator $E_n$ satisfies the definition of a \emph{simplicial extension operator} in \cite[Definition~3.8]{berchenko2026extension}.

Now let $N$ be fixed and consider a simplicial triangulation $\mathcal{T}$ of dimension $N$.  Let $\mathfrak{T}$ be the partially ordered set consisting of the vertices, edges, faces, etc.~of $\mathcal{T}$, ordered by inclusion.  

\begin{definition}
  For each pair of simplices $K \le F$ in the triangulation $\mathcal{T}$, let 
  \[
    \tr^F_K : \mathcal{P}_r\Lambda^{p,q}_m(F) \to \mathcal{P}_r\Lambda^{p,q}_m(K)
  \]
  denote the pullback under the inclusion $K \hookrightarrow F$.
\end{definition}

\begin{definition}
  Let the \emph{global function space} be
  \[
    \mathcal{P}_r \Lambda^{p,q}_m(\mathcal{T})  = \left\{\left(\phi_F\right)_{F\in\mathfrak{T}}\in\prod_{F\in\mathfrak{T}}\mathcal{P}_r \Lambda^{p,q}_m(F)\mid\tr^F_K\phi_F=\phi_K\text{ for all }K\le F\right\}.
  \]
  We define $\tr^{\mathcal{T}}_F\colon\mathcal{P}_r \Lambda^{p,q}_m(\mathcal{T})\to\mathcal{P}_r \Lambda^{p,q}_m(F)$ to be the natural projection maps.
\end{definition}

Note that if $\mathcal{T}$ triangulates an $N$-dimensional polyhedral domain $\Omega \subset \mathbb{R}^{N}$, then $\mathcal{P}_r \Lambda^{p,q}_m(\mathcal{T})$ can be identified with the space of double forms in $\Lambda^{p,q}_m(\Omega)$ that are piecewise polynomial with respect to $\mathcal{T}$ and have single-valued trace on element interfaces.

It is proved in~\cite{berchenko2026extension} that the existence of the simplicial extension operators $E_n$ in all dimensions $n<N$ implies that $\mathcal{P}_r \Lambda^{p,q}_m(\mathcal{T})$ admits \emph{local extension operators}
\[
  E_K^{\mathcal{T}} : \mathring{\mathcal{P}}_r\Lambda^{p,q}_m(K) \to \mathcal{P}_r \Lambda^{p,q}_m(\mathcal{T}), \quad K \in \mathfrak{T}
\]
with the following properties:
\begin{itemize}
\item $\tr^{\cT}_KE^{\cT}_K\colon\mathring{\mathcal{P}}_r\Lambda^{p,q}_m(K)\to\mathcal{P}_r\Lambda^{p,q}_m(K)$ is the inclusion, and
\item $\tr^{\cT}_FE^{\cT}_K=0$ if $K\not\le F$.
\end{itemize}

Moreover, these local extension operators give rise to the following geometric decomposition of $\mathcal{P}_r \Lambda^{p,q}_m(\mathcal{T})$.
\begin{theorem} \label{thm:geodecomp}
  Provided that $(r,m) \neq (0,q)$, we have
  \[
    \mathcal{P}_r \Lambda^{p,q}_m(\mathcal{T}) = \bigoplus_{F \in \mathfrak{T}} E_F^{\mathcal{T}} \mathring{\mathcal{P}}_r\Lambda^{p,q}_m(F).
  \]
\end{theorem}
\begin{proof}
  This follows from Corollary 3.10 in~\cite{berchenko2026extension}.
\end{proof}
The decomposition above gives rise to a basis for $\mathcal{P}_r \Lambda^{p,q}_m(\mathcal{T})$ consisting of double forms $E^{\mathcal{T}}_F \phi$ where $F \in \mathfrak{T}$ and $\phi \in \mathring{\mathcal{P}}_r\Lambda^{p,q}_m(F)$, much like the Bernstein basis for scalar-valued Lagrange finite elements.  A few of these bases for $r=0$ and various choices of $p,q$, and $m$ are listed in Table~\ref{tab:bases}.

We also obtain the following geometric decomposition of the dual space $\mathcal{P}_r \Lambda^{p,q}_m(\mathcal{T})^*$.
\begin{theorem} \label{thm:dualdecomp}
  Provided that $(r,m) \neq (0,q)$, we have
  \[
    \mathcal{P}_r \Lambda^{p,q}_m(\mathcal{T})^* = \bigoplus_{F \in \mathfrak{T}} (\tr_F^{\mathcal{T}})^* \mathring{\mathcal{P}}_r\Lambda^{p,q}_m(F)^\dagger,
  \]
  where $(\tr^{\mathcal{T}}_F)^*\colon\mathcal{P}_r \Lambda^{p,q}_m(F)^*\to\mathcal{P}_r \Lambda^{p,q}_m(\mathcal{T})^*$ denotes the dual map and $\mathring{\mathcal{P}}_r\Lambda^{p,q}_m(F)^\dagger$ denotes the set of all linear functionals of the form $\langle \phi, \cdot \rangle$, where $\phi \in \mathring{\mathcal{P}}_r\Lambda^{p,q}_m(F)$.
\end{theorem}
\begin{proof}
  This follows from Corollary 3.10 and Proposition 2.18 in~\cite{berchenko2026extension}.
\end{proof}

\subsection{Dimensions of the piecewise constant spaces} \label{sec:dim}

We will now focus on the case $r=0$ and compute the dimension of the space $\mathring{\mathcal{P}}_0\Lambda^{p,q}_m(T^n)$.   Recall that this space consists of trace-free $(p,q)$-forms on $T^n$ that have constant coefficients and belong to the eigenspace of $s^*s$ corresponding to the eigenvalue $m(m+p-q+1)$.  

\begin{lemma} \label{lemma:recursion}
  Assume $p,q,n \ge 0$ and $0 \le m \le q-1$.  Then
  \[
    \dim \mathring{\mathcal{P}}_0\Lambda^{p,q}_m(T^n) + \dim \mathring{\mathcal{P}}_0\Lambda^{p,q}_m(T^{n+1}) = \dim \mathring{\mathcal{H}}_0\Lambda^{p,q}_m(\mathbb{R}^{n+1}).
  \]
\end{lemma}
\begin{proof}
  Taking $r=0$ in Theorem~\ref{thm:extension} tells us that as long as $m \neq q$, every member of $\mathring{\mathcal{P}}_0\Lambda^{p,q}_m(T^n)$ admits an extension to $\mathring{\mathcal{H}}_0\Lambda^{p,q}_m(\mathbb{R}^{n+1})$.  Put another way, the map
  \[
    \tr_{T^n}^{\mathbb{R}^{n+1}} : \mathring{\mathcal{H}}_0\Lambda^{p,q}_m(\mathbb{R}^{n+1}) \to \mathring{\mathcal{P}}_0\Lambda^{p,q}_m(T^n)
  \]
  which takes each member of $\mathring{\mathcal{H}}_0\Lambda^{p,q}_m(\mathbb{R}^{n+1})$ to its trace on $T^n$ is surjective when $m \neq q$.  It follows that
  \[
    \dim \mathring{\mathcal{P}}_0\Lambda^{p,q}_m(T^n) + \dim \ker \tr_{T^n}^{\mathbb{R}^{n+1}} = \dim  \mathring{\mathcal{H}}_0\Lambda^{p,q}_m(\mathbb{R}^{n+1}), \quad m \neq q.
  \]
  The kernel of $\tr_{T^n}^{\mathbb{R}^{n+1}}$ consists of those double forms in $\mathring{\mathcal{H}}_0\Lambda^{p,q}_m(\mathbb{R}^{n+1})$ that have vanishing trace on the coordinate hyperplanes as well as on the hyperplane containing $T^n$.  Equivalently, they have vanishing trace on the boundary of the $(n+1)$-simplex $Q^{n+1} = \{(\lambda_0,\dots,\lambda_n) \mid \lambda_i \ge 0, \, \sum_i \lambda_i \le 1 \}$.
  Since $Q^{n+1}$ is isomorphic to $T^{n+1}$ via an affine transformation, it follows from Proposition~\ref{prop:pullbackdecomposition} that the kernel of $\tr_{T^n}^{\mathbb{R}^{n+1}}$ is isomorphic to $\mathring{\mathcal{P}}_0\Lambda^{p,q}_m(T^{n+1})$.  Thus,
  \[
    \dim \mathring{\mathcal{P}}_0\Lambda^{p,q}_m(T^n) + \dim \mathring{\mathcal{P}}_0\Lambda^{p,q}_m(T^{n+1}) = \dim \mathring{\mathcal{H}}_0\Lambda^{p,q}_m(\mathbb{R}^{n+1}), \quad m \neq q.\qedhere
  \]
\end{proof}

The lemma above provides a recursive formula that we can use to compute the dimension of $\mathring{\mathcal{P}}_0\Lambda^{p,q}_m(T^n)$.  To use it, we will first need to compute the dimensions of $\mathring{\mathcal{H}}_0\Lambda^{p,q}(\mathbb{R}^{n+1})$,  $\mathring{\mathcal{H}}_0\Lambda^{p,q}_m(\mathbb{R}^{n+1})$, and (to handle the base case $n=p$) $\mathring{\mathcal{P}}_0\Lambda^{p,q}_m(T^p)$.
\begin{lemma} \label{lemma:dimH0pq}
  The dimension of $\mathring{\mathcal{H}}_0\Lambda^{p,q}(\mathbb{R}^{n+1})$ is
  \[
    \dim \mathring{\mathcal{H}}_0\Lambda^{p,q}(\mathbb{R}^{n+1}) = \binom{n+1}{p}\binom{p}{n+1-q} = \binom{n+1}{q}\binom{q}{n+1-p}.
  \]
\end{lemma}
\begin{proof}
  Using Notation~\ref{not:basis}, let $e^i =\D\lambda_i$ and let
  \[
    \phi = \sum_{I,J} c_{I,J} e^{I,J}
  \]
  be an arbitrary $(p,q)$-form on $\mathbb{R}^{n+1}$ with constant coefficients.  We assume that the multi-indices $I=(i_1,\dots,i_p)$ and $J=(j_1,\dots,j_q)$ are each in increasing order. If we take the trace on a coordinate hyperplane $\{\lambda_i=0\}$, then every term in this sum has vanishing trace except for the terms with $i \notin I \cup J$.  Those terms with $i \notin I \cup J$ are linearly independent $(p,q)$-forms on $\{\lambda_i=0\}$, so the corresponding coefficients $c_{I,J}$ with $i \notin I \cup J$ must vanish if $\phi$ belongs to $\mathring{\mathcal{H}}_0\Lambda^{p,q}(\mathbb{R}^{n+1})$.  Therefore
  \[
    \phi = \sum_{I,J \colon I \cup J = \{0,1,\dots,n\}} c_{I,J} e^{I,J}
  \]
  if $\phi \in \mathring{\mathcal{H}}_0\Lambda^{p,q}(\mathbb{R}^{n+1})$.  Conversely, every $\phi$ of this form clearly has vanishing trace on the coordinate hyperplanes.  It follows that the set
  \[
    \{ e^{I,J} \mid I \cup J = \{0,1,\dots,n\}, \, |I|=p, \, |J|=q \}
  \]
  forms a basis for $\mathring{\mathcal{H}}_0\Lambda^{p,q}(\mathbb{R}^{n+1})$, where, once again, the multi-indices above are understood to be in increasing order. Each $e^{I,J}$ in this basis can be formed by splitting the set $\{0,\dots,n\}$ into three subsets: $\{0,\dots,n\}\setminus I$, $\{0,\dots,n\}\setminus J$, and $I\cap J$. These subsets have cardinalities $n+1-p$, $n+1-q$, and $p+q-n-1$, respectively, so the dimension of $\oH_0\Lambda^{p,q}(\R^{n+1})$ is given by the corresponding trinomial coefficient, which can be written as
  \begin{equation*}
    \dim \mathring{\mathcal{H}}_0\Lambda^{p,q}(\mathbb{R}^{n+1}) = \binom{n+1}{p}\binom{p}{n+1-q}.\qedhere
  \end{equation*}
\end{proof}

To compute the dimension of the subspace $\mathring{\mathcal{H}}_0\Lambda^{p,q}_m(\mathbb{R}^{n+1}) \subseteq \mathring{\mathcal{H}}_0\Lambda^{p,q}(\mathbb{R}^{n+1})$, we introduce some notation.

\begin{notation}
  Let $\mathring{s}$ and $\mathring{s}^*$ denote the restrictions of $s : \Lambda^{p,q}(\mathbb{R}^{n+1}) \to \Lambda^{p+1,q-1}(\mathbb{R}^{n+1})$ and $s^* : \Lambda^{p,q}(\mathbb{R}^{n+1}) \to \Lambda^{p-1,q+1}(\mathbb{R}^{n+1})$ to $\mathring{\mathcal{H}}_0\Lambda^{p,q}(\mathbb{R}^{n+1})$.  Recall from Proposition~\ref{prop:pullbackdecomposition} that $s$ and $s^*$ commute with pullbacks, so $s\phi$ and $s^*\phi$ have vanishing trace on the coordinate hyperplanes whenever $\phi$ does.  Thus, we have maps
  \[
    \mathring{s} : \mathring{\mathcal{H}}_0\Lambda^{p,q}(\mathbb{R}^{n+1}) \to \mathring{\mathcal{H}}_0\Lambda^{p+1,q-1}(\mathbb{R}^{n+1})
  \]
  and
  \[
    \mathring{s}^* : \mathring{\mathcal{H}}_0\Lambda^{p,q}(\mathbb{R}^{n+1}) \to \mathring{\mathcal{H}}_0\Lambda^{p-1,q+1}(\mathbb{R}^{n+1}).
  \]
\end{notation}

Likewise, on the individual summands, we have maps
\[
  \mathring{s} : \mathring{\mathcal{H}}_0\Lambda^{p,q}_m(\mathbb{R}^{n+1}) \to \mathring{\mathcal{H}}_0\Lambda^{p+1,q-1}_{m-1}(\mathbb{R}^{n+1})
\]
and
\[
  \mathring{s}^* : \mathring{\mathcal{H}}_0\Lambda^{p,q}_m(\mathbb{R}^{n+1}) \to \mathring{\mathcal{H}}_0\Lambda^{p-1,q+1}_{m+1}(\mathbb{R}^{n+1}).
\]

We can see that the $\rs$ and $\rs^*$ operators have the same injectivity and surjectivity properties as the $s$ and $s^*$ operators. One way is to observe that we still have $\rs\rs^*-\rs^*\rs=p-q$ on $\oH_0\Lambda^{p,q}(\R^{n+1})$ and we have $\rs^*\rs=m(m+p-q+1)$ on $\oH_0\Lambda^{p,q}_m(\R^{n+1})$, so the reasoning in Section~\ref{sec:doublemulti} still applies. We can also reason directly, as follows.

\begin{proposition}\label{prop:ringinjsurj}
  Let $l$ be a nonnegative integer. If $s^l\colon\Lambda^{p,q}_m(\R^{n+1})\to\Lambda^{p+l,q-l}_{m-l}(\R^{n+1})$ is injective or surjective, then $\rs^l\colon\oH_0\Lambda^{p,q}_m(\R^{n+1})\to\oH_0\Lambda^{p+l,q-l}_{m-l}(\R^{n+1})$ is injective or surjective, respectively. Likewise, if $\left(s^*\right)^l\colon\Lambda^{p,q}_m(\R^{n+1})\to\Lambda^{p-l,q+l}_{m+l}(\R^{n+1})$ is injective or surjective, then $\left(\rs^*\right)^l\colon\oH_0\Lambda^{p,q}_m(\R^{n+1})\to\oH_0\Lambda^{p-l,q+l}_{m+l}(\R^{n+1})$ is injective or surjective, respectively. These statements also hold when applied to the full spaces $\Lambda^{p,q}$ rather than to the individual summands $\Lambda^{p,q}_m$.
\end{proposition}

\begin{proof}
  If $s^l$ is injective, then its restriction $\rs^l$ to a subspace is also injective. If $s^l$ is surjective, then its adjoint $\left(s^*\right)^l$ is injective, so its restriction $\left(\rs^*\right)^l$ to a subspace is also injective, and so its adjoint $\rs^l$ is surjective. The remaining claims are similar.
\end{proof}

Specifically, we will apply this proposition to obtain the following two statements.

\begin{corollary}\label{cor:ringreducetom0}
  If $m\ge0$ and $m^*=m+p-q\ge0$, then the map $\rs^m\colon\oH_0\Lambda^{p,q}_m(\R^{n+1})\to\oH_0\Lambda^{p+m,q-m}_0(\R^{n+1})$ is an isomorphism. Otherwise, if $m<0$ or $m^*<0$, then $\oH_0\Lambda^{p,q}_m(\R^{n+1})$ is zero.
\end{corollary}

\begin{proof}
  Apply Propositions~\ref{prop:ringinjsurj} and~\ref{prop:reducetom0}.
\end{proof}

\begin{corollary}\label{cor:ringsurj}
  If $p\ge q-1$, then $\rs\colon\oH_0\Lambda^{p,q}(\R^{n+1})\to\oH_0\Lambda^{p+1,q-1}(\R^{n+1})$ is surjective.
\end{corollary}

\begin{proof}
  Apply Propositions~\ref{prop:ringinjsurj} and~\ref{prop:sinjsurj2}.
\end{proof}

\begin{lemma} \label{lemma:dimH0pqm}
  Assume $p,q,n \ge 0$ and $m \ge \max\{0,q-p\}$.  
  The dimension of $\mathring{\mathcal{H}}_0\Lambda^{p,q}_m(\mathbb{R}^{n+1})$ is
  \begin{align*}
    &\dim \mathring{\mathcal{H}}_0\Lambda^{p,q}_m(\mathbb{R}^{n+1}) \\
    &\quad= \binom{n+1}{q-m} \binom{q-m}{n+1-p-m} - \binom{n+1}{q-m-1} \binom{q-m-1}{n-p-m} \\
    &\quad= 
      \begin{dcases}
        \frac{p-q+2m+1}{q-m} \binom{n+1}{q-m-1} \binom{q-m}{n+1-p-m}, &\mbox{ if } m<q,\\
        1, &\mbox{ if } m=q \text{ and } \\ &\quad p+q=n+1, \\
        0, &\mbox{ otherwise. }
      \end{dcases}
  \end{align*}
\end{lemma}
\begin{proof}
  Since $m\ge0$ and $m^*=m+p-q\ge0$, we apply Corollary~\ref{cor:ringreducetom0} to find that $\oH_0\Lambda^{p,q}_m(\R^{n+1})$ is isomorphic to $\oH_0\Lambda^{p+m,q-m}_0(\R^{n+1})$, which is the kernel of $\rs\colon\oH_0\Lambda^{p+m,q-m}(\R^{n+1})\to\oH_0\Lambda^{p+m+1,q-m-1}(\R^{n+1})$. Observe that $p+m\ge q\ge q-m-1$, so this map is surjective by Corollary~\ref{cor:ringsurj}. We conclude that
  \begin{equation*}
    \begin{split}
      \dim\oH_0\Lambda^{p,q}_m(\R^{n+1})&=\dim\oH_0\Lambda^{p+m,q-m}_0(\R^{n+1})\\
      &=\dim\oH_0\Lambda^{p+m,q-m}(\R^{n+1})-\dim\oH_0\Lambda^{p+m+1,q-m-1}(\R^{n+1}).
    \end{split}
  \end{equation*}
  The result then follows from Lemma~\ref{lemma:dimH0pq}.
\end{proof}

\begin{lemma} \label{lemma:basecase}
  Assume $p,q,m \ge 0$.  
  The dimension of $\mathring{\mathcal{P}}_0 \Lambda^{p,q}_m(T^p)$ is
  \begin{equation*} 
    \dim \mathring{\mathcal{P}}_0 \Lambda^{p,q}_m(T^p) = 
    \begin{cases}
      \binom{p}{q} &\mbox{ if } m=0, \\
      0 &\mbox{ if } m>0.
    \end{cases}
  \end{equation*}
\end{lemma}
\begin{proof}
  When $q>p$, the space of $(p,q)$-forms on $T^p$ vanishes, as does the right-hand side of the above equation for all $m$.  When $q \le p$, the only nontrivial $(p,q)$-forms on $T^p$ are of the form $\omega \otimes \alpha$, where $\omega$ is the volume $p$-form and $\alpha$ is an arbitrary $q$-form.  These double forms belong to the kernel of $s$ and have vanishing trace on the boundary of $T^p$; hence they belong to $\mathring{\mathcal{P}}_0 \Lambda^{p,q}_0(T^p)$.  Since the space of constant $q$-forms on $T^p$ has dimension $\binom{p}{q}$, the result follows. 
\end{proof}

\begin{proposition} \label{lemma:dimP0pqm}
  Assume $p,q,n \ge 0$ and $\max\{0,q-p\} \le m \le q-1$.
  The dimension of $\mathring{\mathcal{P}}_0\Lambda^{p,q}_m(T^n)$ is
  \begin{align*}
    \dim \mathring{\mathcal{P}}_0\Lambda^{p,q}_m(T^n) 
    &= \binom{n+1}{q-m}\binom{q-m-1}{p+q-n-1} - \binom{n+1}{p+m+1}\binom{p+m}{p+q-n-1} \\
    &= \frac{p-q+2m+1}{p+m+1} \binom{n+1}{q-m} \binom{q-m-1}{n-p-m}.
  \end{align*}
\end{proposition}
\begin{remark}\label{rem:dimP0pqm}
  The number above has a combinatorial interpretation: It counts the number of standard Young tableaux associated with the partition
  \[
    (n+1-q+m, \, n+1-p-m, \, \underbrace{1, 1, \dots, 1}_{p+q-n-1})
  \]
  of $n+1$. In Appendix~\ref{sec:rep}, we discuss the connection to representation theory explicitly, including how it can be used to compute the dimension of $\oP_0\Lambda^{p,q}_m(T^n)$.
\end{remark}
\begin{proof}
  If $n<p$, then both sides of the equation vanish.  We will prove the equation for $n \ge p$ using induction on $n$. In the base case $n=p$, the right-hand side vanishes when $m > 0$ and equals
  \[
    \frac{p-q+1}{p+1} \binom{p+1}{q} \binom{q-1}{0} = \binom{p}{q}
  \]
  when $m=0$, in agreement with Lemma~\ref{lemma:basecase}.  Now let $n > p$ and assume the formula holds for $n-1$.  Using Lemmas~\ref{lemma:recursion} and~\ref{lemma:dimH0pqm} and denoting $A:=p-q+2m+1$, we compute
  \begin{align*}
    &\dim \mathring{\mathcal{P}}_0\Lambda^{p,q}_m(T^n) \\
    &= \dim \mathring{\mathcal{H}}_0\Lambda^{p,q}_m(\mathbb{R}^n) - \dim \mathring{\mathcal{P}}_0\Lambda^{p,q}_m(T^{n-1}) \\
    &= \frac{A}{q-m} \binom{n}{q-m-1} \binom{q-m}{n-p-m} - \frac{A}{p+m+1} \binom{n}{q-m}\binom{q-m-1}{n-p-m-1} \\
    &= \frac{A}{n+1} \binom{n+1}{q-m} \frac{q-m}{p+q-n} \binom{q-m-1}{n-p-m} \\&\quad - \frac{A}{p+m+1} \frac{n+1-q+m}{n+1} \binom{n+1}{q-m} \frac{n-p-m}{p+q-n} \binom{q-m-1}{n-p-m}.
  \end{align*}
  Since
  \[
    q-m - \frac{(n+1-q+m)(n-p-m)}{p+m+1} = \frac{(n+1)(p+q-n)}{p+m+1},
  \]
  the expression above simplifies to
  \[
    \dim \mathring{\mathcal{P}}_0\Lambda^{p,q}_m(T^n) = \frac{A}{p+m+1} \binom{n+1}{q-m} \binom{q-m-1}{n-p-m}.\qedhere
  \]
\end{proof}

Now that we have determined the dimensions of the trace-free spaces, we know how many degrees of freedom to assign to each subsimplex $K \le T^n$ when constructing our constant finite element space on $T^n$. Namely, we assign $\dim\mathring{\mathcal{P}}_0\Lambda^{p,q}_m(T^l)$ degrees of freedom to each of the $\binom{n+1}{l+1}$ subsimplices $K \le T^n$ of dimension $l$, where the formula for $\dim\mathring{\mathcal{P}}_0\Lambda^{p,q}_m(T^l)$ is given in Proposition~\ref{lemma:dimP0pqm}. We can also verify that the total number of degrees of freedom associated with all of the subsimplices of $T^n$ matches the dimension of $\mathcal{P}_0 \Lambda^{p,q}_m(T^n)$.  We know this must be true from the preceding theory, but it is illuminating to verify it with a direct calculation.  We begin with two lemmas.
\begin{lemma} \label{lemma:dimPpqm}
  Assume $p,q,n \ge 0$ and $m \ge \max\{0,q-p\}$.
  The dimension of $\mathcal{P}_0 \Lambda^{p,q}_m(T^n)$ is
  \begin{align*}
    \dim \mathcal{P}_0 \Lambda^{p,q}_m(T^n) 
    &= \binom{n}{q-m}\binom{n}{p+m} - \binom{n}{q-m-1}\binom{n}{p+m+1} \\
    &= \frac{p-q+2m+1}{p+m+1} \binom{n+1}{q-m} \binom{n}{p+m}.
  \end{align*}
\end{lemma}
\begin{proof}
  We proceed as in the proof of Lemma~\ref{lemma:dimH0pqm}. Since $m\ge0$ and $m^*=m+p-q\ge0$, we apply Proposition~\ref{prop:reducetom0} to see that $\cP_0\Lambda^{p,q}_m(T^n)$ is isomorphic to $\cP_0\Lambda^{p+m,q-m}_0(T^n)$, which is the kernel of $s\colon\cP_0\Lambda^{p+m,q-m}(T^n)\to\cP_0\Lambda^{p+m+1,q-m-1}(T^n)$. As before, this map is surjective by Proposition~\ref{prop:sinjsurj2} because $p+m\ge q-m-1$, so
  \begin{equation*}
    \begin{split}
      \dim \mathcal{P}_0 \Lambda^{p,q}_m(T^n) 
      &= \dim\cP_0\Lambda^{p+m,q-m}_0(T^n)\\
      &= \dim \mathcal{P}_0 \Lambda^{p+m,q-m}(T^n) - \dim \mathcal{P}_0 \Lambda^{p+m+1,q-m-1}(T^n) \\
      &= \binom{n}{p+m}\binom{n}{q-m} - \binom{n}{p+m+1}\binom{n}{q-m-1}.\qedhere      
    \end{split}    
  \end{equation*}
\end{proof}

We can alternatively compute the dimension of $\cP_0\Lambda^{p,q}_m(T^n)$ using representation theory; see Appendix~\ref{sec:rep}.

\begin{lemma} \label{lemma:binomial}
  We have
  \[
    \sum_{l=0}^n \binom{n+1}{l+1} \binom{l+1}{q-m} \binom{q-m-1}{l-p-m} = \binom{n+1}{q-m} \binom{n}{p+m}.
  \]
\end{lemma}
\begin{proof}
  Since 
  \begin{align*}
    \binom{n+1}{l+1} \binom{l+1}{q-m} 
    &= \binom{n+1}{q-m} \binom{n+1-q+m}{l+1-q+m} \\
    &= \binom{n+1}{q-m} \binom{n+1-q+m}{n-l},
  \end{align*}
  it is enough to show that
  \[
    \sum_{l=0}^n \binom{n+1-q+m}{n-l} \binom{q-m-1}{l-p-m} = \binom{n}{p+m}.
  \]
  Equivalently, letting $j=l-p-m$, we must show that
  \[
    \sum_{j=0}^{n-p-m} \binom{n+1-q+m}{n-p-m-j} \binom{q-m-1}{j} = \binom{n}{n-p-m}.
  \]
  This holds because of Vandermonde's identity $\sum_{j=0}^a \binom{b}{a-j} \binom{c}{j} = \binom{b+c}{a}$.
\end{proof}

\begin{proposition}
  Assume $p,q,n \ge 0$ and $\max\{0,q-p\} \le m \le q-1$.
  Then
  \[
    \sum_{l=0}^n \binom{n+1}{l+1} \dim \mathring{\mathcal{P}}_0\Lambda^{p,q}_m(T^l) = \dim \mathcal{P}_0 \Lambda^{p,q}_m(T^n).
  \]
\end{proposition}
\begin{proof}
  We use Lemma~\ref{lemma:binomial} together with the formula for $ \dim \mathring{\mathcal{P}}_0\Lambda^{p,q}_m(T^l)$ given in Proposition~\ref{lemma:dimP0pqm} to compute
  \begin{align*}
    \sum_{l=0}^n &\binom{n+1}{l+1} \dim \mathring{\mathcal{P}}_0\Lambda^{p,q}_m(T^l) \\
                 &= \frac{p-q+2m+1}{p+m+1} \sum_{l=0}^n \binom{n+1}{l+1} \binom{l+1}{q-m} \binom{q-m-1}{l-p-m} \\
                 &= \frac{p-q+2m+1}{p+m+1} \binom{n+1}{q-m}\binom{n}{p+m}.
  \end{align*}
  By Lemma~\ref{lemma:dimPpqm}, this matches the dimension of $\mathcal{P}_0 \Lambda^{p,q}_m(T^n)$.
\end{proof}

\subsection{Examples of piecewise constant finite element spaces.} \label{sec:examples}
We are now ready to discuss the finite element spaces produced by our construction.  We continue to focus on the piecewise constant case $r=0$. As before, we let $N$ denote the dimension of the triangulation, and we use $n$ to refer to the dimension of a given simplex in the triangulation.  
For the reader's convenience, we list the values of $\dim \mathring{\mathcal{P}}_0\Lambda^{p,q}_m(T^n)$ for various values of $p,q,m$, and $n$ in Table~\ref{tab:dim}. We also list a few bases for $\mathcal{P}_0\Lambda^{p,q}_m(\mathcal T)$ in Table~\ref{tab:bases}.
\begin{table}
  \renewcommand{\arraystretch}{1.3}
  \begin{tabular}{r|ccccccc}
    & & & & $n$ & & & \\
    & 0 & 1 & 2 & 3 & 4 & 5 & 6\\
    \hline
    $\Lambda^{1,1}_0$ &  & 1 &  &  &  &  & \\\hline
    $\Lambda^{2,1}_0$ &  &  & 2 &  &  &  & \\ \hline
    $\Lambda^{2,2}_0$ &  &  & 1 & 2 &  &  & \\ 
    $\Lambda^{2,2}_1\cong\Lambda^{3,1}_0$ &  &  &  & 3 &  &  & \\ \hline
    $\Lambda^{3,2}_0$ &  &  &  & 3 & 5 &  & \\ 
    $\Lambda^{3,2}_1\cong\Lambda^{4,1}_0$ &  &  &  &  & 4 &  & \\ \hline
    $\Lambda^{3,3}_0$ &  &  &  & 1 & 5 & 5 & \\ 
    $\Lambda^{3,3}_1\cong\Lambda^{4,2}_0$ &  &  &  &  & 6 & 9 & \\ 
    $\Lambda^{3,3}_2\cong\Lambda^{4,2}_1\cong\Lambda^{5,1}_0$ &  &  &  &  &  & 5 & \\ 
  \end{tabular}
  \caption{Dimension of $\mathring{\mathcal{P}}_0\Lambda^{p,q}_m(T^n)$ for various values of $p,q,m$, and $n$. Zero entries are left blank.}
  \label{tab:dim}
\end{table}

\begin{table}
  \renewcommand{\arraystretch}{1.4}  
  \begin{tabular}{c|c|c}
    Space & $n$ & Basis elements associated with $n$-simplices \\
    \hline
    $\mathcal{P}_0\Lambda^{1,1}_0(\mathcal{T})$ & 1 & $\dlsym{i}{j}$ \\
    \hline
    \multirow{2}{*}{$\mathcal{P}_0\Lambda^{2,1}_0(\mathcal{T})$} & \multirow{2}{*}{2} & $\dlx{ij}{k}-\dlx{jk}{i}$, \\
          &  & $\dlx{ij}{k}-\dlx{ki}{j}$ \\
    \hline
    \multirow{3}{*}{$\mathcal{P}_0\Lambda^{2,2}_0(\mathcal{T})$} & 2 & $\dlsym{ij}{jk}+\dlsym{jk}{ki}+\dlsym{ki}{ij}$ \\
    \cline{2-3}
          & \multirow{2}{*}{3} & $\dlsym{ij}{kl}-\dlsym{ik}{lj}$, \\
          &  & $\dlsym{ij}{kl}-\dlsym{il}{jk}$ \\
    \hline
    \multirow{3}{*}{$\mathcal{P}_0\Lambda^{2,2}_1(\mathcal{T})$} & \multirow{3}{*}{3} & $\dlx{ij}{kl}-\dlx{kl}{ij}$, \\
          &  &  $\dlx{ik}{lj}-\dlx{lj}{ik}$, \\
          &  &  $\dlx{il}{jk}-\dlx{jk}{il}$ \\
    \hline
    \multirow{3}{*}{$\mathcal{P}_0\Lambda^{3,1}_0(\mathcal{T})$} & \multirow{3}{*}{3} & $\dlx{ijk}{l}-\dlx{kjl}{i}$, \\
          &  & $\dlx{ijk}{l}-\dlx{ikl}{j}$, \\
          &  & $\dlx{ijk}{l}-\dlx{jil}{k}$
  \end{tabular}
  \caption{Bases for $\mathcal{P}_0\Lambda^{p,q}_m(\mathcal{T})$ on a triangulation $\mathcal{T}$ of dimension $N$.  Each basis element is associated with a simplex $F$ of dimension $n \le N$.  We label the vertices of $F$ with the letters $\{i,j\}$, $\{i,j,k\}$, and $\{i,j,k,l\}$ for $n=1,2,3$, respectively.  Here, $\lambda_i$ denotes the piecewise linear function on $\cT$ that equals one at vertex $i$ and equals zero at all other vertices, and we use shorthand notation $d\lambda_{ij} = d\lambda_i \wedge d\lambda_j$ and $d\lambda_{ijk} = d\lambda_i \wedge d\lambda_j \wedge d\lambda_k$.  The symbol $\odot$ denotes the symmetrized tensor product $\alpha \odot \beta = \frac{1}{2}(\alpha \otimes \beta + \beta \otimes \alpha)$. One can see that the basis elements above are identically zero on any element that does not contain $F$. Note that an alternate basis for $\cP_0\Lambda^{3,1}_0(\mathcal{T})$ could be obtained by applying $s$ to the basis for $\cP_0\Lambda^{2,2}_1(\mathcal{T})$.}
  \label{tab:bases}
\end{table}

\subsubsection{The case $(p,q)=(1,1)$.}

As discussed in Section~\ref{sec:pq11}, the space $\Lambda^{1,1}$ decomposes into two spaces: a space $\Lambda^{1,1}_0$ consisting of symmetric bilinear forms, and a space $\Lambda^{1,1}_1$ consisting of skew-symmetric bilinear forms, i.e. 2-forms.  In dimension $N \ge 3$, the latter space does not admit a piecewise constant discretization, and correspondingly our construction fails to produce one because $m=q=1$.  The space $\Lambda^{1,1}_0$, on the other hand, admits a piecewise constant discretization.  Referring to the first row of Table~\ref{tab:dim}, the corresponding finite element space has 1 degree of freedom per edge.  The elements of this space have single-valued trace on every codimension-1 simplex $F$, which is equivalent to saying that $\phi(X;Y)$ is single-valued on $F$ for all vectors $X,Y$ that are tangent to $F$.  This space is the lowest order Regge finite element space studied by Christiansen~\cite{christiansen2004characterization,christiansen2011linearization} and Li~\cite{li2018regge}.

\subsubsection{The case $(p,q)=(2,1)$.}

As discussed in Section~\ref{sec:pq21}, the space $\Lambda^{2,1}$ decomposes into two spaces which, in dimension $N=3$, can be identified with matrices.  

The first space, $\Lambda^{2,1}_0$, consists of trace-free matrices under this identification.  Our construction yields a piecewise constant finite element space for such trace-free matrices, and, according to Table~\ref{tab:dim}, this finite element space has 2 degrees of freedom per triangle.  The matrices in this finite element space have normal-tangential continuity along element interfaces, meaning that $\nu^T A \tau_1$ and $\nu^T A \tau_2$ are single-valued along every triangle $F$ with normal vector $\nu$ and tangent basis $(\tau_1,\tau_2)$.  This follows from the identifications between $(2,1)$-forms and matrices discussed in Section~\ref{sec:pq21}.  This finite element space coincides with a space introduced by Gopalakrishnan, Lederer, and Sch\"oberl~\cite{gopalakrishnan2020mass}.

The members of the second space, $\Lambda^{2,1}_1$, can be identified with multiples of the identity matrix in dimension $N=3$.  Normal-tangential continuity is automatic for such matrices, so there is a trivial finite element space for $\Lambda^{2,1}_1$ in 3D that consists of all piecewise constant multiples of the identity.  In dimension $N \ge 4$, $\Lambda^{2,1}_1 \simeq \Lambda^3$ fails to admit a piecewise constant discretization with single-valued trace on element interfaces. Correspondingly, our construction fails to produce one since $m=q=1$.

\subsubsection{The case $(p,q)=(2,2)$.}\label{sec:pq22finite}

As discussed in Section~\ref{sec:pq22}, the space $\Lambda^{2,2}$ decomposes into three spaces.  For each space, we will discuss its discretization first in any dimension $N$ and then (if applicable) specialize to $N=3$.

The first space, $\Lambda^{2,2}_0$, consists of algebraic curvature tensors.  Our piecewise constant finite element space for such tensors, which appears to be new (in dimension $N \ge 4$), has 1 degree of freedom per triangle and 2 degrees of freedom per tetrahedron according to the third row of Table~\ref{tab:dim}.  The tensors in this finite element space have the property that for every element interface $F$, $\phi(X,Y;Z,W)$ is single-valued on $F$ for all vectors $X,Y,Z,W$ that are tangent to $F$.  (The same is true for shared simplices of lower dimension, too.)  In dimension $N=3$, we can identify each member of $\Lambda^{2,2}_0$ with a symmetric $3 \times 3$ matrix $A$, and the aforementioned continuity property reduces to the statement that $\nu^T A \nu$ is single-valued, where $\nu$ is the unit normal to $F$.  This finite element space in dimension $N=3$ coincides with a space introduced by Pechstein and Sch\"oberl~\cite{pechstein2011tangential}.

The second space, $\Lambda^{2,2}_1$, consists of skew-symmetric $(2,2)$-forms.  Its piecewise constant finite element discretization has 3 degrees of freedom per tetrahedron according to Table~\ref{tab:dim}.  In dimension $N=3$, every skew-symmetric $(2,2)$-form can be identified with a skew-symmetric $3 \times 3$ matrix $A$, and the aforementioned continuity property---normal-normal continuity---is vacuous since $\nu^T A \nu$ automatically vanishes.  Thus, this finite element space simply consists of all piecewise constant skew-symmetric $3 \times 3$ matrices.

The third space, $\Lambda^{2,2}_2$ consists of $(2,2)$-forms that alternate in all 4 arguments; i.e. 4-forms.  This space is zero in dimension $N\le3$, admits a discretization isomorphic to the discontinuous Galerkin space of piecewise constant 4-forms if $N=4$, and fails to admit a piecewise constant discretization in dimension $N \ge 5$.  Correspondingly, our construction fails to produce one since $m=q=2$.

\subsubsection{Symmetric and skew-symmetric $(p,p)$-forms.}

As discussed in Section~\ref{sec:symmetric}, symmetric and skew-symmetric $(p,p)$-forms correspond to the even and odd summands, respectively, in the decomposition $\Lambda^{p,p} = \bigoplus_{m=0}^p \Lambda^{p,p}_m$.  Since we can construct piecewise constant finite element spaces for all of the summands except $\Lambda^{p,p}_p$, we can sum those with an even index $m$ to obtain piecewise constant finite element spaces for $\Lambda^{p,p}_{\rm sym}$ if $p$ is odd, and we can sum those with an odd index $m$ to obtain piecewise constant finite element spaces for $\Lambda^{p,p}_{\rm skw}$ if $p$ is even.

\subsection{Related work} \label{sec:related}

For $k \le \ell+p-1$, Hu and Lin~\cite{hu2025finite} work with a space $\mathbb{W}^{k,\ell}_{[p]}$ which is defined as the kernel of $(s^*)^p : \Lambda^{k,\ell} \to \Lambda^{k-p,\ell+p}$.  They abbreviate
\[
  \mathbb{W}^{k,\ell} := \mathbb{W}^{k,\ell}_{[1]}.
\]
In our notation, by Proposition~\ref{prop:dualdecomp},
\begin{align*}
  \mathbb{W}^{k,\ell}_{[p]} 
  &= \bigoplus_{m^*=0}^{p-1} \prescript{}{m^*}\Lambda^{k,\ell} \\
  &= \bigoplus_{m+k-\ell=0}^{p-1} \Lambda^{k,\ell}_{m} \\
  &= \bigoplus_{m=\max\{0,\ell-k\}}^{p-1+\ell-k} \Lambda^{k,\ell}_m.
\end{align*}

For example,
\begin{align*}
  \mathbb{W}^{1,1} &= \Lambda^{1,1}_0, \\
  \mathbb{W}^{2,1}_{[2]} &= \Lambda^{2,1}_0, \\
  \mathbb{W}^{2,2} &= \Lambda^{2,2}_0, \\
  \mathbb{W}^{3,3} &= \Lambda^{3,3}_0, \\
  \mathbb{W}^{3,3}_{[2]} &= \Lambda^{3,3}_0 \oplus \Lambda^{3,3}_1, \\
  \mathbb{W}^{3,3}_{[3]} &= \Lambda^{3,3}_0 \oplus \Lambda^{3,3}_1 \oplus \Lambda^{3,3}_2.
\end{align*}
Among other things, Hu and Lin~\cite{hu2025finite} construct a finite element space for $\mathbb{W}^{k,\ell}_{[p]}$ with the same continuity properties as ours, which they denote by $C_{\iota^*\iota^*}\mathcal{P}^-\mathbb{W}^{k,\ell}_{[p]}$.  This finite element space coincides with our piecewise constant space when $p=1$.  When $p>1$, it generally differs; for example, the spaces $C_{\iota^*\iota^*}\mathcal{P}^-\mathbb{W}^{3,3}_{[2]}$ and $C_{\iota^*\iota^*}\mathcal{P}^-\mathbb{W}^{3,3}_{[3]}$ contain some linear functions~\cite[Lemma 2.5, item (3)]{hu2025finite}.  Hu and Lin also construct finite element spaces with other continuity properties and with general polynomial degrees.

\appendix

\section{Representation theory}\label{sec:rep}
In this section, we will discuss our results in the context of representation theory. Because our results do not depend on this appendix, we will only outline some of the proofs, viewing the role of representation theory more as a source of motivation and intuition. We discuss the decomposition $\Lambda^{p,q}=\bigoplus_m\Lambda^{p,q}_m$ from the perspective of representation theory in Section~\ref{sec:repdecomp}, giving the Young diagram of $\Lambda^{p,q}_m$ in Section~\ref{sec:youngpqm}, from which we obtain $\dim\Lambda^{p,q}_m$ in Section~\ref{sec:dimlpqm}. In Section~\ref{sec:projection}, we discuss how representation theory can yield projection operators $\Lambda^{p,q}\to\Lambda^{p,q}_m$ via symmetrization and antisymmetrization of tensors. In Section~\ref{sec:polyrep}, we move on to double forms with polynomial coefficients, discussing that $\cH_r\Lambda^{p,q}_m$ decomposes into up to four irreducible summands, and we identify the Young diagram for one of the summands. In Section~\ref{sec:constant}, we use this Young diagram to compute the constant in Proposition~\ref{prop:kkdd}. Then, in Section~\ref{sec:vanishing}, we show that, after reindexing, this Young diagram is also the Young diagram for the vanishing trace space $\oP_0\Lambda^{p,q}_m(T^n)$, and we use it to compute $\dim\oP_0\Lambda^{p,q}_m(T^n)$ in Section~\ref{sec:dimensionvanishing}.

\subsection{The decomposition of $\Lambda^{p,q}$ into a sum of irreducible representations}\label{sec:repdecomp}
As in Section~\ref{sec:doublemulti}, let $V$ be an $n$-dimensional vector space, and let $\Lambda^{p,q}=\Lambda^p\otimes\Lambda^q$ be the set of double multicovectors on $V$. The action of the general linear group $GL(V)$ on $V$ induces corresponding actions on the spaces $\Lambda^{p,q}$, so these are $GL(V)$-representations. However, unlike the space of simple multicovectors $\Lambda^k$, the space of double multicovectors $\Lambda^{p,q}$ is generally \emph{not} irreducible as a representation of $GL(V)$. One can decompose $\Lambda^{p,q}$ as a sum of irreducible representations using standard representation theory techniques; see, for instance, \cite[Exercise 6.13*]{fulton2004representation} for the more general case $\Lambda^{p_1}\otimes\dots\otimes\Lambda^{p_r}$, as well as \cite[Exercise 6.16*, Exercise 15.32*]{fulton2004representation} for the special case $\Lambda^p\otimes\Lambda^p$. See also \cite{olver1982hyperforms}. This decomposition is precisely our decomposition $\Lambda^{p,q}=\bigoplus_m\Lambda^{p,q}_m$, which we can see as follows.

The maps $s$ and $s^*$ are equivariant, that is, they commute with the action of $GL(V)$. Consequently, the eigenspaces of $s^*s$ are subrepresentations of $\Lambda^{p,q}$. A priori, they need not be irreducible subrepresentations, but we can check that they are by showing that the number of summands given by Proposition~\ref{prop:lpqmnonzero} is equal to the number of summands given by \cite[Exercise 6.13*]{fulton2004representation}.

Another way of looking at it is that equivariant maps between irreducible representations are either isomorphisms or zero. Per \cite[Exercise 6.13*]{fulton2004representation}, each irreducible subrepresentation of $\Lambda^{p,q}$ appears exactly once in the decomposition. So, by~\cite[Schur's Lemma 1.7]{fulton2004representation}, any equivariant map is a multiple of the identity when restricted to an irreducible subrepresentation. In other words, each irreducible subrepresentation is contained in an eigenspace. Generically, the eigenvalues corresponding to different irreducible representations will be distinct, in which case the irreducible subrepresentations are exactly the eigenspaces. In our case, we can verify that the eigenvalues are distinct in this sense by counting that the number of eigenspaces is equal to the number of irreducible subrepresentations.
\subsection{The Young diagram of $\Lambda^{p,q}_m$}\label{sec:youngpqm}
Irreducible representations of $GL(n)$ are given by Young diagrams. The representation $\Lambda^{p,q}_m$ is given by the Young diagram in Figure~\ref{fig:youngdiagram}. Note that the conditions in Proposition~\ref{prop:lpqmnonzero} imply that $p+m\ge q-m$, as required for Young diagrams.

\begin{figure}
  \begin{tikzpicture}[
    BC/.style = {decorate, 
      decoration={calligraphic brace, amplitude=5pt, raise=1mm},
      very thick, pen colour={black}
    },
    ]
    \matrix (m) [matrix of math nodes,
    nodes={draw, minimum size=6mm, anchor=center},
    column sep=-\pgflinewidth,
    row sep=-\pgflinewidth
    ]
    {
      ~   & ~  \\
      ~   & ~  \\
      |[draw=none,text height=3mm]| \vdots
      & |[draw=none,text height=3mm]| \vdots  \\
      ~   & ~ \\
      ~   &   \\
      ~   &   \\
      |[draw=none,text height=3mm]| \vdots
      &   \\
      ~   &   \\
    };
    \draw[BC] (m-8-1.south west) -- node[left =2.2mm] {$p+m$} (m-1-1.north west);
    \draw[BC] (m-1-2.north east) -- node[right=2.2mm] {$q-m$} (m-4-2.south east);
  \end{tikzpicture}
  \caption{The Young diagram for $\Lambda^{p,q}_m$}
  \label{fig:youngdiagram}
\end{figure}
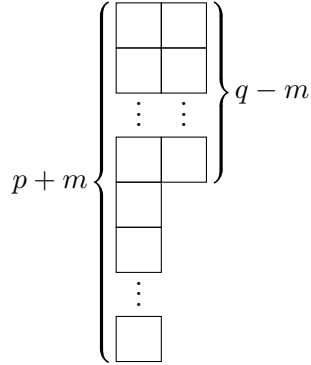

One way to see that the diagram in Figure~\ref{fig:youngdiagram} is indeed the diagram corresponding to $\Lambda^{p,q}_m$ is as follows. First, by Proposition~\ref{prop:reducetom0}, if $\Lambda^{p,q}_m$ is nonzero, then $s^m$ is an isomorphism from $\Lambda^{p,q}_m$ to $\Lambda^{p+m,q-m}_0$. Because $s$ is equivariant, $s^m$ is an isomorphism of representations (not just vector spaces). Thus, it suffices to consider the case where $m=0$.

Considering the space $\Lambda^{p,q}_0$, for concreteness, we assign indices $1,\dots,k=p+q$ to the boxes as in Figure~\ref{fig:youngindices}. The second part of \cite[Exercise 6.14*]{fulton2004representation} tells us how to construct an irreducible representation corresponding to this Young diagram as follows. We start with tensors that are symmetric with respect to the indices in each row, that is, tensors that are symmetric with respect to indices $a$ and $p+a$ for every $1\le a\le q$. Then, we antisymmetrize each column, obtaining tensors in $\Lambda^{p,q}$. Per \cite[Exercise 6.14*]{fulton2004representation}, the resulting space is an irreducible representation corresponding to this Young diagram.

\begin{figure}
  \begin{tikzpicture}[
    BC/.style = {decorate, 
      decoration={calligraphic brace, amplitude=5pt, raise=1mm},
      very thick, pen colour={black}
    },
    ]
    \matrix (m) [matrix of math nodes,
    nodes={draw, minimum height=6mm, minimum width=12mm, anchor=center},
    column sep=-\pgflinewidth,
    row sep=-\pgflinewidth
    ]
    {
      1   & p+1  \\
      2   & p+2  \\
      |[draw=none,text height=3mm]| \vdots
      & |[draw=none,text height=3mm]| \vdots  \\
      q   & p+q \\
      q+1   &   \\
      q+2   &   \\
      |[draw=none,text height=3mm]| \vdots
      &   \\
      p   &   \\
    };
    \draw[BC] (m-8-1.south west) -- node[left =2.2mm] {$p$} (m-1-1.north west);
    \draw[BC] (m-1-2.north east) -- node[right=2.2mm] {$q$} (m-4-2.south east);
  \end{tikzpicture}
  \caption{Tensor indices associated to each box of the Young diagram of $\Lambda^{p,q}_0$}
  \label{fig:youngindices}
\end{figure}
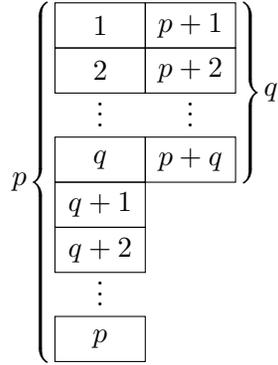

Now, compare with Lemma~\ref{lem:pq0}, which constructs a nonzero element $(\alpha_1\wedge\dots\wedge\alpha_p)\otimes(\alpha_1\wedge\dots\wedge\alpha_q)$ of $\ker s=\Lambda^{p,q}_0$. This tensor is exactly what we would get by applying the map in \cite[Exercise 6.14*]{fulton2004representation} to the tensor $(\alpha_1\otimes\dots\otimes\alpha_p)\otimes(\alpha_1\otimes\dots\otimes\alpha_q)$, which is indeed symmetric with respect to indices $a$ and $p+a$ for each $1\le a\le q$. We know $\Lambda^{p,q}_0$ is an irreducible representation, and we now know that it has nonempty intersection with the irreducible representation constructed by \cite[Exercise 6.14*]{fulton2004representation} for this Young diagram, so they must in fact be equal.

Alternatively, the $s$ operator is considered explicitly in \cite[Exercise~15.30]{fulton2004representation}, where it states that $\ker s$ is an irreducible representation. In the exercise, the representation is expressed using $\Gamma$ notation, but, using the notation correspondence in \cite[p.~223]{fulton2004representation}, we can see that the Young diagram for the representation is the one given in Figure~\ref{fig:youngindices}.

\subsection{The dimension of $\Lambda^{p,q}_m$}\label{sec:dimlpqm}
\begin{figure}
  \begin{tikzpicture}[
    BC/.style = {decorate, 
      decoration={calligraphic brace, amplitude=5pt, raise=1mm},
      very thick, pen colour={black}
    },
    ]
    \matrix (m1) at (-1,0) [matrix of math nodes,
    nodes={draw, minimum height=6mm, minimum width=18mm, anchor=center},
    column sep=-\pgflinewidth,
    row sep=-\pgflinewidth,
    matrix anchor = east,
    ]
    {
      n   & n+1  \\
      n-1   & n  \\
      |[draw=none,text height=3mm]| \vdots
      & |[draw=none,text height=3mm]| \vdots  \\
      n-q+1   & n-q+2 \\
      n-q   &   \\
      n-q-1   &   \\
      |[draw=none,text height=3mm]| \vdots
      &   \\
      n-p+1   &   \\
    };
    \draw[BC] (m1-8-1.south west) -- node[left =2.2mm] {$p$} (m1-1-1.north west);
    \draw[BC] (m1-1-2.north east) -- node[right=2.2mm] {$q$} (m1-4-2.south east);

    \matrix (m2) at (1,0) [matrix of math nodes,
    nodes={draw, minimum height=6mm, minimum width=18mm, anchor=center},
    column sep=-\pgflinewidth,
    row sep=-\pgflinewidth,
    matrix anchor = west,
    ]
    {
      p+1   & q  \\
      p   & q-1  \\
      |[draw=none,text height=3mm]| \vdots
      & |[draw=none,text height=3mm]| \vdots  \\
      p-q+2   & 1 \\
      p-q   &   \\
      p-q-1   &   \\
      |[draw=none,text height=3mm]| \vdots
      &   \\
      1   &   \\
    };
    \draw[BC] (m2-8-1.south west) -- node[left =2.2mm] {$p$} (m2-1-1.north west);
    \draw[BC] (m2-1-2.north east) -- node[right=2.2mm] {$q$} (m2-4-2.south east);
  \end{tikzpicture}
  \caption{The numbers $n-i+j$ (left) and the hook lengths (right) for the Young diagram of $\Lambda^{p,q}_0$}
  \label{fig:youngdimension}
\end{figure}
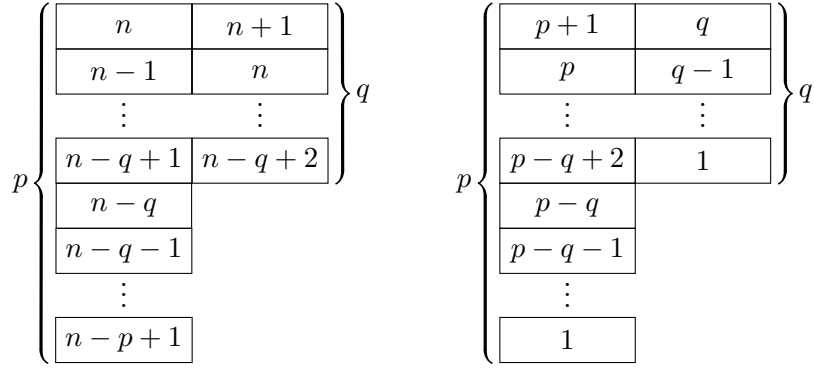

Now that we have the Young diagram corresponding to the representation $\Lambda^{p,q}_m$, its dimension is given by standard representation theory. In \cite{fulton2004representation}, the dimension of this representation is given in Theorem~6.3, but we find it more convenient to apply Exercise~6.4*. As before, without loss of generality, we may consider the case $m=0$. The formula in \cite[Exercise~6.4*]{fulton2004representation} instructs us to compute two numbers associated to each box in the Young diagram: the number $n-i+j$, where $i$ indexes the row and $j$ indexes the column, and the hook length, which is the number of boxes to the right of or below the given box, including the box itself. We fill in these numbers in the Young diagram for $\Lambda^{p,q}_0$ in Figure~\ref{fig:youngdimension}. Then, \cite[Exercise~6.4*]{fulton2004representation} tells us that the dimension of $\Lambda^{p,q}_0$ is obtained by dividing the product of the $n-i+j$ by the product of the hook lengths. The product of the $n-i+j$ is
\begin{equation*}
  \frac{n!}{(n-p)!}\frac{(n+1)!}{(n-q+1)!}.
\end{equation*}
The product of the hook lengths is
\begin{equation}\label{eq:pq0hook}
  \frac{(p+1)!}{p-q+1}q!=\frac{p+1}{p-q+1}p!q!.
\end{equation}
Dividing, we obtain that, whenever the space $\Lambda^{p,q}_0$ is nonzero, its dimension is
\begin{equation*}
  \dim\Lambda^{p,q}_0=\frac{p-q+1}{p+1}\frac{n!}{(n-p)!p!}\frac{(n+1)!}{(n-q+1)!q!}=\frac{p-q+1}{p+1}\binom np\binom{n+1}q.
\end{equation*}
Consequently, whenever the space $\Lambda^{p,q}_m$ is nonzero, its dimension is
\begin{equation*}
  \dim\Lambda^{p,q}_m=\dim\Lambda^{p+m,q-m}_0=\frac{p-q+2m+1}{p+m+1}\binom n{p+m}\binom{n+1}{q-m},
\end{equation*}
which matches the dimension computed in Lemma~\ref{lemma:dimPpqm} for $\cP_0\Lambda^{p,q}_m(T^n)$. (Here, we are implicitly identifying the constant coefficient space $\cP_0\Lambda^{p,q}_m(T^n)$ with $\Lambda^{p,q}_m(V)$, where $V$ is the tangent space to $T^n$ at a point in $T^n$.)

\subsection{Projecting onto $\Lambda^{p,q}_m$ using the Young symmetrizer}\label{sec:projection}
In this section, we discuss the projections onto the summands of the decomposition $\Lambda^{p,q}\to\Lambda^{p,q}_m$. Since $\Lambda^{p,q}_m$ is an eigenspace of $s^*s$, we can construct projection operators using the characteristic polynomial of $s^*s$. However, there is an easier construction using the Young symmetrizer from representation theory.

We recall key facts from \cite[pp.~46 and 76--77]{fulton2004representation}. Given a Young diagram with $k$ boxes, the \emph{Young symmetrizer} $c$ is the operation on $k$-tensors that symmetrizes the indices in each row, and then antisymmetrizes the indices in each column. The image of the Young symmetrizer is an irreducible representation of the general linear group. However, $c$ is not quite a projection; instead $c^2=Nc$, where $N$ is the product of the hook lengths \cite[Hook Length Formula~4.12 and Lemma~4.26]{fulton2004representation}. Note that, here, the convention when symmetrizing or antisymmetrizing is that we add or subtract the permuted terms \emph{without} dividing by a factorial.

We can immediately apply these results to obtain the projection onto the $m=0$ summand of $\Lambda^{p,q}$. Letting $k=p+q$ as before, we symmetrize the indices in each row of the Young diagram in Figure~\ref{fig:youngindices} and then antisymmetrize the indices in each column, obtaining a form in $\Lambda^{p,q}_0$. To obtain a projection, we then divide by the product of the hook lengths, which we have computed to be $\frac{p+1}{p-q+1}p!q!$ in Equation~\eqref{eq:pq0hook}.

We can use this projection $\Lambda^{p,q}\to\Lambda^{p,q}_0$ to construct a projection $\Lambda^{p,q}\to\Lambda^{p,q}_m$ with the following composition.
\begin{equation*}
  \Lambda^{p,q}\xrightarrow{s^m}\Lambda^{p+m,q-m}\xrightarrow{c}\Lambda^{p+m,q-m}_0\xrightarrow{\left(s^*\right)^m}\Lambda^{p,q}_m.
\end{equation*}
Indeed, presuming $\Lambda^{p,q}_m$ is nonzero, the map $s^m$ is an isomorphism from $\Lambda^{p,q}_m$ to $\Lambda^{p+m,q-m}_0$, so the above composition is, up to a constant, the projection from $\Lambda^{p,q}\to\Lambda^{p,q}_m$.

Note also that, up to a constant, the operator $s^m\colon\Lambda^{p,q}\to\Lambda^{p+m,q-m}$ is simply antisymmetrization in the first $p+m$ indices of the tensor. Similarly, $\left(s^*\right)^m\colon\Lambda^{p+m,q-m}\to\Lambda^{p,q}$ is antisymmetrization in the last $q$ indices. So, up to a constant, we can express the projection $\Lambda^{p,q}\to\Lambda^{p,q}_m$ as the following sequence of symmetrization operators on $\Lambda^{p,q}$, interpreted as a $k$-tensor that is antisymmetric in the first $p$ indices and antisymmetric in the last $q$ indices.
\begin{enumerate}
\item Antisymmetrize the first $p+m$ indices, obtaining a form in $\Lambda^{p+m,q-m}$.
\item Symmetrize each pair of corresponding indices in the left and right factors of $\Lambda^{p+m,q-m}$.
\item Antisymmetrize the first $p+m$ indices.
\item Antisymmetrize the last $q$ indices.
\end{enumerate}
Note that, in the third step, we skipped antisymmetrization in the last $q-m$ indices as required by the Young symmetrizer, because doing so would be redundant with the fourth step. Less obviously, we can actually skip the first step; doing so corresponds to interpreting the original $(p,q)$-form as a $k$-tensor and applying the Young symmetrizer for $\Lambda^{p+m,q-m}_0$ to it directly, without first passing to $\Lambda^{p+m,q-m}$.

Getting back to the map $\left(s^*\right)^mcs^m$, which is a projection up to a constant, it remains to compute the constant, which we do by applying the above map to $\Lambda^{p,q}_m$. Grouping the map as
\begin{equation*}
  s^*(s^*(\dotsm s^*(s^*cs)s\dotsm)s)s,
\end{equation*}
we have that the Young symmetrizer $c$ is applied to the space $\Lambda^{p+m,q-m}_0$ yielding the constant $\frac{p+m+1}{p-q+2m+1}(p+m)!(q-m)!$. So, the innermost parentheses are a constant multiple of $s^*s$ applied to $\Lambda^{p+m-1,q-m+1}_1$, and we know that $s^*s$ acts by the constant $(1)(1+p+m-1-(q-m+1)+1)=(1)(p-q+2m)$ on this space. We can continue reasoning this way until we get to the outermost $s^*s$ operator applied to $\Lambda^{p,q}_m$, which acts by $m(p-q+m+1)$. Multiplying, we find that the map $\left(s^*\right)^mcs^m$ is a projection times the constant
\begin{equation*}
  \frac{p+m+1}{p-q+2m+1}(p+m)!(q-m)!m!\frac{(p-q+2m)!}{(p-q+m)!}.
\end{equation*}

Here is one way to interpret the $m!\frac{(p-q+2m)!}{(p-q+m)!}$ constant. Recall that the operator $s$ can be thought of as moving one index from the right antisymmetric factor to the left antisymmetric factor. So, $s^m$ moves $m$ indices from the right antisymmetric factor to the left antisymmetric factor. But the order in which we move the indices does not matter, so a more natural operator would be $\frac{s^m}{m!}\colon\Lambda^{p,q}\to\Lambda^{p+m,q-m}$, and likewise for $\left(s^*\right)^m$. So, dividing by $m!^2$, the map $\frac{\left(s^*\right)^m}{m!}\frac{s^m}{m!}$ acts on $\Lambda^{p,q}_m$ by multiplication by $\binom{p-q+2m}m$, which can be thought of as a way to choose $m$ indices from the $p-q+2m$ singleton rows in Figure~\ref{fig:youngdiagram} to travel back and forth via the $s^m$ and $(s^*)^m$ operators.

\subsection{Polynomial double forms}\label{sec:polyrep}
As noted by Arnold, Falk, and Winther \cite[p.~37]{arnold2006finite}, we can think of elements of $\cH_r$ as symmetric $r$-tensors, and hence $\cH_r\Lambda^k$ can be thought of as a space of $(r+k)$-tensors that are symmetric in the first $r$ indices and antisymmetric in the last $k$ indices. From this perspective, the exterior derivative $\D$ is (up to a constant) simply antisymmetrization in the last $k+1$ indices, yielding a form in $\cH_{r-1}\Lambda^{k+1}$. Conversely, the Koszul operator $\kappa$ is symmetrization in the first $r+1$ indices, yielding a form in $\cH_{r+1}\Lambda^{k-1}$. In other words, $\D$ moves an index from the symmetric part to the antisymmetric part, and $\kappa$ moves an index from the antisymmetric part to the symmetric part.

Likewise, we can think of polynomial double forms in $\cH_r\Lambda^{p,q}$ as $(r+p+q)$-tensors that are symmetric in the first $r$ indices, antisymmetric in the next $p$ indices, and antisymmetric in the next $q$ indices. Up to constant multiples, the symmetrization/antisymmetrization operators that move indices from one group to the other are precisely $\D_L$, $\D_R$, $\kappa_L$, $\kappa_R$, $s$, and $s^*$.

The space $\cH_r\Lambda^{p,q}_m$ is a subrepresentation of $\cH_r\Lambda^{p,q}$, but it is not irreducible. Its decomposition into irreducibles is given by the Pieri formula \cite[Equation~(6.8)]{fulton2004representation}, giving up to four irreducible summands. In fact, our space $\cH_r^-\Lambda^{p,q}_m:=\kappa_L\kappa_R\cH_{r-2}\Lambda^{p+1,q+1}_m$ is one of the irreducible summands. However, for ease of discussion, we will focus our attention on the subspace $d_Ld_R\cH_{r+2}\Lambda^{p-1,q-1}_m$ of $\cH_r\Lambda^{p,q}_m$ instead. As before, we can assume $m=0$ without loss of generality, and we will show that the image of $d_Ld_R$ in $\cH_r\Lambda^{p,q}_0$ is an irreducible representation given by the Young diagram in Figure~\ref{fig:youngimagedd}.

\begin{figure}
  \begin{tikzpicture}[
    BC/.style = {decorate, 
      decoration={calligraphic brace, amplitude=5pt, raise=1mm},
      very thick, pen colour={black}
    },
    ]
    \matrix (m) [matrix of math nodes,
    nodes={draw, minimum size=6mm, anchor=center},
    column sep=-\pgflinewidth,
    row sep=-\pgflinewidth
    ]
    {
      \star   & \star & ~ & ~ & |[draw=none,text height=3mm]| \cdots & ~ \\
      ~   & ~ \\
      |[draw=none,text height=3mm]| \vdots
      & |[draw=none,text height=3mm]| \vdots  \\
      ~   & ~ &&&&|[draw=none]|~\\
      ~   &   \\
      ~   &   \\
      |[draw=none,text height=3mm]| \vdots
      &   \\
      ~   &   \\
    };
    \draw[BC] (m-1-3.north west) -- node[above=2.2mm] {$r$} (m-1-6.north east);
    \draw[BC] (m-8-1.south west) -- node[left =2.2mm] {$p$} (m-1-1.north west);
    \draw[BC] (m-1-6.north east) -- node[right=2.2mm] {$q$} (m-4-6.south east);
  \end{tikzpicture}
  \caption{The Young diagram for the image of $d_Ld_R\colon\cH_{r+2}\Lambda^{p-1,q-1}_0\to\cH_r\Lambda^{p,q}_0$}
  \label{fig:youngimagedd}
\end{figure}
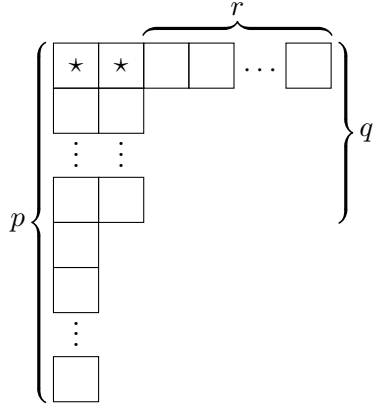

As before, to obtain the irreducible representation, we start with tensors that are symmetric in the indices in each row, and then we antisymmetrize the columns. We can antisymmetrize the columns in two steps, where, in the first step, we ignore the indices marked with a $\star$ and just antisymmetrize the other $p-1$ and $q-1$ indices in the column. We then fully antisymmetrize the columns in the second step. After the first step, the tensor is still symmetric with respect to the indices in the first row, so we can view that part of the tensor as a homogeneous polynomial of degree $r+2$. For the indices in the remaining rows, we did the same construction as in Section~\ref{sec:youngpqm}, so we obtain a tensor in $\Lambda^{p-1,q-1}_0$. So, at the end of the first step, we have the space $\cH_{r+2}\Lambda^{p-1,q-1}_0$. The final antisymmetrization step moves two indices from the symmetric part into each of the antisymmetric parts, which, as discussed above, is simply the operator $d_Ld_R$, as desired.

\subsection{The constant in Proposition~\ref{prop:kkdd}}\label{sec:constant}
Now that we have identified the Young diagram for $d_Ld_R\cH_{r+2}\Lambda^{p-1,q-1}_0$, we can use representation theory to obtain the constant in Proposition~\ref{prop:kkdd} by applying the Young symmetrizer for Figure~\ref{fig:youngimagedd} to $d_Ld_R\cH_{r+2}\Lambda^{p-1,q-1}_0$. As discussed, on the corresponding irreducible representation, the Young symmetrizer acts by multiplication by the product of the hook lengths, which we can compute to be
\begin{equation}\label{eq:youngpqrhook}
  \frac{p+r+1}{p-q+1}\frac{q+r}qp!q!r!.
\end{equation}
using Figure~\ref{fig:youngpqrhook}.

\begin{figure}
  \begin{tikzpicture}[
    BC/.style = {decorate, 
      decoration={calligraphic brace, amplitude=5pt, raise=1mm},
      very thick, pen colour={black}
    },
    ]
    \matrix (m) [matrix of math nodes,
    nodes={draw, minimum height=6mm, minimum width=18mm, anchor=center},
    column sep=-\pgflinewidth,
    row sep=-\pgflinewidth
    ]
    {
      p+r+1   & q+r & r & r-1 & |[draw=none,text height=3mm]| \cdots & 1 \\
      p   & q-1 \\
      |[draw=none,text height=3mm]| \vdots
      & |[draw=none,text height=3mm]| \vdots  \\
      p-q+2   & 1 &&&&|[draw=none]|~\\
      p-q   &   \\
      p-q-1   &   \\
      |[draw=none,text height=3mm]| \vdots
      &   \\
      1   &   \\
    };
    \draw[BC] (m-1-3.north west) -- node[above=2.2mm] {$r$} (m-1-6.north east);
    \draw[BC] (m-8-1.south west) -- node[left =2.2mm] {$p$} (m-1-1.north west);
    \draw[BC] (m-1-6.north east) -- node[right=2.2mm] {$q$} (m-4-6.south east);
  \end{tikzpicture}
  \caption{The hook lengths of the Young diagram in Figure~\ref{fig:youngimagedd}}
  \label{fig:youngpqrhook}
\end{figure}
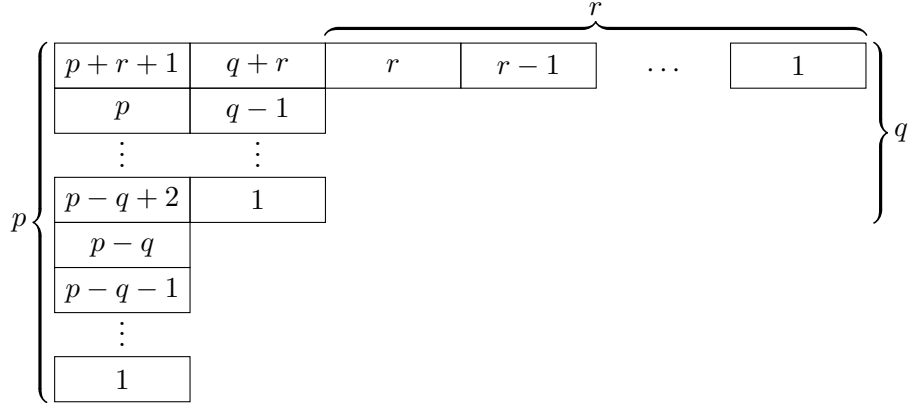

On the other hand, consider what happens when we apply the Young symmetrizer to $d_Ld_R\cH_{r+2}\Lambda^{p-1,q-1}_0 \subseteq \cH_r\Lambda^{p,q}_0$ via the following four steps:
\begin{enumerate}
\item Symmetrize the indices in the first row except for the indices in the boxes marked with a $\star$.
\item Finish symmetrizing the first row by symmetrizing the indices marked with a $\star$ with the rest of the row.
\item Symmetrize the indices in the remaining rows, and then antisymmetrize the columns except for the boxes in the first row marked with a $\star$.
\item Finish antisymmetrizing the columns by antisymmetrizing the indices marked with a $\star$ with the rest of their respective columns.
\end{enumerate}
We see that the first step is simply multiplication by $r!$, since forms in $\cH_r\Lambda^{p,q}$ are already symmetric in those $r$ indices. As discussed, the second step is $\kappa_L\kappa_R$, so we obtain a form in $\cH_{r+2}^-\Lambda^{p-1,q-1}_0$. We recognize the third step as the Young symmetrizer for the Young diagram of $\Lambda^{p-1,q-1}_0$ that is obtained by removing the first row of Figure~\ref{fig:youngimagedd}, so it acts by multiplication by the product of the hook lengths, which is $\frac p{p-q+1}(p-1)!(q-1)!$ by Equation~\eqref{eq:pq0hook}. Finally, the fourth step is $d_Ld_R$.

So, putting everything together, we reach the conclusion that $d_Ld_R\kappa_L\kappa_R$ acts on the space $d_Ld_R\cH_{r+2}\Lambda^{p-1,q-1}_0$ by multiplication by
\begin{equation*}
  \frac{\frac{p+r+1}{p-q+1}\frac{q+r}qp!q!r!}{\frac p{p-q+1}(p-1)!(q-1)!r!}=(p+r+1)(q+r).
\end{equation*}
Recall that, partway through the process, we obtained a form in $\cH_{r+2}^-\Lambda_0^{p-1,q-1}$. Taking on faith that one can deduce from the Pieri formula that $\cH_{r+2}^-\Lambda_0^{p-1,q-1}$ is irreducible, we have that $\kappa_L\kappa_R$ is an isomorphism between $d_Ld_R\cH_{r+2}\Lambda_0^{p-1,q-1}$ and $\cH_{r+2}^-\Lambda_0^{p-1,q-1}$. Consequently, $\kappa_L\kappa_Rd_Ld_R$ acts on $\cH_{r+2}^-\Lambda_0^{p-1,q-1}$ by multiplication by the same constant $(p+r+1)(q+r)$. Reindexing, we have that $\kappa_L\kappa_Rd_Ld_R$ acts on $\cH_r^-\Lambda_0^{p,q}$ by multiplication by the constant $(p+r)(q+r-1)$. Recalling the isomorphism $s^m\colon\Lambda_m^{p,q}\to\Lambda_0^{p+m,q-m}$, we conclude that $\kappa_L\kappa_Rd_Ld_R$ acts on $\cH_r^-\Lambda_m^{p,q}$ by multiplication by $(p+m+r)(q-m+r-1)$, matching Proposition~\ref{prop:kkdd}.

\subsection{The vanishing trace space $\oP_0\Lambda^{p,q}_m(T^n)$}\label{sec:vanishing}
We can also use representation theory to compute the dimension of the vanishing trace space $\oP_0\Lambda^{p,q}_m(T^n)$ given in Proposition~\ref{lemma:dimP0pqm}. However, the situation here is somewhat different, because the vanishing trace space is \emph{not} invariant under the action of the general linear group, since the general linear group does not preserve the boundary of $T^n$. Instead, $\oP_0\Lambda^{p,q}_m(T^n)$ is invariant under the action of the symmetric group $S_{n+1}$ given by permuting the $n+1$ vertices of $T^n$, or, equivalently, by permuting the $n+1$ coordinates $(\lambda_0,\dots,\lambda_n)$ of points in $T^n$. So, $\oP_0\Lambda^{p,q}_m(T^n)$ is an $S_{n+1}$-representation, and, as we will see, it is, in fact, irreducible. Irreducible representations of $S_{n+1}$ are also given by Young diagrams. As suggested by Remark~\ref{rem:dimP0pqm}, we will see that the Young diagram corresponding to $\oP_0\Lambda^{p,q}_m(T^n)$ is the diagram in Figure~\ref{fig:youngvanish}.

\begin{figure}
  \begin{tikzpicture}[
    BC/.style = {decorate, 
      decoration={calligraphic brace, amplitude=5pt, raise=1mm},
      very thick, pen colour={black}
    },
    ]
    \matrix (m) [matrix of math nodes,
    nodes={draw, minimum size=6mm, anchor=center},
    column sep=-\pgflinewidth,
    row sep=-\pgflinewidth
    ]
    {
      ~   & ~ & ~ & ~ & |[draw=none,text height=3mm]| \cdots & ~ \\
      ~   & ~ \\
      |[draw=none,text height=3mm]| \vdots
      & |[draw=none,text height=3mm]| \vdots  \\
      ~   & ~ &&&&|[draw=none]|~\\
      ~   &   \\
      ~   &   \\
      |[draw=none,text height=3mm]| \vdots
      &   \\
      ~   &   \\
    };
    \draw[BC] (m-1-3.north west) -- node[above=2.2mm] {$r':=p+q-(n+1)$} (m-1-6.north east);
    \draw[BC] (m-8-1.south west) -- node[left =2.2mm] {$p':=n+1-(q-m)$} (m-1-1.north west);
    \draw[BC] (m-1-6.north east) -- node[right=2.2mm] {$q':=n+1-(p+m)$} (m-4-6.south east);
  \end{tikzpicture}
  \caption{The Young diagram for $\oP_0\Lambda^{p,q}_m(T^n)$}
  \label{fig:youngvanish}
\end{figure}
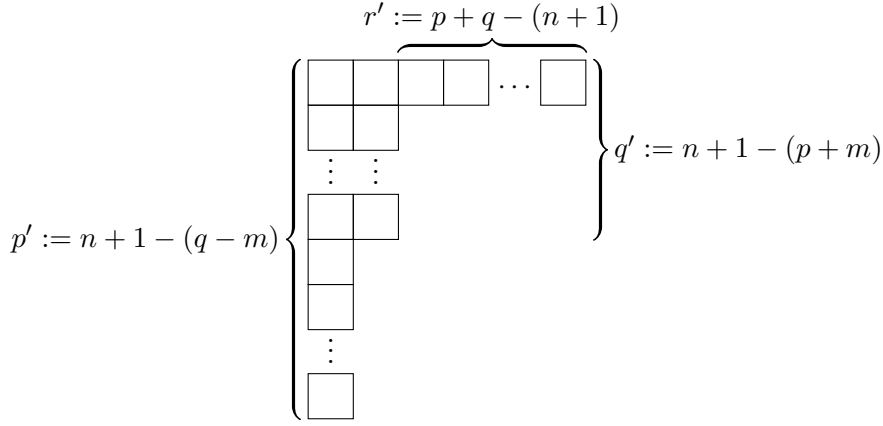

A priori, it is unclear where this diagram is coming from, since, previously, in Figure~\ref{fig:youngdiagram}, we had $p+m$ boxes in the first column and $q-m$ boxes in the second column, but now we have $n+1-(q-m)$ and $n+1-(p+m)$ boxes, respectively, as well as an additional $p+q-(n+1)$ singleton columns. The answer is found in the proof of Theorem~\ref{thm:extension}, where, with $r=0$, we find these numbers from
\begin{equation*}
  d_Ld_R(u_N^{-1}\ostar_{S^n}\Phi^*\phi)\in\cH_{p+q-(n+1)}\Lambda_{m^*}^{n+1-p,\,n+1-q}(\R^{n+1})
\end{equation*}
via the isomorphisms
\begin{equation*}
  \Lambda_{m^*}^{n+1-p,\,n+1-q}\cong\Lambda_m^{n+1-q,\,n+1-p}\cong\Lambda_0^{n+1-(q-m),\,n+1-(p+m)}.
\end{equation*}
In the above, we recall that, in Theorem~\ref{thm:extension}, $\phi$ is an arbitrary extension of $\bar\phi\in\oP_0\Lambda^{p,q}_m(T^n)$, but, once we apply $\ostar_{S^n}$, the form only depends on $\bar\phi$, not on $\phi$. Meanwhile, the isomorphisms are given by the map $\tau$ that swaps the two factors of the double form (Proposition~\ref{prop:taudecomp}) and the map $s^m$ (Proposition~\ref{prop:reducetom0}). Putting everything together, we have an isomorphism given by
\begin{equation*}
  Z\colon\bar\phi\mapsto s^m\tau d_Ld_R(u_N^{-1}\ostar_{S^n}\Phi^*\phi)
\end{equation*}
between $\oP_0\Lambda^{p,q}_m(T^n)$ and a subspace of $\cH_{p+q-(n+1)}\Lambda_0^{n+1-(q-m),\,n+1-(p+m)}(\R^{n+1})$. As we will show, the image of $Z$ is a standard form of the irreducible representation of $S_{n+1}$ given by the Young diagram in Figure~\ref{fig:youngvanish}.

In a way that can be made more precise, this observation highlights the utility of the approach leading to Theorem~\ref{thm:extension}. Pulling back forms on the simplex to forms on the sphere is not just a ``trick'' for constructing extensions, but rather it is the more natural way of expressing these spaces, leading to a better understanding of the spaces, as well as their relationships to representation theory.

To simplify notation, let $p'=n+1-(q-m)$, $q'=n+1-(p+m)$, and $r'=p+q-(n+1)$, so $p'+q'+r'=n+1$. As discussed, the image of $Z$ is a subspace of $\cH_{r'}\Lambda_0^{p',q'}$. More precisely, we will eventually show that it is the subspace of forms that are in the image of $d_Ld_R\colon\cH_{r'+2}\Lambda_0^{p'-1,q'-1}\to\cH_{r'}\Lambda_0^{p',q'}$ and that are odd, in the sense of being anti-invariant under reflections across all coordinate planes. We begin by understanding this space.

First, we see that, because $p'+q'+r'=n+1$, the condition that a form be odd is actually quite stringent. A form being odd means that, when expressed in the standard basis, in every term, every variable $u_i$ appears an odd number of times, totaling both the factors of $u_i$ in the monomial part and the factors of $\du_i$ in the double form part. In particular, each of the $n+1$ variables $u_i$ must appear at least once, but, because $p'+q'+r'=n+1$, they must, in fact, appear exactly once.

Recalling that polynomials can be identified with symmetric tensors, let $A$ be the space of $(n+1)$-tensors spanned by tensors of the form $\du_{i_0}\otimes\dots\otimes \du_{i_n}$ where each $\du_i$ appears exactly once, or, in other words, the $i_0,\dots,i_n$ are a permutation of $0,\dots,n$. We can also think of $A$ as the formal linear combinations of elements of the symmetric group $S_{n+1}$, that is, the group algebra of $S_{n+1}$ \cite[Section~3.4]{fulton2004representation}. Let $c$ be the operation that symmetrizes the indices corresponding to the rows of the Young diagram in Figure~\ref{fig:youngvanish} and then antisymmetrizes the indices corresponding to the columns; this is the Young symmetrizer. When applied to $A$, the Young symmetrizer projects (up to a constant) to an irreducible $S_{n+1}$-representation, and, when applied to the full tensor space $V^{\otimes(n+1)}$, the Young symmetrizer projects (up to the same constant) to the corresponding $GL(n+1)$-representation; see \cite[p.~46 and 76--78]{fulton2004representation}. As discussed in Section~\ref{sec:polyrep}, applying this symmetrizer to the full tensor space yields the image of $d_Ld_R\colon\cH_{r'+2}\Lambda_0^{p'-1,q'-1}\to\cH_{r'}\Lambda_0^{p',q'}$, so applying the symmetrizer to $A$ yields the forms in this space that are odd. (We clarify that, as before, we use the action of $S_{n+1}$ on $A$ that permutes the tensor indices, not the action that permutes the coordinates of $\R^{n+1}$.)

We have shown that the odd forms in the image of $d_Ld_R\colon\cH_{r'+2}\Lambda_0^{p'-1,q'-1}\to\cH_{r'}\Lambda_0^{p',q'}$ are an irreducible $S_{n+1}$-representation corresponding to the Young diagram in Figure~\ref{fig:youngvanish}. Because the space is irreducible, in order to show that it is the image of $Z$ it suffices to show that it contains the image of $Z$.

So, it remains to show that forms in the image of $Z$ are odd and in the image of $d_Ld_R$. The fact that $Z\bar\phi$ is odd follows from the discussion at the end of the proof of Theorem~\ref{thm:extension}, along with the facts that $s$ and $\tau$ commute with all pullbacks (Proposition~\ref{prop:pullbackdecomposition}), so, in particular, they commute with pullbacks via coordinate reflections. Next, $d_Ld_R$ commutes with $\tau$ because it is equal to $d_Rd_L$ by Proposition~\ref{prop:ddcommute}, and it commutes with $s^m$ by Proposition~\ref{prop:ddscommute}, so $Z\bar\phi=d_Ld_R(s^m\tau u_N^{-1}\ostar_{S^n}\Phi^*\phi)$. 

\subsection{The dimension of the vanishing trace space}\label{sec:dimensionvanishing}
We can now compute $\dim\oP_0\Lambda^{p,q}_m(T^n)$ using representation theory, specifically, the \cite[Hook Length Formula 4.12]{fulton2004representation}. (Note that the formula is different from the dimension formula we used earlier because we are now working with a representation of the symmetric group, not the general linear group.) To obtain the dimension, we divide $(n+1)!$ by the product of the hook lengths of the Young diagram in Figure~\ref{fig:youngvanish}. Noting that this diagram is the same as the one in Figure~\ref{fig:youngimagedd} except with $p'$, $q'$, and $r'$ in place of $p$, $q$, and $r$, we have from Equation~\eqref{eq:youngpqrhook} that the product of the hook lengths is
\begin{equation*}
  \frac{p'+r'+1}{p'-q'+1}\frac{q'+r'}{q'}p'!q'!r'!,
\end{equation*}
so the dimension of the representation is
\begin{equation*}
  \dim\oP_0\Lambda^{p,q}_m(T^n)=\frac{(n+1)!}{p'!q'!r'!}\frac{p'-q'+1}{p'+r'+1}\frac{q'}{q'+r'}
\end{equation*}
Note that $\frac{(n+1)!}{p'!q'!r'!}$ is just a trinomial coefficient, which can be given, for instance, by $\binom{n+1}{q'+r'}\binom{q'+r'}{q'}$. We can see that this formula matches the one in Proposition~\ref{lemma:dimP0pqm} because $p-q+2m+1=p'-q'+1$, $p+m+1=p'+r'+1$, $\binom{n+1}{q-m}=\binom{n+1}{q'+r'}$, and $\binom{q-m-1}{n-p-m}=\frac{n+1-p-m}{q-m}\binom{q-m}{n+1-p-m}=\frac{q'}{q'+r'}\binom{q'+r'}{q'}$.

\section*{Acknowledgements}

We would like to thank Anil Hirani for the many stimulating discussions we have had about double forms. Yakov Berchenko-Kogan was supported by NSF grant DMS-2411209.  Evan Gawlik was supported by NSF grant DMS-2533499 and the Simons Foundation award MPS-TSM-00002615.

\printbibliography

\end{document}